\pdfoutput=1
\documentclass[onefignum,onetabnum,reqno,a4paper]{siamart220329}

\usepackage{stmaryrd}
\usepackage{float}
\usepackage{microtype}
\usepackage{upgreek}
\usepackage{mathtools}
\usepackage{paper_diening}
\usepackage{mathrsfs}
\usepackage{enumitem}
\usepackage{hyperref}
\usepackage{booktabs}
\usepackage{colortbl}
\usepackage{tabulary}
\usepackage{diagbox}
\usepackage{caption}

\newsiamremark{assumption}{Assumption}

\usepackage{lipsum}
\usepackage{amsfonts}
\usepackage{graphicx}
\usepackage{epstopdf}
\usepackage{algorithmic}
\ifpdf
  \DeclareGraphicsExtensions{.eps,.pdf,.png,.jpg}
\else
  \DeclareGraphicsExtensions{.eps}
\fi

\newsiamremark{remark}{Remark}
\newsiamremark{hypothesis}{Hypothesis}
\crefname{hypothesis}{Hypothesis}{Hypotheses}
\newsiamthm{claim}{Claim}

\headers{LDG approximation for the \lowercase{$p$}-Navier--Stokes system}{A. Kaltenbach, M. R\r{U}\v{Z}I\v{C}KA}

\title{A Local Discontinuous Galerkin approximation\\ for the \lowercase{$p$}-Navier--Stokes system,\\ Part I: Convergence analysis\thanks{Submitted to the editors \today.
}}

\author{Alex Kaltenbach\thanks{Department of Applied Mathematics, University of Freiburg, Ernst--Zermelo-Straße 1, D-79104 Freiburg, GERMANY. (\email{alex.kaltenbach@mathematik.uni-freiburg.de}).}
	\and Michael R\r{U}\v{Z}I\v{C}KA\thanks{Department of Applied Mathematics, University of Freiburg, Ernst--Zermelo-Straße 1, D-79104 Freiburg, GERMANY.
		(\email{rose@mathematik.uni-freiburg.de}).}}

\usepackage{amsopn}

\usepackage{xcolor}
\definecolor{rltred}{rgb}{0.75,0,0}
\definecolor{rltgreen}{rgb}{0,0.5,0}
\definecolor{rltblue}{rgb}{0,0,0.75}

\ifpdf
\hypersetup{
  pdftitle={A Local Discontinuous Galerkin Approximation for the $p$-Navier--Stokes system},
  pdfauthor={A. Kaltenbach, M. \Ruzicka}}
\fi

\providecommand{\meantmp}[2]{#1\langle{#2}#1\rangle}
\providecommand{\mean}[1]{\meantmp{}{#1}}

\providecommand{\jumptmp}[2]{#1\llbracket{#2}#1\rrbracket}
\providecommand{\jump}[1]{\jumptmp{}{#1}}

\providecommand{\avgtmp}[2]{#1\{{#2}#1\}}
\providecommand{\avg}[1]{\avgtmp{}{#1}}
\providecommand{\bigavg}[1]{\avgtmp{\big}{#1}}

\providecommand{\flux}[1]{{\widehat{#1}}}

\providecommand{\PiDG}{{\Uppi_{h}^{k}}}
\providecommand{\PiDGn}{{\Uppi_{h_n}^{k}}}

\providecommand{\Pia}{{\Uppi_h^{0}}}
\providecommand{\Pian}{{\Uppi_{h_n}^{0}}}

\providecommand{\Vo}{\mathaccent23 V}
\providecommand{\Qo}{\mathaccent23 Q}
\providecommand{\Xhk}{\smash{X_h^k}}
\providecommand{\Vhk}{\smash{V_h^k}}
\providecommand{\Qhk}{\smash{Q_h^k}}
\providecommand{\Qhkc}{\smash{Q_{h,c}^k}}
\providecommand{\Qhko}{{\mathaccent23 Q}_h^k}
\providecommand{\Qhkco}{{\mathaccent23 Q}_{h,c}^k}
\providecommand{\Vhkc}{\smash{V_{h,c}^k}}

\providecommand{\Vhkco}{\mathaccent23 V_{h,c}^k}

\providecommand{\Vhnk}{\smash{V_{h_n}^k}}

\providecommand{\Qhnkco}{{\mathaccent23 Q}_{h_n,c}^k}

\providecommand{\SSS}{\boldsymbol{\mathcal{S}}}

\newcommand{\Ghk}{\boldsymbol{\mathcal{G}}_h^k}
\newcommand{\Ghnk}{\boldsymbol{\mathcal{G}}_{h_n}^k\! }
\newcommand{\Dhk}{\boldsymbol{\mathcal{D}}_h^k}
\newcommand{\Dhnk}{\boldsymbol{\mathcal{D}}_{h_n}^k}
\newcommand{\Divhk}{\mathcal{D}\dot{\iota}\nu_h^k}
\newcommand{\Divhnk}{\mathcal{D}\dot{\iota}\nu_{h_n}^k}
\newcommand{\Rhk}{\boldsymbol{\mathcal{R}}_h^k}
\newcommand{\Rhnk}{\boldsymbol{\mathcal{R}}_{h_n}^{\smash{k,{\textup{sym}}}}}
\newcommand{\Rhks}{\boldsymbol{\mathcal{R}}_h^{\smash{k,{\textup{sym}}}}}

\newcommand{\WDG}{W^{1,p}(\mathcal{T}_h)}

\newcommand{\WDGn}{W^{1,p}(\mathcal{T}_{h_n})}

\providecommand{\divo}{\textrm{div}\,}

\providecommand{\sss}{\avg{\abs{\Pia \bfL_h^{\textup{sym}}}}}
\providecommand{\sssn}{\avg{\abs{\Pian \!\bfL_{h_n}^{\textup{sym}}}}}
\providecommand{\sssl}{\avg{\abs{\Pia (\Dhk \bfv_h +\Rhks\bfv^*)}}}
\providecommand{\ssslo}{\avg{\abs{\Pia (\Dhk \bfu_h + [\bfH_h^k\bfv^*]^{\textup{sym}})}}}

\begin{document}

\maketitle

\begin{abstract}
   In the present paper, we propose a Local Discontinuous Galerkin
  (LDG) approximation for fully non-homogeneous systems of $p$-Navier--Stokes type.
  On the basis~of~the~primal~formulation, we
  prove well-posedness, stability (\textit{a priori}
  estimates), and weak convergence of the method. To this end, we propose a new DG discretization of the
  convective term and develop an abstract~non-conforming theory of
  pseudo-monotonicity, which is applied to our problem. We also use our approach to treat the $p$-Stokes problem. 
\end{abstract}



\section{Introduction}

We consider the numerical approximation of steady systems~of $p$-Navier--Stokes type, i.e., 
\begin{equation}
  \label{eq:p-navier-stokes}
  \begin{aligned}
    -\divo\SSS(\bfD\bfv)+[\nabla\bfv]\bfv+\nabla q&=\bfg -\divo \bfG \qquad&&\text{in }\Omega\,,\\
    \divo\bfv&=g \qquad&&\text{in }\Omega\,,
    \\
    \bfv &= \bfv_0 &&\text{on } \partial\Omega\,,
  \end{aligned}
\end{equation}
using a Local Discontinuous Galerkin (LDG) scheme. The physical
problem motivating this study is the steady motion of a
homogeneous, incompressible~fluid~with~shear-dependent viscosity. More
precisely, for a given vector field $\bfg\hspace{-0.15em}\colon\hspace{-0.15em}\Omega\hspace{-0.15em}\to\hspace{-0.15em} \mathbb{R}^d$~and~a~given~tensor field $\bfG\hspace{-0.1em}\colon\hspace{-0.1em}\Omega\hspace{-0.1em}\to\hspace{-0.1em} \mathbb{R}^{d\times d}$, jointly
describing external body forces, as well as a given divergence
$g\hspace{-0.1em}\colon\hspace{-0.1em}\Omega\hspace{-0.1em}\to\hspace{-0.1em} \mathbb{R}$, and a given Dirichlet boundary condition
$\bfv_0\colon\partial\Omega\to \mathbb{R}^d$, we~seek~for~a~velocity vector
field ${\bfv\hspace{-0.1em}=\hspace{-0.1em}(v_1,\dots,v_d)^\top\hspace{-0.1em}\colon\hspace{-0.1em}\Omega\hspace{-0.1em}\to\hspace{-0.1em}
  \mathbb{R}^d}$~and~a~kinematic~\mbox{pressure}~${q\hspace{-0.1em}\colon\hspace{-0.1em}\Omega\hspace{-0.1em}\to\hspace{-0.1em} \mathbb{R}}$~\mbox{solving}~\eqref{eq:p-navier-stokes}.  Here,
$\Omega\hspace{-0.1em}\subseteq\hspace{-0.1em} \mathbb{R}^d$, $d\hspace{-0.1em}\in\hspace{-0.1em} \{2,3\}$, is a bounded polyhedral
domain having~a~Lipschitz~\mbox{continuous} boundary $\partial\Omega$. The
extra stress tensor $\SSS(\bfD\bfv)\hspace{-0.1em}\colon\hspace{-0.1em}\Omega\hspace{-0.1em}\to\hspace{-0.1em} \mathbb{R}^{d\times d}$
depends on the~strain~rate~tensor
$\smash{\bfD\bfv\hspace{-0.1em}\coloneqq \hspace{-0.1em}\frac{1}{2}(\nabla
  \bfv+\nabla\bfv^\top)\hspace{-0.1em}\colon\hspace{-0.1em}\Omega\hspace{-0.1em}\to\hspace{-0.1em} \mathbb{R}^{d\times d}_{\textup{sym}}}$,
i.e., the symmetric part~of~the~\mbox{velocity}~\mbox{gradient} tensor $\bfL\hspace{-0.1em}\coloneqq \hspace{-0.1em}\nabla \bfv=(\partial_j v_i)_{i,j=1,\dots,d}\colon\Omega\to \mathbb{R}^{d\times
  d}$. The convective term $\smash{[\nabla\bfv]\bfv\colon\Omega\to
  \mathbb{R}^d}$~is defined via
$([\nabla\bfv]\bfv)_i\hspace{-0.1em}\coloneqq \hspace{-0.1em}\smash{\sum_{j=1}^d{v_j\partial_j
    v_i}}$ for all $i\hspace{-0.1em}=\hspace{-0.1em}1,\dots,d$.  The prescribed~Dirichlet~boundary
condition ${\bfv_0\colon\partial\Omega\to \mathbb{R}^d}$ and the prescribed
divergence $g\colon\Omega\to
\mathbb{R}$~have~to~satisfy
\begin{align}
    \int_{\Omega}{g\,\textup{d}x}=\int_{\partial\Omega}{\bfv_0\cdot\bfn\,\textup{d}s}\,,\label{eq:compatibility}
\end{align}
where $\bfn\colon\partial\Omega\to \smash{\mathbb{S}^{d-1}}$ denotes the unit normal vector field to $\Omega$ pointing outwards.
Physical interpretation and discussion of some non-Newtonian fluid models~can~be~found,~e.g., in \cite{bird,mrr,ma-ra-model}.

Throughout the paper, we  assume that the extra stress tensor~$\SSS$~has~\mbox{$(p,\delta)$-structure} (cf.~Assumption~\ref{assum:extra_stress}). The relevant example falling into this class is 
\begin{align*}
    \SSS(\bfD\bfv)=\mu\, (\delta+\vert \bfD\bfv\vert)^{p-2}\bfD\bfv\,,
\end{align*}
where $p\in (1,\infty)$, $\delta\ge 0$, and $\mu>0$.\newpage

Mathematical investigation of fluids with shear-dependent viscosities
started with the celebrated work of O. Ladyzhenskaya
(cf.~\cite{lady-bo}). In recent years, there has been enormous
progress in the understanding of this problem, and we refer the~reader~to \cite{mnrr,ma-ra-model,hugo-boundary,hugo-petr-rose} and the references
therein for a detailed discussion.

\enlargethispage{4mm}%

Introducing the unknowns $\bfL,\bfS ,\bfK\colon\Omega\to
\mathbb{R}^{d\times d}$, the system \eqref{eq:p-navier-stokes} can be
re-written as a ``first order'' system, i.e., \\[-3mm]
\begin{equation}
  \label{eq:p-navier-stokes-DG}
  \begin{aligned}
    \bfL=\nabla\bfv,\quad \bfS  = \SSS(\bfL^{\textup{sym}}),\quad \bfK&=\bfv\otimes\bfv\qquad&&\text{in }\Omega\,,\\-\divo \bfS +\smash{\tfrac{1}{2}}\divo\bfK+\smash{\tfrac{1}{2}}(\bfL-g\mathbf{I}_d)\bfv +\nabla q&=\bfg -\divo \bfG \qquad&&\text{in }\Omega\,,\\
   \divo\bfv&=g \qquad&&\text{in }\Omega\,,
    \\
    \bfv &= \bfv_0 &&\text{on } \partial\Omega\,,\\[-1.5mm]
  \end{aligned}
\end{equation}
where we used in \eqref{eq:p-navier-stokes-DG}$_2$ that
$[\nabla \bfv]\bfv=\frac{1}{2}\divo(\bfv\otimes
\bfv)+\frac{1}{2}(\nabla \bfv-\divo\bfv\,\mathbf{I}_d)\bfv$ and
\eqref{eq:p-navier-stokes}$_{1,2}$. {Note that this is a
  modification of the standard Temam modification of the convective
  term taking into account that the prescribed divergence is
  non-trivial.}  Discontinuous Galerkin (DG) methods for elliptic
problems have been introduced in the late 90's. They are by now
well-understood and rigorously analyzed in the context of linear
elliptic problems (cf.~\cite{arnold-brezzi} for the Poisson
problem).~In~contrast~to~this, only few papers treat $p$-Laplace type
problems using DG methods
(cf.~\cite{ern-p-laplace,BufOrt09,dkrt-ldg,CS16,sip,QS19}), or
non-conforming methods (cf.~\cite{Bar21}). There are several
investigations of $p$-Stokes problems using FE methods
(cf.~\cite{San1993,BL1994b,BL1993a,bdr-phi-stokes,Hi13a}) 
and using DG
methods (cf.~\cite{bustinza-fluid,CHSW13,GS15,BCPH20}). Several papers
treat the Navier--Stokes problem ($p=2$) by means of DG methods. We
refer to \cite{CKS05,CKS09,NPC11,ern,ern-book,CCQ17} to name a few.
In addition,~the~convergence~analysis for the $p$-Navier--Stokes
problem using FE methods can be found in
\cite{die-sueli-2013,KS16,KPS18,KS19,FGS22}.  On the other hand, to
the best of the authors knowledge, there are no investigations using
DG methods for the $p$-Navier--Stokes problem
\eqref{eq:p-navier-stokes} {and it is desirable to fill this
  lacuna. This motivated us to develop an abstract framework for the
  convergence analysis of DG formulations
  of~\eqref{eq:p-navier-stokes} and related problems. This framework
  extends the quasi non-conforming setting from
  \cite{alex-rose-nonconform} to the fully non-conforming setting
  (cf.~\cite{kr-unsteady-dg} for the evolutionary case). We exemplify
  the potential of this abstract framework on an equal order LDG
  scheme of~\eqref{eq:p-navier-stokes} with discontinuous velocity and
  continuous pressure for low regularity data from $(W^{1,p}_0(\Omega))^*$
  and prove the first
  convergence for a DG~formulation of~\eqref{eq:p-navier-stokes}. On
  the other hand, the framework is general enough to be used for other
  DG schemes, which exploit the advantages of DG methods as
  e.g.~flexibility of the local polynomial degree, flexibility of the
  meshes, possibility of low order methods, enforcing the divergence
  constraint exactly with low order elements, or using discontinuous
  pressure. Moreover, using hybridisation DG methods become
  competitive in terms of computational costs. None of these aspects is
  discussed in this paper, but is left for future work.}%
In Part I of the paper, we develop an abstract theory
for steady non-conforming pseudo-monotone problems and propose a new DG
discretization of the convective term. Moreover, we present easily
verifiable sufficient conditions, which ensure the
applicability of the abstract theory. Finally, we apply this theory to the
$p$-Navier--Stokes and~the $p$-Stokes problem. For the former, we
restrict ourselves to the case $p\in (2,\infty)$,
since~we~deal~with~fully non-homogeneous boundary conditions, while
for the latter one, all $p\in (1,\infty)$~are~covered. {In~Part~II of
the paper, we prove convergence rates for the velocity of the
homogeneous $p$-Navier--Stokes~problem and the homogeneous
$p$-Stokes~problem. In particular, we prove optimal linear convergence
rates for the velocity of an LDG scheme with only linear ansatz functions, which is
not possible with conforming inf-sup stable FE methods in general
shape-regular meshes. Part III of the paper provides convergence
rates for the pressure of the same problems. In particular, we prove linear convergence
rates for the pressure~of~the~same~LDG~scheme.}

\textit{This paper is organized as follows:} \!In Section
\ref{sec:preliminaries}, we introduce the employed~\mbox{notation}, define
 relevant function spaces, 
 basic assumptions on the extra stress~tensor~$\SSS$~and~its consequences,
 weak formulations in Problem (Q) and Problem (P) of the~system~\eqref{eq:p-navier-stokes}, and discrete operators.
 In Section \ref{sec:ldg}, we define our numerical
fluxes~and~derive~the flux \hspace{-0.1mm}and \hspace{-0.1mm}the \hspace{-0.1mm}primal
\hspace{-0.1mm}formulation, \hspace{-0.1mm}i.e, \hspace{-0.1mm}Problem~\hspace{-0.1mm}(Q$_h$) \hspace{-0.1mm}and \hspace{-0.1mm}Problem~\hspace{-0.1mm}(P$_h$),~\hspace{-0.1mm}of~\hspace{-0.1mm}the~\hspace{-0.1mm}\mbox{system}~\hspace{-0.1mm}\eqref{eq:p-navier-stokes}.  In Section \ref{sec:non-conform}, we
introduce a general concept of non-conforming approximations and
non-conforming pseudo-monotonicity.  In Section \ref{sec:application},
we prove the existence~of~DG~solutions
(cf.~Proposition~\ref{prop:well_posed}), the stability of the method,
i.e., a priori estimates (cf.~Proposition~\ref{prop:stab}), and the
convergence~of~DG solutions (cf.~Theorem~\ref{thm:convergence}, Theorem~\ref{thm:convergence.1}). Theorem~\ref{thm:convergence} 
is the first (weak) convergence result for an LDG method for steady systems of \mbox{$p$-Navier--Stokes type.} In Section~\ref{sec:experiments},
we confirm our theoretical findings via numerical experiments.

\section{Preliminaries}\label{sec:preliminaries}

\subsection{Function spaces}

We employ $c, C>0$ to denote generic constants,~that~may change from line
to line, but are not depending on the crucial quantities.~Moreover,~we write $f\sim g$ if and only if there exists constants $c,C>0$ such
that $c\, f \le g\le C\, f$.

For $k\in \setN$ and $p\in [1,\infty]$, we employ the customary
Lebesgue spaces $(L^p(\Omega), \smash{\norm{\cdot}_p}) $ and Sobolev
spaces $(W^{k,p}(\Omega), \smash{\norm{\cdot}_{k,p}})$, where $\Omega
\subseteq \setR^d$, $d\in \{2,3\}$, is a bounded,~polyhedral Lipschitz domain. The space $\smash{W^{1,p}_0(\Omega)}$
is defined as those functions from $W^{1,p}(\Omega)$ whose trace vanishes on $\partial\Omega$. We equip
$\smash{W^{1,p}_0(\Omega)}$ 
with the norm $\smash{\norm{\nabla\,\cdot\,}_p}$. 

We do not distinguish between function spaces for scalar,
vector-~or~\mbox{tensor-valued} functions. However, we denote
vector-valued functions by boldface letters~and~tensor-valued
functions by capital boldface letters. The standard scalar product
between~two vectors is denoted by $\bfa \cdot\bfb$, while the
  Frobenius scalar product between~two~tensors is denoted by
  $\bfA: \bfB$.  The mean value of a locally integrable function $f$
  over a measurable set $M\subseteq \Omega$ is denoted by
  ${\mean{f}_M\coloneqq \smash{\dashint_M f
      \,\textup{d}x}\coloneqq \smash{\frac 1 {|M|}\int_M f
      \,\textup{d}x}}$. Moreover, we employ the notation
  $\hskp{f}{g}\coloneqq \int_\Omega f g\,\textup{d}x$, whenever the
  right-hand~side~is~\mbox{well-defined}.

  We will also use Orlicz and Sobolev--Orlicz spaces
  (cf.~\cite{ren-rao}).~A~real~convex~\mbox{function}
  $\psi \colon \mathbb{R}^{\geq 0} \to \mathbb{R}^{\geq 0}$ is said to be
  an {N-function}, if it holds $\psi(0)=0$,
  $\psi(t)>0$~for~all~${t>0}$, $\lim_{t\rightarrow0} \psi(t)/t=0$, and
  $\lim_{t\rightarrow\infty} \psi(t)/t=\infty$. 
  We~define the (convex) \mbox{conjugate N-func-} tion $\psi^*$ by
  ${\psi^*(t)\hspace*{-0.1em}\coloneqq \hspace*{-0.1em} \sup_{s \geq
      0} (st \hspace*{-0.1em}- \hspace*{-0.1em}\psi(s))}$~for all
  ${t \!\geq \!0}$, which satisfies
  $(\psi^*)'\! =\!  (\psi')^{-1}\!\!$. A given N-function
  $\psi$~satisfies~the~\mbox{$\Delta_2$-condition} (in short,
  $\psi \hspace*{-0.1em}\in\hspace*{-0.1em}\Delta_2$), if there exists
  $K\hspace*{-0.1em}>\hspace*{-0.1em} 2$ such that ${\psi(2\,t) \hspace*{-0.1em}\leq\hspace*{-0.1em} K\,
    \psi(t)}$ for all
  $t \hspace*{-0.1em}\ge\hspace*{-0.1em}
  0$. We denote the smallest such constant by
  $\Delta_2(\psi)\hspace*{-0.1em}>\hspace*{-0.1em}0$.  We need the
  following version of the $\varepsilon$-Young inequality: for every
  ${\varepsilon> 0}$,~there~exists a constant $c_\epsilon>0 $,
  depending  on $\Delta_2(\psi),\Delta_2( \psi ^*)<\infty$, such
  that~for~every~${s,t\geq 0}$,~it~holds
\begin{align}
  \label{ineq:young}
    t\,s&\leq \epsilon \, \psi(t)+ c_\epsilon \,\psi^*(s)\,.
\end{align}

\subsection{Basic properties of the extra stress tensor}

Throughout the paper, we will assume that the extra stress tensor $\SSS$ has $(p,\delta)$-structure, which will be~defined~now. A detailed discussion and full proofs can be found in \cite{die-ett,dr-nafsa}. For a tensor $\bfA\in \mathbb{R}^{d\times d}$, we denote its symmetric part by ${\bfA^{\textup{sym}}\coloneqq \frac{1}{2}(\bfA+\bfA^\top)\in \mathbb{R}^{d\times d}_{\textup{sym}}\coloneqq \{\bfA\in \mathbb{R}^{d\times d}\mid \bfA=\bfA^\top\}}$.

For $p \in (1,\infty)$ and~$\delta\ge 0$, we define a special N-function
$\phi=\phi_{p,\delta}\colon\smash{\mathbb{R}^{\ge 0}\to \mathbb{R}^{\ge 0}}$~by
\begin{align} 
  \label{eq:5} 
  \varphi(t)\coloneqq  \int _0^t \varphi'(s)\, \mathrm ds\,,\quad\text{where}\quad
  \varphi'(t) \coloneqq  (\delta +t)^{p-2} t\quad\textup{ for all }t\ge 0\,.
\end{align}
It is well-known that $\phi$ is balanced
(cf.~\cite{dr-nafsa,br-multiple-approx}), since $ {\min\set{1,p-1}\,( \delta+t)^{p-2} \le \varphi''(t)}$ $\leq
\max\set{1,p-1}( \delta+t)^{p-2}$ for~all~${t,\delta\ge 0}$. Moreover, $\phi$ satisfies, independent~of~${\delta\ge  0}$, the
$\Delta_2$-condition with $\Delta_2(\phi) \leq c\, 2^{\max \set{2,p}}$. 
Apart from that, the conjugate function $\phi^*$ satisfies, independent of ${\delta\ge  0}$, the
$\Delta_2$-condition with $\Delta_2(\phi^*) \leq c\,2^{\max \set{2,p'}}$~as~well~as ${\phi^*(t) \sim
(\delta^{p-1} + t)^{p'-2} t^2}$ uniformly in $t\ge 0$, where $\smash{1= \frac{1}{p} +
\frac{1}{p'}}$.

An important tool in our analysis are shifted N-functions
$\set{\psi_a}_{\smash{a \ge 0}}$  (cf.~\cite{DK08,dr-nafsa}). For a given N-function $\psi\colon\mathbb{R}^{\ge 0}\to \mathbb{R}^{\ge
  0}$, we define the family  of shifted N-functions ${\psi_a\colon\mathbb{R}^{\ge
    0}\to \mathbb{R}^{\ge 0}}$,~${a \ge 0}$, for every $a\ge 0$, via
\begin{align}
  \label{eq:phi_shifted}
  \psi_a(t)\coloneqq  \int _0^t \psi_a'(s)\, \mathrm ds\,,\quad\text{where }\quad
  \psi'_a(t)\coloneqq \psi'(a+t)\frac {t}{a+t}\,,\quad\textup{ for all }t\ge 0\,.
\end{align}
Shifted N-functions have the following important property: if
$\psi,\psi^* \in \Delta_2$, then for every $\varepsilon \in (0,1)$, there exists a
constant $c_\varepsilon=c_\varepsilon( \Delta_2(\psi))>0$ such that for~all~${a,b,t \ge 0}$,~it~holds
\begin{align}
  \label{eq:shift}
  \psi_a(t) \le c_\varepsilon\, \psi_b(t) +\varepsilon\, \psi_b(\abs{b-a})\,.
\end{align}

\begin{remark} \label{rem:phi_a}
  {\rm 
    For the above defined N-function $\varphi $ we~have, uniformly in $a,t\!\ge\! 0$, that
    $\phi_a(t) \sim (\delta+a+t)^{p-2} t^2$~and $(\phi_a)^*(t)
    \sim ((\delta+a)^{p-1} + t)^{p'-2} t^2$.  The families
    $\set{\phi_a}_{\smash{a \ge 0}}$~and $\set{(\phi_a)^*}_{\smash{a \ge 0}}$ satisfy, uniformly in $a \ge 0$,
    the $\Delta_2$-condition~with~$\Delta_2(\phi_a) \leq c\, 2^{\smash{\max \set{2,p}}}$ and
    $\Delta_2((\phi_a)^*) \leq c\, 2^{\smash{\max \set{2,p'}}}$,
    respectively. Moreover, note that $(\phi_{p,\delta})_a(t)=\phi_{p,
      \delta+a}(t)$ for all $t,a,\delta\ge 0$, and that the function
    $(a \mapsto \phi_a(t))\colon\mathbb{R}^{\ge 0}\to\mathbb{R}^{\ge 0} $ for every $t\ge 0$~is~non-decreasing for $p\ge 2$ and non-increasing for $p\le 2$.
    }
  \end{remark}

%
\begin{assumption}[Extra stress tensor]\label{assum:extra_stress}
  We assume that the extra stress tensor \linebreak$ \SSS\hspace{-0.15em}\colon\hspace{-0.15em} \mathbb{R}^{d \times d}
  \hspace{-0.15em}\to\hspace{-0.15em} \mathbb{R}^{d \times d}_{\textup{sym}} $ belongs to $C^0(\mathbb{R}^{d \times
    d}\hspace{-0.1em},\mathbb{R}^{d \times d}_{\textup{sym}} ) $, satisfies $\SSS (\hspace{-0.05em}\bfA\hspace{-0.05em}) \hspace{-0.15em}=\hspace{-0.15em} \SSS 
  (\hspace{-0.05em}\bfA^{\textup{sym}}\hspace{-0.05em} )$~for~all~${\bfA\hspace{-0.15em}\in\hspace{-0.15em} \mathbb{R}^{d \times d}}\hspace{-0.1em}$, and $\SSS (\mathbf 0)=\mathbf 0$. Furthermore, we
  assume that the tensor $\SSS$ has {\rm $(p, \delta)$-structure}, i.e.,
  for some $p \in (1, \infty)$, $ \delta\in [0,\infty)$, and the
  N-function $\varphi=\varphi_{p,\delta}$ (cf.~\eqref{eq:5}), there
  exist constants $C_0, C_1 >0$ such that
   \begin{equation}
     \label{eq:ass_S}
     \begin{aligned}
       \big({\SSS}(\bfA) - {\SSS}(\bfB)\big) : \big(\bfA-\bfB
       \big) &\ge C_0 \,\phi_{\abs{\bfA^{\textup{sym}}}}(\abs{\bfA^{\textup{sym}} -
         \bfB^{\textup{sym}}}) \,,
       \\
       \abs{\SSS(\bfA) - \SSS(\bfB)} &\le C_1 \,
       \phi'_{\abs{\bfA^{\textup{sym}}}}\big(\abs{\bfA^{\textup{sym}} -
         \bfB^{\textup{sym}}}\big)
     \end{aligned}
   \end{equation}
   are satisfied for all $\bfA,\bfB \in \mathbb{R}^{d \times d} $.  The
   constants $C_0,C_1>0$ and $p\in (1,\infty)$ are called the 
     characteristics of $\SSS$.
\end{assumption}

\begin{remark}
  {\rm (i) Let $\phi$ be defined in \eqref{eq:5} and let
    $\{\phi_a\}_{a\ge 0}$ be the corresponding family of the~shifted \mbox{N-functions}. Then, the operators 
    $\SSS_a\colon\mathbb{R}^{d\times d}\to \smash{\mathbb{R}_{\textup{sym}}^{d\times
      d}}$, $a \ge 0$, defined, for every $a \ge 0$
    and~$\bfA \in \mathbb{R}^{d\times d}$, via
\begin{align}
  \label{eq:flux}
  \SSS_a(\bfA) \coloneqq 
  \frac{\phi_a'(\abs{\bfA^{\textup{sym}}})}{\abs{\bfA^{\textup{sym}}}}\,
  \bfA^{\textup{sym}}\,, 
\end{align}
have $(p, \delta +a)$-structure (cf.~Remark~\ref{rem:phi_a}).  In this case, the characteristics of
$\SSS_a$ depend only on $p\in (1,\infty)$ and are independent of
$\delta \geq 0$ and $a\ge 0$.

{(ii) Note that the $(p,\delta)$-structure contains as particular cases, e.g.,~the Ladyzhenskaya model, the Smagorinski model, power law models, the
Carreau--Yasuda model and the  Powell--Eyring model. 
  }}
\end{remark}

\subsection{The $p$-Navier--Stokes system} 
Let us briefly recall some well-known~facts about the $p$-Navier--Stokes system \eqref{eq:p-navier-stokes}. For $p\in (1,\infty)$, we define the function spaces
 \begin{align*}
     \begin{aligned}
     V&\coloneqq (W^{1,p}(\Omega))^d\,,&&\Vo\coloneqq (W^{1,p}_0(\Omega))^d\,,\\
     Q&\coloneqq L^{p'}(\Omega)\,,&&\Qo\coloneqq L_0^{p'}(\Omega)\coloneqq \big\{f\in L^{p'}(\Omega)\;|\;\mean{f}_\Omega=0\big\}\,.
     \end{aligned}
 \end{align*}
    With this notation, the weak formulation of problem \eqref{eq:p-navier-stokes} is the following.
    
\textbf{Problem (Q).} For given $(\bfg,\bfG,g,\bfv_0)^\top\!\in\!
L^{p'}(\Omega)\times L^{p'}(\Omega)\times
L^p(\Omega)\times W^{\smash{1-\frac{1}{p},p}}(\partial\Omega)$ with
\eqref{eq:compatibility}, find $(\bfv,q)\in V\times \Qo$ such that
$\bfv=\bfv_0$ in $L^p(\partial\Omega)$ and for all $(\bfz,z)^\top\in
\Vo\times Q $, it holds 
\begin{align}\label{eq:q1}
    (\SSS(\bfD\bfv),\bfD\bfz)+([\nabla\bfv]\bfv,\bfz)-(q,\divo\bfz)&=(\bfg,\bfz)+(\bfG,\nabla\bfz)\,,\\
    (\divo\bfv,z)&=(g,z)\label{eq:q2}\,.
\end{align}

Alternatively, we can reformulate Problem (Q) ``hiding'' the pressure.

\textbf{Problem (P).} For given $(\bfg,\bfG,g,\bfv_0)^\top\!\in\! L^{p'}(\Omega)\times L^{p'}(\Omega)\times L^p(\Omega)\times W^{\smash{1-\frac{1}{p},p}}(\partial\Omega)$ with \eqref{eq:compatibility}, find $\bfv\in V(g)$ with $\bfv=\bfv_0$ in $L^p(\partial\Omega)$
such~that~for~all~${\bfz\in \Vo(0)}$, it holds
\begin{align}\label{eq:p}
    (\SSS(\bfD\bfv),\bfD\bfz)+([\nabla\bfv]\bfv,\bfz)&=(\bfg,\bfz)+(\bfG,\nabla\bfz)\,,
\end{align}
where for every $f\in L^p_0(\Omega)$
\begin{align*}
    V(f)\coloneqq \{\bfz\in V\mid \divo \bfz=f\}\,,\quad \Vo(f)\coloneqq V(f)\cap \Vo\,.
\end{align*}

The notions ``Problem (Q)'' and ``Problem (P)'' are traditional, cf.~\cite{BF1991}. Note that $\Vo(f)\hspace{-0.17em}\neq\hspace{-0.17em} \emptyset$ \hspace{-0.1mm}for \hspace{-0.1mm}all \hspace{-0.1mm}$f\hspace{-0.17em}\in\hspace{-0.17em} L^p_0(\Omega)$ \hspace{-0.1mm}by \hspace{-0.1mm}the \hspace{-0.1mm}solvability \hspace{-0.1mm}of \hspace{-0.1mm}the \hspace{-0.1mm}divergence \hspace{-0.1mm}equation~\hspace{-0.1mm}(cf.~\hspace{-0.1mm}\mbox{\cite[\hspace{-0.15em}Thm.~\hspace{-0.15em}6.6]{john}}). This and the theory of pseudo-monotone operators  yield the existence of a weak solution of Problem (P) if $p> 2$, cf.~\cite[Theorem 2.27]{r-mol-inhomo}. DeRham's lemma, the solvability of the divergence equation, and the negative norm theorem, then, ensure the solvability~of Problem (Q).

\subsection{DG spaces, jumps and averages}\label{sec:dg-space}

\subsubsection{Triangulations}

\!In \hspace{-0.1mm}what \hspace{-0.1mm}follows, \hspace{-0.1mm}we \hspace{-0.1mm}always \hspace{-0.1mm}denote \hspace{-0.1mm}by \hspace{-0.1mm}$\mathcal{T}_h$, \hspace{-0.1mm}$h\!>\!0$,~\hspace{-0.1mm}a~\hspace{-0.1mm}\mbox{family}~\hspace{-0.1mm}of uniformly \hspace{-0.1mm}shape \hspace{-0.1mm}regular \hspace{-0.1mm}and \hspace{-0.1mm}conforming \hspace{-0.1mm}triangulations \hspace{-0.1mm}of~\hspace{-0.1mm}${\Omega\hspace{-0.15em}\subseteq\hspace{-0.15em} \mathbb{R}^d}$,~${d\hspace{-0.15em}\in\hspace{-0.15em} \set{2,3}}$,~\hspace{-0.1mm}cf.~\hspace{-0.1mm}\cite{BS08},~\hspace{-0.1mm}each consisting of \mbox{$d$-dimensional} simplices $K$.
The parameter $h>0$, refers to the maximal mesh-size of $\mathcal{T}_h$, i.e., if we define $h_K\coloneqq \textup{diam}(K)$ for all $K\in \mathcal{T}_h$,~then~${h\coloneqq \max_{K\in \mathcal{T}_h}{h_K}}$.
For simplicity, we always assume  that  $h \le 1$.
 For a simplex $K\! \in\! \mathcal{T}_h$,
we denote by $\rho_K\!>\!0$, the supremum of diameters~of~inscribed~balls. We assume that there~is~a~constant~$\omega_0\!>\!0$, independent
of $h>0$, such that ${h_K}{\rho_K^{-1}}\le
\omega_0$ for every $K \in \mathcal{T}_h$. The smallest such constant~is~called~the~chunkiness~of $(\mathcal{T}_h)_{h>0}$. 
By $\Gamma_h^{i}$, we denote the~interior faces and put $\Gamma_h\coloneqq  \Gamma_h^{i}\cup \partial\Omega$.
We assume~that~each~simplex~${K \in \mathcal{T}_h}$ has at most one~face~from~$\partial\Omega$.  We introduce the following scalar products~on~$\Gamma_h$
\begin{align*}
  \skp{f}{g}_{\Gamma_h} \coloneqq  \smash{\sum_{\gamma \in \Gamma_h} {\langle f, g\rangle_\gamma}}\,,\quad\text{ where }\quad\langle f, g\rangle_\gamma\coloneqq \int_\gamma f g \,\textup{d}s\quad\text{ for all }\gamma\in \Gamma_h\,,
\end{align*}
if all the integrals are well-defined. Similarly, we define the products 
$\skp{\cdot}{\cdot}_{\partial\Omega}$ and~$\skp{\cdot}{\cdot}_{\Gamma_h^{i}}$. 

\subsubsection{Broken function spaces and projectors}

\!For every $m \hspace{-0.1em}\in\hspace{-0.1em} \setN_0$~and~${K\hspace{-0.1em}\in\hspace{-0.1em} \mathcal{T}_h}$,
we denote by ${\mathcal P}_m(K)$, the space of
polynomials of degree at most $m$ on $K$. Then, for
given~$p\in (1,\infty)$ and $k \in \setN_0$,
we define the spaces
\begin{align}
  \begin{split}
    Q_h^k&\coloneqq \big\{ q_h\in L^1(\Omega)\,\mid q_h|_K\in \mathcal{P}_k(K)\text{ for all }K\in \mathcal{T}_h\big\}\,,\\
    V_h^k&\coloneqq \big\{\bfv_h\in L^1(\Omega)^d\,\mid \bfv_h|_K\in \mathcal{P}_k(K)^d\text{ for all }K\in \mathcal{T}_h\big\}\,,\\
    X_h^k&\coloneqq \big\{\bfX_h\in L^1(\Omega) ^{d \times d}\,\mid \bfX_h|_K\in \mathcal{P}_k(K) ^{d \times d}\text{ for all }K\in \mathcal{T}_h\big\}\,,\\
        W^{1,p}(\mathcal T_h)&\coloneqq \big\{\bfw_h\in L^1(\Omega)^d\mid \bfw_h|_K\in W^{1,p}(K)^d\text{ for all }K\in \mathcal{T}_h\big\}\,.
  \end{split}\label{eq:2.19}
\end{align}
In addition, for given $k \in \setN_0$, we define $\Qhkc\coloneqq  Q_h^k\cap C^0(\overline{\Omega})$ and $\Vhkc\coloneqq  V_h^k\cap C^0(\overline{\Omega})$.
Note that  $W^{1,p}(\Omega)\subseteq \WDG$ and
$V_h^k\subseteq \WDG$. We denote by ${\PiDG:L^1(\Omega)\to V_h^k}$, the (local)
$L^2$--projection into $V_h^k$, which for every $\bfv \hspace{-0.1em}\in\hspace{-0.1em}
\smash{L^1(\Omega)}$ and $\bfz_h\hspace{-0.1em}
\in \hspace{-0.1em}V_h^k$ is~defined~via
\begin{align}
  \label{eq:PiDG}
  \bighskp{\PiDG \bfv}{\bfz_h}=\hskp{\bfv}{\bfz_h}\,.
\end{align}
Analogously, we define the (local)
$L^2$--projection into $X_h^k$, i.e., ${\PiDG\colon L^1(\Omega) \to \Xhk}$.

For  every $\bfw_h\hspace{-0.1em}\in\hspace{-0.1em} \WDG$, we denote by $\nabla_h \bfw_h\hspace{-0.1em}\in\hspace{-0.1em} L^p(\Omega)$,
the \textbf{local gradient},~defined via
$(\nabla_h \bfw_h)|_K\hspace{-0.1em}\coloneqq \hspace{-0.1em}\nabla(\bfw_h|_K)$ for~all~${K\hspace{-0.1em}\in\hspace{-0.1em}\mathcal{T}_h}$.
For every $K\hspace{-0.1em}\in\hspace{-0.1em} \mathcal{T}_h$, ${\bfw_h\hspace{-0.1em}\in\hspace{-0.1em} \WDG}$~admits~an
interior~trace~${\textrm{tr}^K(\bfw_h)\in L^p(\partial K)}$. For each face
$\gamma\in \Gamma_h$ of a given simplex $K\in \mathcal{T}_h$, we define this
interior trace by
$\smash{\textup{tr}^K_\gamma(\bfw_h)\in L^p(\gamma)}$. Then, for
 every $\bfw_h\in \WDG$ and interior faces $\gamma\in \Gamma_h^{i}$ shared by
adjacent elements $K^-_\gamma, K^+_\gamma\in \mathcal{T}_h$,
we~denote~by
\begin{align}
  \{\bfw_h\}_\gamma&\coloneqq \smash{\frac{1}{2}}\big(\textup{tr}_\gamma^{K^+}(\bfw_h)+
  \textup{tr}_\gamma^{K^-}(\bfw_h)\big)\in
  L^p(\gamma)\,, \label{2.20}\\
  \llbracket\bfw_h\otimes\bfn\rrbracket_\gamma
  &\coloneqq \textup{tr}_\gamma^{K^+}(\bfw_h)\otimes\bfn^+_\gamma+
    \textup{tr}_\gamma^{K^-}(\bfw_h)\otimes\bfn_\gamma^- 
    \in L^p(\gamma)\,,\label{eq:2.21}
\end{align}
the \textbf{average} and \textbf{normal jump}, resp., of $\bfw_h$ on $\gamma$.
    Moreover,  for every $\bfw_h\in \WDG$ and boundary faces $\gamma\!\in\! \partial\Omega$, we define boundary averages and 
    boundary~jumps,~resp.,~via
    \begin{align}
      \{\bfw_h\}_\gamma&\coloneqq \textup{tr}^\Omega_\gamma(\bfw_h) \in L^p(\gamma)\,,\label{eq:2.23a} \\
      \llbracket \bfw_h\otimes\bfn\rrbracket_\gamma&\coloneqq 
      \textup{tr}^\Omega_\gamma(\bfw_h)\otimes\bfn \in L^p(\gamma)\,,\label{eq:2.23} 
    \end{align}
    where $\bfn\colon\partial\Omega\to \mathbb{S}^{d-1}$ denotes the unit normal vector field to $\Omega$ pointing outward. 
    Analogously, we
    define $\{\bfX_h\}_\gamma$ and $ \llbracket\bfX_h\bfn\rrbracket_\gamma
    $~for all $\bfX_h \in \Xhk$ and $\gamma\in \Gamma_h$. Furthermore, if there is no
    danger~of~confusion, we will omit the index $\gamma\in \Gamma_h$,~in~particular,  if we interpret jumps and averages as global functions defined on the whole of $\Gamma_h$.

\subsubsection{DG gradient and jump operators}

For every $k\in \mathbb{N}_0$ and  face $\gamma\in \Gamma_h$, we define the
\textbf{(local)~jump~\mbox{operator}}
${\boldsymbol{\mathcal{R}}_{h,\gamma}^k \colon\WDG \to X_h^k}$~for~every~${\bfw_h\in \smash{\WDG}}$ (using
Riesz representation)~via
\begin{align}
  \big(\boldsymbol{\mathcal{R}}_{h,\gamma}^k\bfw_h,\bfX_h\big)
  \coloneqq \big\langle \llbracket\bfw_h\otimes\bfn\rrbracket_\gamma,\{\bfX_h\}_\gamma\big\rangle_\gamma
  \quad\text{ for all }\bfX_h\in X_h^k\,.\label{eq:2.25}
\end{align}
For every $k\in \mathbb{N}_0$, the \textbf{(global) jump operator} $\smash{\Rhk\coloneqq \sum_{\gamma\in \Gamma_h}{\boldsymbol{\mathcal{R}}_{\gamma,h}^k}\colon\WDG \to X_h^k}$, 
by definition, for every $\bfw_h\in \smash{\WDG}$ and  $\bfX_h\in X_h^k$ satisfies
\begin{align}
  \big(\Rhk\bfw_h,\bfX_h\big)=\big\langle
  \llbracket\bfw_h\otimes\bfn\rrbracket,\{\bfX_h\}\big\rangle_{\Gamma_h}\,.\label{eq:2.25.1}
\end{align}
In addition, for every $k\in \mathbb{N}_0$, the \textbf{DG gradient operator} 
$  \Ghk\colon\WDG\to L^p(\Omega)$ is  defined, for every $\bfw_h\in \smash{\WDG}$, via
\begin{align}
  \boldsymbol{\mathcal G}^k_{h}\bfw_h\coloneqq 
  \nabla_h\bfw_h-\boldsymbol{\mathcal R}^k_h\bfw_h
  \quad\text{ in }L^p(\Omega)\,.\label{eq:DGnablaR} 
\end{align}
In particular, for every $\bfw_h\in \smash{\WDG}$ and $\bfX_h\in X_h^k$, we have that 
\begin{align}
\big(\Ghk\bfw_h,\bfX_h\big)=(\nabla_h\bfw_h,\bfX_h)
  -\big\langle \llbracket
  \bfw_h\otimes\bfn\rrbracket,\{\bfX_h\}\big\rangle_{\Gamma_h}
\,.  \label{eq:DGnablaR1}
\end{align}
Note that for every $\bfv\in \Vo$, we have that $\Ghk \bfv=\nabla\bfv
$ in $L^p(\Omega)$.
Apart from that, for every $\bfw_h\in \WDG$, we  introduce the \textbf{DG norm} via
\begin{align}
    \|\bfw_h\|_{\nabla,p,h}\coloneqq \|\nabla_h\bfw_h\|_p+h^{\frac{1}{p}}\big\|h^{-1}\jump{\bfw_h\otimes \bfn}\big\|_{p,\Gamma_h}\,,
\end{align}
which turns $\WDG$ into a Banach space\footnote{The completeness of $\WDG$ equipped with $\|\cdot\|_{\nabla,p,h}$, for every fixed $h>0$, follows from ${\|\bfw_h\|_p\leq c\,\|\bfw_h\|_{\nabla,p,h}}$ for all $\bfw_h\in \smash{\WDG}$ (cf.~\cite[Lemma A.9]{dkrt-ldg}) and an element-wise application of the trace theorem.\vspace{-1cm}}. Owing to \cite[{(A.26)--(A.28)}]{dkrt-ldg}, there exists a constant $c>0$ such that $\bfw_h\in \smash{\WDG}$, it holds
\begin{align}\label{eq:eqiv0}
    c^{-1}\,\|\bfw_h\|_{\nabla,p,h}\leq \big\|\Ghnk\bfw_h\big\|_p+h^{\frac{1}{p}}\big\|h^{-1}\jump{\bfw_h\otimes \bfn}\big\|_{p,\Gamma_h}\leq c\,\|\bfw_h\|_{\nabla,p,h}\,.
\end{align}

\subsubsection{Symmetric \hspace{-0.1mm}DG \hspace{-0.1mm}gradient \hspace{-0.1mm}and \hspace{-0.1mm}symmetric \hspace{-0.1mm}jump \hspace{-0.1mm}operators}

\!\!\mbox{Following} \cite[Sec. 4.1.1]{BCPH20}, we define a symmetrized
version~of~the~DG~gradient.~For~${\bfw_h\hspace{-0.1em}\in\hspace{-0.1em}
  \WDG}$, we denote by
$\bfD_h\bfw_h\coloneqq [\nabla_h\bfw_h]^{\textup{sym}}\in
L^p(\Omega;\mathbb{R}^{d\times
  d}_{\textup{sym}})$,~the~\mbox{\textbf{local~symmetric~gradient}}.\linebreak
In addition, for every $k\in \setN_0$ and
$\smash{X_h^{\smash{k,\textup{sym}}}\coloneqq X_h^k\cap
  L^p(\Omega;\mathbb{R}^{d\times d}_{\textup{sym}})}$, we define the
\textbf{symmetric DG gradient
  operator}~$ \smash{\Dhk\colon\WDG\to L^p(\Omega;\mathbb{R}^{d\times
    d}_{\textup{sym}})}$ 
via
$\smash{\Dhk\bfw_h\coloneqq [\Ghk\bfw_h]^{\textup{sym}}}
  $, i.e., if we
introduce the \textbf{symmetric jump {operator}}
$\smash{\Rhks\colon\WDG\to X_h^{\smash{k,\textup{sym}}}}$~via
$\smash{\Rhks\bfw_h\coloneqq [\Rhk\bfw_h]^{\textup{sym}}\in
  X_h^{\smash{k,\textup{sym}}}}$, for every ${\bfw_h\in
  \WDG}$, we have that
	\begin{align}\label{eq:defD}
	\smash{\Dhk\bfw_h =\bfD_h\bfw_h
	-\Rhks\bfw_h
	\quad\text{ in }L^p(\Omega;\mathbb{R}^{d\times d}_{\textup{sym}})\,.}
	\end{align}
	In particular, for every $\bfw_h\in \WDG$ and $\bfX_h\in X_h^{k,\textup{sym}}$, we have that
	\begin{align}
		\smash{\big(\Dhk\bfw_h,\bfX_h\big)
		=(\bfD_h\bfw_h,\bfX_h)
		-\big\langle \llbracket \bfw_h\otimes\bfn\rrbracket,\{\bfX_h\}\big\rangle_{\Gamma_h}\,.}\label{eq:2.24}
	\end{align}
	Apart from that, for every $\bfw_h\in \WDG$, we  introduce the \textbf{symmetric DG norm}~via 
	\begin{align}
		\smash{\|\bfw_h\|_{\bfD,p,h}\coloneqq \|\bfD_h\bfw_h\|_p
		+\smash{h^{\frac{1}{p}}\big\|  h^{-1} \llbracket\bfw_h\otimes\bfn\rrbracket\big\|_{p,\Gamma_h}}}\,.\label{eq:2.29}
	\end{align}
	The following discrete Korn inequality on $V_h^k$ shows that $\|\cdot\|_{\bfD,p,h}\sim\|\cdot\|_{\nabla,p,h}$ on $V_h^k$ 
	and, thus, forms a cornerstone of the numerical  analysis of the $p$-Navier--Stokes~system~\eqref{eq:p-navier-stokes}.
	 
	 \begin{proposition}[Discrete Korn inequality]\label{korn}
	    For every $p\in (1,\infty)$ and $k\in \setN$, there
		exists a constant~${c_{\mathbf{Korn}}>0}$ such that 
		for every $\bfv_h\in V_h^k$, it holds
		\begin{align}
		\smash{\|\bfv_h\|_{\nabla,p,h}\leq c_{\mathbf{Korn}}\,\|\bfv_h\|_{\bfD,p,h}}\label{eq:equi1}\,.
		\end{align}
	\end{proposition}

	\begin{proof}
          In principle, we proceed as in \cite[Lemma 1]{BPG19}, where
          the case $p=2$~is~treated, and solely replace
          \cite[Inequality (12)]{BPG19} by \cite[Proposition
          14]{BCPH20} in doing so.  However, for the convenience
          of~the~reader, we give a thorough proof here.
		
          Let $\mathcal{N}_h^k$ be the set of degrees of freedom
          associated~with~$\Vhkc$ and $\mathcal{T}_h(z)$ the set of
          simplices sharing $z$ if $z\in \mathcal{N}_h^k$.  For every
          $k\hspace{-0.1em}\in\hspace{-0.1em} \setN$, the
          node-averaging quasi-interpolant
          $\smash{\mathcal{J}_{\textup{av},h}^k:V_h^k\to
            \Vhkco\coloneqq \Vhkc\cap \Vo}$ is defined  for every
          $\bfv_h\in V_h^k$ via 
		\begin{align*}
			\mathcal{J}_{\textup{av},h}^k\bfv_h\coloneqq \smash{\sum_{z\in \mathcal{N}_h^k}}{\langle\bfv_h\rangle_z\psi_z},\quad
		\langle\bfv_h\rangle_z\coloneqq 
			\begin{cases}
				\frac{1}{\textup{card}(\mathcal{T}_h(z))}
				\sum_{T\in \mathcal{T}_h(z)}{(\bfv_h|_{T})(z)}&\;\text{ if }
				z\in \Omega\\\mathbf{0}&\;\textup{ if }z\in \partial\Omega
			\end{cases}\,,
		\end{align*}
		where $\smash{(\psi_z)_{z\in \mathcal{N}_h^k}}$ is the respective shape basis of $\Vhkc$. 
		Appealing to \cite[Proposition~14]{BCPH20}, there exists a constant $c_{\textup{av}}>0$ such that for every $\bfv_h\in V_h^k$, it holds
		 \begin{align}\label{ineq0}
		 		\smash{\big\|\nabla_h\big(\bfv_h-\mathcal{J}_{\textup{av},h}^k\bfv_h\big)\big\|_p
		 	\leq c_{\textup{av}}\,h^{\frac{1}{p}}\big\|  h^{-1}\llbracket \bfv_h\rrbracket\big\|_{p,\Gamma_h}\,.}
		 \end{align}
	 	Using $\llbracket \mathcal{J}_{\textup{av},h}^k\bfv_h\otimes\bfn\rrbracket=\mathbf{0}$ in $L^p(\Gamma_h)$ since $\mathcal{J}_{\textup{av},h}^k\bfv_h\in  \Vo$ and $\vert \llbracket \bfv_h\rrbracket\vert =\vert \llbracket \bfv_h\otimes\bfn\rrbracket\vert $~in~$L^p(\Gamma_h)$, we deduce from \eqref{ineq0} that
		 \begin{align}
		 	\smash{\big\|\bfv_h-\mathcal{J}_{\textup{av},h}^k\bfv_h\big\|_{\nabla,p,h}
		 	\leq (1+c_{\textup{av}})\,h^{\frac{1}{p}}\big\|  h^{-1}\llbracket \bfv_h\otimes\bfn\rrbracket\big\|_{p,\Gamma_h}\,.}\label{ineq}
		 \end{align}
		 From \eqref{ineq0} and \eqref{ineq}, using $\|\mathcal{J}_{\textup{av},h}^k\bfv_h\|_{\nabla,p,h}
		   =\|\nabla\mathcal{J}_{\textup{av},h}^k\bfv_h\|_p
		  \leq \tilde{c}_{\textbf{Korn}}\|\bfD\mathcal{J}_{\textup{av},h}^k\bfv_h\|_p$, by the classical Korn inequality, inasmuch as $\smash{\mathcal{J}_{\textup{av},h}^k\bfv_h\in \Vo}$,~we~conclude~that
		 \begin{align*}
		 	\|\bfv_h\|_{\nabla,p,h}&\leq \big\|\bfv_h-\mathcal{J}_{\textup{av},h}^k\bfv_h\big\|_{\nabla,p,h}
		 	+\big\|\mathcal{J}_{\textup{av},h}^k\bfv_h\big\|_{\nabla,p,h}
		 	 \\[-0.25mm]&\leq (1+c_{\textup{av}})\,h^{\frac{1}{p}}\big\|  h^{-1} \llbracket\bfv_h\otimes\bfn\rrbracket\big\|_{p,\Gamma_h}
		 	 +\tilde{c}_{\textbf{Korn}}\,\big\|\bfD\mathcal{J}_{\textup{av},h}^k\bfv_h\big\|_p
		 	 \\[-0.25mm]&\leq (1+c_{\textup{av}})\,\|\bfv_h\|_{\bfD,p,h}
		 	 +\tilde{c}_{\textbf{Korn}}\,\|\bfD_h\bfv_h\|_p
		 	 +\tilde{c}_{\textbf{Korn}}\,\big\|\bfD_h\big(\mathcal{J}_{\textup{av},h}^k\bfv_h-\bfv_h\big)\big\|_p
		 	 \\[-0.25mm]&\leq (1+c_{\textup{av}})\,\|\bfv_h\|_{\bfD,p,h}
		 	 +\tilde{c}_{\textbf{Korn}}\,\|\bfD_h\bfv_h\|_p
		 	+\tilde{c}_{\textbf{Korn}}\,c_{\textup{av}}\,h^{\frac{1}{p}}\big\|  h^{-1}\llbracket \bfv_h\otimes\bfn\rrbracket\big\|_{p,\Gamma_h}
		 	 \\[-0.25mm]&\leq (1+\tilde{c}_{\textbf{Korn}})\,(1+c_{\textup{av}})\,\|\bfv_h\|_{\bfD,p,h}\,.
		 \end{align*}
	\end{proof}
	
	For the symmetric DG norm, there holds a similar relation like \eqref{eq:eqiv0}.
	
	\begin{proposition}\label{equivalences}
          For every $p\in (1,\infty)$
          and $k\in \setN_0$, there exists
          a constant~${c>0}$
          such that for every $\bfw_h\in \WDG$, it holds
		\begin{align}
		\smash{c^{-1}\,\|\bfw_h\|_{\bfD,p,h}
		\leq \big\|\Dhk\bfw_h\big\|_p
		+h^{\frac{1}{p}}\big\|  h^{-1}\llbracket\bfw_h\otimes\bfn\rrbracket\big\|_{p,\Gamma_h}
		\leq c\,\|\bfw_h\|_{\bfD,p,h}\,.}
		\label{eq:equi2}
		\end{align}
	\end{proposition}
	
	\begin{proof}
          Appealing to \cite[(A.25)]{dkrt-ldg}, there exists a
          constant $c_{\boldsymbol{\mathcal{R}}}>0$ such that for
          every $\bfw_h\in \WDG$, it holds
		\begin{align}
		\smash{\|\Rhk \bfw_h\|_p
			\leq c_{\boldsymbol{\mathcal{R}}}\, h^{\frac{1}{p}}\big\|h^{-1}\llbracket \bfw_h\otimes\bfn\rrbracket\big\|_{p,\Gamma_h}\,.}\label{eq:eqiv1}
		\end{align}
		Using the identity \eqref{eq:defD} and \eqref{eq:eqiv1}, for every $\bfw_h\in V_h^k$, we conclude that
		\begin{align*}
                  \|\bfw_h\|_{\bfD,p,h}
                  &  \leq \big\|\Dhk\bfw_h\big\|_p
                    +\big\|\Rhks\bfw_h\big\|_p
                    +h^{\frac{1}{p}}\big\| h^{-1}\llbracket\bfw_h\otimes\bfn\rrbracket\big\|_{p,\Gamma_h}
                  \\[-0.25mm]
                  &\leq  \|\Dhk\bfw_h\|_p 
                    +(1+c_{\boldsymbol{\mathcal{R}}})\,h^{\frac{1}{p}}\big\| h^{-1}\llbracket\bfw_h\otimes\bfn\rrbracket\big\|_{p,\Gamma_h}
                  \\[-0.25mm]
                  &\leq  \|\bfD_h\bfw_h\|_p +\big\|\Rhks\bfw_h\big\|_p
                    +(1+c_{\boldsymbol{\mathcal{R}}})\,h^{\frac{1}{p}}\big\| h^{-1}\llbracket\bfw_h\otimes\bfn\rrbracket\big\|_{p,\Gamma_h}
                  \\[-0.25mm]
                  &\leq 2\,(1+c_{\boldsymbol{\mathcal{R}}})\,\|\bfw_h\|_{\bfD,p,h}\,.
		\end{align*}
	\end{proof}
	
\subsubsection{DG divergence operator}
    The \textbf{local divergence} is defined, for every $\bfw_h\!\in\!\WDG$, via $\textup{div}_h\bfw_h\coloneqq \text{tr}(\nabla_h \bfw_h)\!\in\! L^p(\Omega)$. 
	In addition, for every~${k\in \setN_0}$, the \textbf{DG~\hspace{-0.2mm}divergence \hspace{-0.2mm}operator} 
	\hspace{-0.2mm}$\Divhk\hspace{-0.17em}\colon\hspace{-0.17em}\WDG\hspace{-0.17em}\to \hspace{-0.17em}L^p(\Omega)$ \hspace{-0.2mm}is \hspace{-0.2mm}defined,~\hspace{-0.2mm}for~\hspace{-0.2mm}every~\hspace{-0.2mm}${\bfw_h\hspace{-0.17em}\in\hspace{-0.17em} \WDG}$, via $\smash{\Divhk\bfw_h\coloneqq \text{tr}(\Ghk\bfw_h)
		=\text{tr}(\Dhk\bfw_h)\in L^p(\Omega)}$,~i.e., 
	\begin{align*}    
	    \smash{\Divhk\bfw_h=\textup{div}_h\bfw_h-\textup{tr}(\Rhk\bfw_h)\quad\text{ in }L^p(\Omega)\,.}
	\end{align*}
	In particular, 
	for every $\bfw_h\in \WDG$ and $z_h  \in \smash{Q_h^k}$, we have that
	\begin{align}
	    \label{eq:div-dg.0}
        \begin{aligned}
	    \big(\Divhk\bfw_h,z_h\big)&=(\textup{div}_h\bfw_h,z_h)
		-\big\langle \llbracket \bfw_h\cdot\bfn\rrbracket,\{z_h\}\big\rangle_{\Gamma_h}\\[-0.5mm]
		&=-(\bfw_h,\nabla_h z_h)
		+\big\langle \{\bfw_h\cdot\bfn\}, \llbracket z_h\rrbracket\big\rangle_{\Gamma_h^{i}}\,,
		\end{aligned}
	\end{align}
    i.e., $\smash{(\Divhk \bfw_h,z_h)=-( \bfw_h,\nabla z_h)}$ if $z_h  \in \smash{\Qhkc}$. Thus,  for all $\bfv\in \smash{W^{1,p}_0(\Omega)}$~and~$z_h\in \smash{\Qhkc}$
\begin{align}
  \label{eq:div-dg}
  \begin{aligned}
    \smash{\big(\Divhk \PiDG \bfv,z_h\big)=-( \bfv,\nabla z_h)=(\divo\bfv, z_h)}\,.
\end{aligned}
\end{align}
\section{Fluxes and LDG formulations}\label{sec:ldg}
	
	To obtain the LDG formulation of \eqref{eq:p-navier-stokes} for 
    $k \in \setN$, we multiply the equations in \eqref{eq:p-navier-stokes-DG}$_{1,2}$ by
    $\bfX_h,\bfY_h, \bfZ_h\in \Xhk$,~${\bfz_h\hspace{-0.1em}\in\hspace{-0.1em} \Vhk}$,~and~${z_h\hspace{-0.1em}\in\hspace{-0.1em} \Qhk}$, resp., use integration-by-parts, replace in the volume integrals 
    $\bfL$,~$\bfS$,~$\bfK$,~$\bfG$,~$\bfv$,~and~$q$ by the discrete objects $\bfL_h,\bfS_h,\bfK_h, \PiDG \bfG\in\Xhk$,~$\bfv_h\in \Vhk$,~and~${q_h\in \Qhko\coloneqq \Qhk\cap \Qo}$,~resp., and in the surface integrals $\bfS$, $\bfK$, $\bfG$, $\bfv$, and $q$ by the numerical fluxes     
    $\smash{\widehat{\bfS}_h\hspace{-0.12em}\coloneqq \hspace{-0.12em}\widehat{\bfS}(\bfv_h,\bfL_h)}$, $\smash{\widehat{\bfK}_h\hspace{-0.12em}\coloneqq \hspace{-0.12em}\widehat{\bfK}(\bfv_h)}$, $\smash{\widehat{\bfG}_h\hspace{-0.12em}\coloneqq \hspace{-0.12em}\widehat{\bfG}(\PiDG\bfG)}$,
    $\smash{\widehat{\bfv}_{h,\sigma}\hspace{-0.2em}\coloneqq \hspace{-0.12em}\widehat{\bfv}_\sigma(\bfv_h)}$,
    $\smash{\widehat{\bfv}_{h,q}}\hspace{-0.12em}\coloneqq \hspace{-0.12em}\smash{\widehat{\bfv}_q(\bfv_h,q_h)}$,
    and~$\smash{\widehat{q}_{h}\hspace{-0.12em}\coloneqq \hspace{-0.12em}\widehat{q}(q_h)}$,~resp., and obtain, also using that $\bfS_h -\tfrac{1}{2} \bfK_h-q_h \mathbf{I}_d\in X^{k,\textup{sym}}_h$, that
    \begin{align}
        \label{eq:flux-K}  \int_K \bfL_h : \bfX_h\,\textup{d}x &= -\int_K \bfv_h
            \cdot\divo \bfX_h\,\textup{d}x + \int_{\partial K} \widehat{\bfv}_{h,\sigma}
            \cdot \bfX_h \bfn\,\textup{d}s\,,
        \\
        \int_K \bfS_h : \bfY_h \,\textup{d}x &= \int_K \SSS(\bfL_h^{\textup{sym}}) : \bfY_h\,\textup{d}x\,,\notag
        \\
        \int_K \bfK_h : \bfZ_h \,\textup{d}x &= \int_K \bfv_h\otimes \bfv_h : \bfZ_h\,\textup{d}x\,,\notag
        \\
        \int_K \big(\bfS_h -\tfrac{1}{2} \bfK_h-q_h \mathbf{I}_d\big): \bfD_h \bfz_h\,\textup{d}x &= \int_K     \big(\bfg-\tfrac{1}{2}(\bfL_h-g\mathbf{I}_d)\bfv_h \big)\cdot
        \bfz_h +\PiDG\bfG \colon\nabla_h \bfz_h
        \,\textup{d}x\notag\\&\quad +\int_{\partial K} \bfz_h \cdot
        \big (\widehat{\bfS}_h-\widehat{\bfK}_h-\widehat{q}_h \mathbf{I}_d-\widehat{\bfG}_h\big) \bfn
        \,\textup{d}s \,,\notag\\
        \int_K \bfv_h \cdot \nabla_h z_h \,\textup{d}x &= \int_{\partial K} \widehat{\bfv}_{h,q}\cdot \bfn z_h     \,\textup{d}s-\int_{K} g z_h     \,\textup{d}x\,.\notag
    \end{align}
    For given boundary data
    $\bfv_0\in W^{\smash{1-\frac{1}{p}},p}(\partial\Omega)$ and given
    divergence $g\in L^p(\Omega)$ such that \eqref{eq:compatibility}
    is satisfied, we denote by
    $\bfv^* \hspace{-0.1em}\in\hspace{-0.1em} W^{1,p}(\Omega)$ an
    extension of $\bfv_0$ with
    $\divo \bfv^*\hspace{-0.1em}=\hspace{-0.1em}g$~in~$L^p(\Omega)$,
    which exists according to \cite[(3.13)]{r-mol-inhomo}.  Using this
    extension, {and restricting ourselves to the case that $q_h\in \Qhkco\vcentcolon =\Qhko\cap C^0(\overline{\Omega})$}
    the numerical~fluxes~are, for every stabilization
    parameter $\alpha>0$, defined~via
    \begin{align}
    \label{def:flux-v1}
     \flux{\bfv}_{h,\sigma}(\bfv_h) &\coloneqq  
      \begin{cases}
        \avg{\bfv_h} &\text{on $\Gamma_h^{i}$}
        \\
        \bfv^* &\text{on $\partial\Omega$}
      \end{cases}\,,
      \quad \flux{\bfv}_{h,q}(\bfv_h) \coloneqq  
      \begin{cases}
        \avg{\bfv_h} &\text{on $\Gamma_h^{i}$}\\
        \bfv^* &\text{on $\partial\Omega$}
      \end{cases}\,,\\[-0.25mm]
          \label{def:flux-q}
     \flux{q}(q_h) &\coloneqq 
        q_h 
         \quad \text{on $\Gamma_h$}\,,\\[-0.25mm]
      \label{def:flux-S}
     \flux{\bfS}(\bfv_h, \bfS_h,\bfL_h) &\coloneqq 
        \avg{\bfS_h} \hspace*{-0.1em}- \hspace*{-0.1em}\alpha\, \SSS_{\smash{\sss}}\big(h^{-1}\jump{(\bfv_h     \hspace*{-0.1em}-\hspace*{-0.1em} \bfv_0^*)\hspace*{-0.1em}\otimes\hspace*{-0.1em} \bfn}\big)
         \quad \text{on $\Gamma_h$}\,,\\[-0.25mm]
         \label{def:flux-K}
     \flux{\bfK}(\bfv_h) &\coloneqq 
        \avg{\bfK_h} 
         \quad \text{on $\Gamma_h$}\,,\\[-0.25mm]
      \label{def:flux-F}
      \flux{\bfG}(\PiDG\bfG) &\coloneqq  
        \bigavg{\PiDG \bfG} \quad\text{on $\Gamma_h$}\,,
    \end{align}
    where the operator $\SSS_{\smash{\sss}}$ is defined as in
    \eqref{eq:flux}.  Thus we arrive~at~an~inf-sup stable system
    without using a pressure stabilization. {If one wants to work with a
    discontinuous pressure one has to modify the fluxes as follows:
    $\flux{q}(q_h) \coloneqq  \avg{q_h} $ on $\Gamma_h$ and
    $\flux{\bfv}_{h,q}(\bfv_h) \coloneqq 
    \avg{\bfv_h}+h\jump{q_h\bfn}$ on $\Gamma_h^{i}$,
    $\flux{\bfv}_{h,q}(\bfv_h) \coloneqq  \bfv^* $ on
    $\partial\Omega$.}
    
    Proceeding as in \cite{dkrt-ldg} and, in addition, using that $\PiDG$
    is self-adjoint,~we~arrive~at~the \textbf{flux formulation}
    of~\eqref{eq:p-navier-stokes}: For given 
    $\bfv^*\in W^{1,p}(\Omega)$, 
    $\bfg\in \smash{L^{p'}(\Omega)} $, $\bfG\in
    \smash{L^{p'}(\Omega)}$, and $g \in L^p(\Omega)$,
    find $(\bfL_h,\bfS_h,\bfK_h,\bfv_h,q_h)^\top \in \Xhk \times \Xhk\times\Xhk\times\Vhk \times \Qhkco$ such
    that for all $(\bfX_h,\bfY_h,\bfZ_h,\bfz_h,z_h)^\top$ $ \in  \Xhk \times \Xhk\times\Xhk\times\Vhk \times \Qhkc$,  it     holds
    \begin{align}
        \hskp{\bfL_h}{\bfX_h} &= \bighskp{\Ghk \bfv_h + \Rhk\bfv^*}{
          \bfX_h}\,,  \notag
        \\[-0.25mm]
        \hskp{\bfS_h}{\bfY_h} &= \hskp{\SSS(\bfL_h^{\textup{sym}})}{
          \bfY_h}\,, \notag
           \\[-0.25mm]
       \label{eq:DG} \hskp{\bfK_h}{\bfZ_h} &= \hskp{\bfv_h\otimes \bfv_h}{
          \bfZ_h}\,, 
        \\[-0.25mm]
        \bighskp{\bfS_h-\tfrac{1}{2}\bfK_h-q_h\mathbf{I}_d}{\Dhk \bfz_h} &=
        \hskp{\bfg-\tfrac{1}{2}(\bfL_h-g\mathbf{I}_d)\bfv_h}{\bfz_h}+\bighskp{\bfG}{\Ghk\bfz_h} \notag
        \\[-0.25mm]
        &\quad - \alpha \bigskp{\SSS_{\smash{\sss}}(h^{-1} \jump{(\bfv_h -\bfv^*)\otimes
            \bfn})}{ \jump{\bfz_h \otimes \bfn}}_{\Gamma_h}\,, \notag\\[-1mm]
        \bighskp{\Divhk \bfv_h}{z_h}&=\bighskp{g-\textrm{tr}(\Rhk\bfv^*)}{z_h}\,.  \notag
    \end{align}  
    
    Next, we eliminate in the system~\eqref{eq:DG} the
    variables $ \bfL_h\hspace{-0.15em}\in\hspace{-0.15em} \Xhk$, $ \bfS_h\hspace{-0.15em}\in\hspace{-0.15em}     \Xhk$~and~${\bfK_h\hspace{-0.15em}\in\hspace{-0.15em} \Xhk}$ to derive a system only expressed in
    terms of the two variables ${\bfv_h\in\Vhk}$~and~${q_h\in \Qhkco}$. 
{    Note that it follows from \eqref{eq:DG}$_{1,2,3}$ that
    \begin{align*}
      \bfL_h= \Ghk \bfv_h + \Rhk\bfv^*\,, \quad
      \bfS_h=\PiDG\SSS(\bfL_h^{\textup{sym}})\,,\quad \bfK_h= \PiDG
      (\bfv_h \otimes \bfv_h)\,.
    \end{align*}
    If we insert this into \eqref{eq:DG}$_4$,
    we get the discrete counterpart~of~Problem~(Q):}
    
    \textbf{Problem (Q$_h$).} For given 
    $(\bfg,\bfG,g,\bfv^*)^\top\!\in\!
    \smash{L^{p'}(\Omega)}\times\smash{L^{p'}(\Omega)}\times
    L^p(\Omega)\times W^{1,p}(\Omega)$,
    find $(\bfv_h,q_h)^\top\!\in V_h^k\times \Qhkco$ such that 
     for all $(\bfz_h,z_h)^\top \in \Vhk\times\Qhkco$, it holds
    \begin{align}
        &\bighskp{\SSS(\Dhk \bfv_h + \Rhks\bfv^*)-\tfrac{1}{2}\bfv_h\otimes \bfv_h-q_h\mathbf{I}_d}{\Dhk
          \bfz_h}\notag
        \\[-0.25mm]
       \label{eq:primal0.1} &=     \bighskp{\bfg-\tfrac{1}{2}(\Ghk \bfv_h +  \Rhk\bfv^*-g\mathbf{I}_d)\bfv_h}{\bfz_h}+\bighskp{\bfG}{\Ghk\bfz_h}
        \\[-0.25mm]
        &\quad - \alpha \bigskp{\SSS_{\smash{\sssl}}(h^{-1} \jump{(\bfv_h -\bfv^*)\otimes
            \bfn})}{ \jump{\bfz_h \otimes \bfn}}_{\Gamma_h}\,,\notag\\
      \label{eq:primal0.2}&\bighskp{\Divhk \bfv_h}{z_h}=\bighskp{g-\textrm{tr}(\Rhk\bfv^*)}{z_h}\,.
    \end{align}
    
    Next, we eliminate in the system~\eqref{eq:primal0.1}, \eqref{eq:primal0.2}, the
    variable $q_h\in \Qhkco$ to derive a system only expressed in
    terms of the single variable $\bfv_h\hspace{-0.1em}\in\hspace{-0.1em} \smash{\Vhk}$. To~this~end,~for~${f\hspace{-0.1em}\in\hspace{-0.1em} L^p_0(\Omega)}$, we introduce
    \begin{align*}
        V_h^k(f)\coloneqq \big\{\bfv_h\in V_h^k\mid \bighskp{\Divhk \bfv_h}{z_h}=(f,z_h)\textup{ for all }z_h\in     \Qhkc\big\}\,,
    \end{align*}
    which is non-empty due to the surjectivity of the Bogovski\u{\i}
    operator from~$W^{1,p}_0(\Omega)$~into $L^p_0(\Omega)$ and
    \eqref{eq:div-dg}.  Then, because
    $\hskp{z_h\mathbf{I}_d}{\Dhk \bfz_h}=\hskp{z_h}{\Divhk \bfz_h}=0$
    for all $z_h\in \Qhkc$ and $\bfz_h\in V_h^k(0)$,
    we~obtain~the~discrete counterpart of Problem (P):
    
    \textbf{Problem (P$_h$).} For given 
    $(\bfg,\bfG,g,\bfv^*)^\top\in
    \smash{L^{p'}(\Omega)}\times\smash{L^{p'}(\Omega)}\times L^p(\Omega)\times  W^{1,p}(\Omega)$, find $\bfv_h\in V_h^k(g-\textup{tr}(\Rhk\bfv^*))$ such that for all $\bfz_h\in     V_h^k(0)$,~it~holds
    \begin{align} \label{eq:primal}
        \begin{aligned}
          &\bighskp{\SSS(\Dhk \bfv_h + \Rhks\bfv^*)-\tfrac{1}{2}\bfv_h\otimes \bfv_h}{\Dhk
          \bfz_h}\\
       &=     \bighskp{\bfg-\tfrac{1}{2}(\Ghk \bfv_h +  \Rhk\bfv^*-g\mathbf{I}_d)\bfv_h}{\bfz_h}+\bighskp{\bfG}{\Ghk\bfz_h}\\
       &\quad - \alpha \bigskp{\SSS_{\smash{\sssl}}(h^{-1} \jump{(\bfv_h -\bfv^*)\otimes
            \bfn})}{ \jump{\bfz_h \otimes \bfn}}_{\Gamma_h}\,.
        \end{aligned}
    \end{align}
    In particular, note that
    $V_h^k(g-\textup{tr}(\Rhk\bfv^*))\neq \emptyset$ since
    $g-\textup{tr}(\Rhk\bfv^*)\in
    L^p_0(\Omega)$~due~to~\eqref{eq:compatibility}.  Problem (Q$_h$)
    and Problem (P$_h$) are called \textbf{primal formulations} of the
    system~\eqref{eq:p-navier-stokes}. In \hspace{-0.1mm}order
    \hspace{-0.1mm}to \hspace{-0.1mm}formulate \hspace{-0.1mm}Problem
    \hspace{-0.1mm}(P$_h$) \hspace{-0.1mm}as \hspace{-0.1mm}an
    \hspace{-0.1mm}equivalent \hspace{-0.1mm}operator
    \hspace{-0.1mm}equation \hspace{-0.1mm}in
    \hspace{-0.1mm}$\Vhk(0)$,~\hspace{-0.1mm}we~\hspace{-0.1mm}make
    the ansatz
    \begin{align}\label{ansatz}
        \bfu_h\coloneqq \bfv_h -\PiDG \bfv^*\in
    \Vhk(0)\,.
    \end{align}
    Using \eqref{ansatz}, the system \eqref{eq:primal} is equivalent
    to~seeking~${\bfu_h \in \Vhk(0)}$ such that for all
    ${\bfz_h \in \Vhk(0)}$, it holds
    \begin{align}
      \label{eq:op}
      \smash{\bigskp{A_h^k\bfu_h}{\bfz_h}_{\WDG}= \bigskp{f_h^k}{\bfz_h}_{\WDG}\,,}
    \end{align}
    where, introducing the abbreviations
    \begin{align}
        \begin{aligned}
            \smash{\bfh_h^k\bfv^*}&\coloneqq \smash{\PiDG \bfv^*-\bfv^*\in L^p(\Omega)\,,}\\
            \smash{\bfH_h^k\bfv^*}&\coloneqq \smash{\Ghk(\PiDG     \bfv^*-\bfv^*)+\nabla\bfv^*\in L^p(\Omega)\,,}
        \end{aligned}
    \end{align}
    the non-linear operator 
    \begin{align*}
        \smash{A_h^k\coloneqq S_h^k+B_h^k\colon \WDG \to     (\WDG)^*\,,}
    \end{align*}
    and the linear functional
    $f_h^k\in (\WDG)^* $ are defined, for  every ${\bfu_h, \bfz_h \in \WDG}$,~via
    \begin{align}
      \bigskp{S_h^k\bfu_h}{\bfz_h}_{\WDG}
      &\coloneqq  \bighskp{\SSS(\Dhk \bfu_h + [\bfH_h^k\bfv^*]^{\textup{sym}})}{\Dhk
        \bfz_h}\label{eq:bS}  
      \\[-1mm]
      &\quad + \!\alpha \bigskp{\SSS_{\smash{\!\ssslo}}  (h^{\!-1} \jump{(\bfu_h\!+\!\bfh_h^k\bfv^*)\!\otimes\!
        \bfn})}{ \jump{\bfz_h\! \otimes\! \bfn}}_{\Gamma_h},\notag
      \\
      \bigskp{B_h^k\bfu_h}{\bfz_h}_{\WDG}
      &\coloneqq -\tfrac{1}{2}\bighskp{(\bfu_h+\PiDG\bfv^*)\otimes  (\bfu_h+\PiDG\bfv^*)}{\Dhk
        \bfz_h}      \label{eq:bB}
      \\[-1mm]
      &\quad+
        \tfrac{1}{2}\bighskp{\PiDG\bfz_h\otimes(\bfu_h+\PiDG\bfv^*)}{
        \Ghk\bfu_h +\bfH_h^k\bfv^*-g\mathbf{I}_d}\,,
        \notag
      \\\label{eq:bb}
      \bigskp{f_h^k}{\bfz_h}_{\WDG} &\coloneqq \hskp{\bfg}{\bfz_h}+\bighskp{\bfG}{\Ghk\bfz_h}\,.
    \end{align}
    The derivation of well-posedness (i.e., solvability), stability (i.e., a priori estimates),~and weak convergence of the operator equation~\eqref{eq:op}~to a weak formulation~of~the~system 
    \eqref{eq:p-navier-stokes} \hspace{-0.2mm}naturally \hspace{-0.2mm}motivates \hspace{-0.2mm}an \hspace{-0.2mm}abstract \hspace{-0.2mm}weak \hspace{-0.2mm}convergence \hspace{-0.2mm}theory \hspace{-0.2mm}based~\hspace{-0.2mm}on~\hspace{-0.2mm}a~\hspace{-0.2mm}\mbox{non-con}forming generalization of the standard notion of pseudo-monotonicity.
	
    \section{Non-conforming \hspace{-0.1mm}pseudo-monotonicity}\label{sec:non-conform}

    \!In \hspace{-0.1mm}this \hspace{-0.1mm}section, \hspace{-0.1mm}we \hspace{-0.1mm}introduce~\hspace{-0.1mm}a~\hspace{-0.1mm}general concept of non-conforming pseudo-monotonicity. It can be~understood~as~a~\mbox{further} generalization of the recently introduced  quasi-non-conforming pseudo-monotonicity, cf.~\cite{alex-rose-nonconform}.
    
	\begin{definition}[Non-conforming approximation]\label{def:non-conform_approx}
    Let $V$ be a reflexive Banach space and $(X_n)_{n\in \setN}$ a sequence of Banach spaces such that $V\hspace{-0.15em}\subseteq\hspace{-0.15em} X_n$ with ${\|\hspace{-0.1em}\cdot\hspace{-0.1em}\|_{V}\hspace{-0.15em}=\hspace{-0.15em}\|\hspace{-0.1em}\cdot\hspace{-0.1em}\|_{X_n}}$~on $V$ for all $n\hspace{-0.1em}\in\hspace{-0.1em} \setN$~and~let $Y$ be a 
	Banach space such that $X_n\subseteq Y$~for~all~${n\in \mathbb{N}}$.~Then, a sequence of reflexive Banach spaces $V_n\hspace{-0.1em}\subseteq\hspace{-0.1em} X_n$, $n\hspace{-0.1em}\in\hspace{-0.1em} \mathbb{N}$, is called \textbf{a non-conforming~ap-proximation of $V$ with respect to $(X_n)_{n\in \setN}$ and $Y$}, if the following~conditions~are satisfied:
	\begin{itemize}
	    \item[(NC.1)]\hypertarget{NC.1}{} There exists a dense subset $D\subseteq V$ such that for every $v\in D$, there exists a sequence $v_n\in V_n$, $n\in \mathbb{N}$, with $\|v_n-v\|_{X_n}\to 0$ $(n\to \infty)$.
	    \item[(NC.2)]\hypertarget{NC.2}{} For each sequence
              $v_n\hspace{-0.1em}\in \hspace{-0.1em}V_{m_n}$,
              $n\hspace{-0.15em}\in \hspace{-0.15em}\setN$, where
              $(m_n)_{n\in
                \setN}\hspace{-0.15em}\subseteq\hspace{-0.15em}
              \setN$~with~${m_n\hspace{-0.15em}\to\hspace{-0.15em}
                \infty}$~${(n\hspace{-0.1em}\to\hspace{-0.1em}
                \infty)}$, from $\sup_{n\in
                \mathbb{N}}{\|v_n\|_{X_{m_n}}}<\infty$, it follows the
              existence of a {subsequence} $(v_{n_\ell})_{\ell \in
                \setN}$ of $(v_n)_{n\in \setN}$ 
              and an element $v\in V$ such that $v_{n_\ell}\weakto v$ in $Y$
              $(\ell\to \infty)$.
	\end{itemize}
	\end{definition}

{        Condition (\hyperlink{NC.1}{NC.1}) represents the strong
        approximability of the spaces $(V_n)_{n\in \setN}$ of $V$,
        which can be found in the literature
        (cf.~\cite{ern-theory}). Condition (\hyperlink{NC.2}{NC.2}) is
        the non-conforming analogue of the weak compactness of bounded
        sets in reflexive Banach spaces (cf.~\cite{ern-book} for a
        particular realization). The abstract condition seems to be
        new, since previous convergence results for DG schemes are
        mostly based on convergence rates, for which this weak
        condition is not needed. }
        
	Next, we introduce an extension of the notion of pseudo-monotonicity~to~the~frame-work of non-conforming approximations. 
	
	\begin{definition}[Non-conforming pseudo-monotonicity]\label{def:non-conform_pseudo}
	Let  $(V_n)_{n\in \setN}$ be a  non-conforming approximation of $V$ with respect to $(X_n)_{n\in \setN}$ and $Y$. Then, a sequence of operators  $A_n:X_n\to X_n^*$, $n\in \setN$, is called \textbf{non-conforming pseudo-monotone with respect to $(V_n)_{n\in \setN}$ and $A:V\to V^*$}, if for 
	each sequence $v_n\in V_{m_n}$, $n\in \setN$, where $(m_n)_{n\in \setN}\subseteq \setN$ with $m_n\to \infty$ $(n\to \infty)$, from
	\begin{gather}
	    \sup_{n\in \mathbb{N}}{\|v_n\|_{X_{m_n}}}<\infty\,,\qquad v_n\weakto v\quad\text{ in } Y\quad(n\to \infty)\,,\label{def:non-conform_pseudo.1}\\[-1mm]
	    \limsup_{n\to \infty}{\langle A_{m_n}\!v_n,v_n-v\rangle_{X_{m_n}}}\leq 0\,,\label{def:non-conform_pseudo.2}
	\end{gather}
	it follows that $\langle Av,v-z\rangle_V\hspace{-0.15em}\leq\hspace{-0.15em} \liminf_{n\to \infty}{\langle A_{m_n}\!v_n,v_n-z\rangle_{X_{m_n}}}\hspace{-0.15em}$ for all~$z\hspace{-0.15em}\in\hspace{-0.15em} V\hspace{-0.15em}$.~In~\mbox{particular}, note that \eqref{def:non-conform_pseudo.1} already implies that 
		$v\in V$ due to (\hyperlink{NC.2}{NC.2}).
	\end{definition}
	
	A  useful property of pseudo-monotone~operators is its stability under summation. This property is  shared by non-conforming pseudo-monotone operators.
	
	\begin{lemma}\label{def:non-conform_sum}
		Let  $(V_n)_{n\in \setN}$ be a  non-conforming approximation of $V$ with respect to $(X_n)_{n\in \setN}$ and $Y$. If
		$A_n,B_n:X_n\to X_n^*$, $n\in \setN$,~are~non-conforming pseudo-monotone 
		with respect to $(V_n)_{n\in\setN}$ and ${A,B:V\to V^*}$, resp., then   
		${A_n+B_n:X_n\to X_n^*}$, $n\hspace{-0.1em}\in\hspace{-0.1em} \setN $, is non-conforming pseudo-monotone 
		with respect to $(V_n)_{n\in\setN}$ and ${A\hspace{-0.1em}+\hspace{-0.1em}B\hspace{-0.1em}\colon\hspace{-0.1em}V\hspace{-0.1em}\to\hspace{-0.1em} V^*\hspace{-0.1em}}$.
	\end{lemma}
	
	\begin{proof}
		Let $v_n\in V_{m_n}$, $n\in\setN$, where $(m_n)_{n\in\setN}\subseteq\setN$ 
		with $m_n\to \infty$ $(n\to\infty)$,~be~a sequence satisfying 
		\eqref{def:non-conform_pseudo.1}  and $\limsup_{n\to\infty}{\langle(A_{m_n}+B_{m_n})v_n,v_n
			-v\rangle_{X_{m_n}}}\leq
		0$.~For~${n\in\setN}$, set ${a_n\coloneqq \langle A_{m_n}\!v_n,v_n
		-v\rangle_{X_{m_n}}}$ and
		$b_n\coloneqq \langle B_{m_n}\!v_n,v_n
		-v\rangle_{X_{m_n}}$.~Then,~${\limsup_{n\to\infty}{a_n}\hspace{-0.1em}\leq\hspace{-0.1em}
		0}$ and $\limsup_{n\to\infty}{b_n}\hspace{-0.1em}\leq\hspace{-0.1em}
		0$. In fact, suppose the contrary,~e.g.,~that
		${\limsup_{n\to\infty}{a_n}\hspace{-0.1em}=\hspace{-0.1em}a\hspace{-0.1em}>\hspace{-0.1em}
			0}$. Then, there exists a subsequence such that
		$a_{n_\ell}\hspace{-0.1em}\to\hspace{-0.1em}
		a$ $(\ell\hspace{-0.1em}\to\hspace{-0.1em}\infty)$ and, consequently, ${\limsup_{\ell\to\infty}{b_{n_\ell}}\hspace{-0.1em}\leq\hspace{-0.1em}
		\limsup_{\ell\to\infty}{a_{n_\ell}+b_{n_\ell}}-\lim_{\ell\to\infty}{a_{n_\ell}}\hspace{-0.1em}\leq\hspace{-0.1em}-a\hspace{-0.1em}<\hspace{-0.1em}0}$,~i.e., a contradiction, as then the non-conforming pseudo-monotonicity of
		${B_n:X_n\to X_n^*}$, $n\in \setN$, {choosing $z=v$,}
		implies $0\leq \liminf_{ n\to\infty}{b_n}<
		0$. Therefore, ${\limsup_{n\to\infty}{a_n}\leq
		0}$ and $\limsup_{n\to\infty}{b_n}\leq
		0$~hold,  and the non-conforming pseudo-monotonicity of the sequences
		${A_n,B_n:X_n\to  X_n^*}$,~${n\in \setN}$,
		implies $\langle Av,v-z\rangle_V\leq
		\liminf_{n\to\infty}{\langle A_{m_n}\!v_n,v_n
			-z\rangle_{X_{m_n}}}$ and
		$\langle Bv,v-z\rangle_V\leq
		\liminf_{n\to\infty}$ $ \langle B_{m_n}\!v_n, v_n
			-z\rangle_{X_{m_n}}$ for all $z\in V$. Summing these
		inequalities yields the assertion.  
	\end{proof}
	
	The following theorem gives sufficient conditions on a family of non-conforming pseudo-monotone operators such that the associated non-conforming approximation scheme is well-posed, stable, and weakly convergent to a solution of the problem to be approximated.
	
	\begin{theorem}\label{thm:non-conform_lim}
		Let  $(V_n)_{n\in \setN}$ be a  non-conforming approximation of $V$ with respect to $(X_n)_{n\in \setN}$ and $Y$. Moreover, let $A_n:X_n\to X_n^*$, $ n\in \setN$, be non-conforming pseudo-monotone with respect to $(V_n)_{n\in \setN}$ and $A:V\to V^*$ with~the~following~properties:
		\begin{itemize}
		    \item[(AN.1)]\hypertarget{AN.1}{} For every $n\in \setN$, $A_n:X_n\to X_n^*$ is pseudo-monotone.
		    \item[(AN.2)]\hypertarget{AN.2}{} There exists a
                      weakly coercive mapping
                      $\mathscr{C}\hspace{-0.1em}\colon\hspace{-0.1em}\mathbb{R}^{\ge
                        0}\hspace{-0.1em}\to\hspace{-0.1em}
                      \mathbb{R}$, i.e.,
                      ${\mathscr{C}(a)\hspace{-0.1em}\to\hspace{-0.1em}
                        \infty}$~${(a\hspace{-0.1em}\to\hspace{-0.1em}
                        \infty)}$, and a constant $c>0$ such that for all
                      $z_n\in V_n$ and $n\in \setN$, it holds\vspace*{-0.75mm}
                      \begin{align*}
                      \smash{\langle A_n z_n,z_n\rangle_{X_n}\ge
                      \mathscr{C}(\|z_n\|_{X_n})\|z_n\|_{X_n} -c\,.}
                      \end{align*}
		    \item[(AN.3)]\hypertarget{AN.3}{}  There exists a
                      non-decreasing mapping
                      $\mathscr{B}\colon\mathbb{R}^{\ge 0}\to
                      \mathbb{R}^{\ge 0}$ such that for all
                      $z_n\in V_n$ and $n\in \setN$, it holds\vspace*{-0.75mm}
                      \begin{align*}
                      \smash{\| A_n z_n\|_{X_n^*}\leq
                      \mathscr{B}(\|z_n\|_{X_n})\,.}
                        \end{align*}
		\end{itemize}
		Let $f\in V^*$ be a linear functional such that there exists a sequence  $f_n\in X_n^*$, $n\in \mathbb{N}$, with the following properties:
		\begin{itemize}
		    \item[(BN.1)]\hypertarget{BN.1}{} It holds $\sup_{n\in \setN} {\|f_n\|_{X_n^*}}<\infty$. 
		    \item[(BN.2)]\hypertarget{BN.2}{} For each sequence $v_n\hspace{-0.15em}\in\hspace{-0.15em} V_{m_n}$, $n\hspace{-0.15em}\in\hspace{-0.15em} \setN$, where $(m_n)_{n\in\setN}\hspace{-0.15em}\subseteq\hspace{-0.15em}\setN$ 
		with~${m_n\hspace{-0.15em}\to\hspace{-0.15em} \infty}$~${(n\hspace{-0.15em}\to\hspace{-0.15em}\infty)}$, satisfying \eqref{def:non-conform_pseudo.1} for some $v\hspace{-0.1em}\in \hspace{-0.1em} V$, it follows that ${\langle f_n,v_n\rangle_{X_{m_n}}\hspace{-0.1em}\to\hspace{-0.1em} \langle f,v\rangle_V}$~${(n\hspace{-0.1em}\to\hspace{-0.1em} \infty)}$.

		\end{itemize}
		Then, the following statements apply:
	    \begin{itemize}
	        \item[(i)] For every $n\in \setN$, there exists $v_n\in V_n$ such that $(\textup{id}_{V_n})^*A_nv_n=f_n$~in~$V_n^*$, i.e.,
		for every $z_n\in V_n$, it holds $\langle A_n v_n,z_n\rangle_{X_n}=\langle f_n,z_n\rangle_{V_n}$.
		\item[(ii)] It holds $\sup_{n\in\setN}{\|A_nv_n\|_{X_n^*}+\|v_n\|_{X_n}}<\infty$.
		\item[(iii)] There exists a subsequence $v_{m_n}\hspace{-0.15em}\in\hspace{-0.15em} V_{m_n}$, $n\hspace{-0.15em}\in\hspace{-0.15em} \mathbb{N}$,~where~${(m_n)_{n\in\setN}\hspace{-0.15em}\subseteq\hspace{-0.15em}\setN}$~with~${m_n\hspace{-0.15em}\to \hspace{-0.15em}\infty}$ $(n\to\infty)$, and $v\in V$ such that $v_{m_n}\weakto v$ $(n\to \infty)$ in $Y$ and $Av=f$~in~$V^*$. 
	    \end{itemize}
	\end{theorem}
	
	\begin{proof}
	    \textit{(i).} Due to the conditions 
            (\hyperlink{AN.1}{AN.1})--(\hyperlink{AN.3}{AN.3}), for
            every fixed $n\in \mathbb{N}$, the operator ${A_n:X_n\to X_n^*}$ is pseudo-monotone, coercive and bounded, which, implies that also the restriction $(\textup{id}_{V_n})^*A_n:V_n\to V_n^*$ is pseudo-monotone, coercive and bounded.~As~a~result, since the spaces $(V_n)_{n\in \setN}$ are reflexive,
	    the main theorem on pseudo-monotone operators, cf.~\cite{zei-IIB}, yields the existence of $v_n\in V_n$~such~that~${(\textup{id}_{V_n})^*A_nv_n=f_n}$~in~$V_n^*$. 
	    
	    \textit{(ii).} For every $n\in \setN$, using (\hyperlink{AN.2}{AN.2}) and (i),~we~find~that
	   \begin{align*}
	     \mathscr{C}(\|v_n\|_{X_n})\|v_n\|_{X_n} -c\leq  \langle A_n v_n,v_n\rangle_{X_n}= \langle f_n,v_n\rangle_{V_n}\leq \|f_n\|_{V_n^*}\|v_n\|_{X_n}\,.
	   \end{align*}
	   Thus, (\hyperlink{BN.1}{BN.1})  and the  weakly coercivity of 
           $\mathscr{C}\colon\mathbb{R}^{\ge 0}\to \mathbb{R}$ imply
           $\sup_{n\in\setN}{\|v_n\|_{X_n}}<\infty$. Using
           (\hyperlink{AN.3}{AN.3}), we conclude that also $\sup_{n\in\setN}{\|A_nv_n\|_{X_n^*}}<\infty$.
	   
	   \textit{(iii).} Appealing to (ii) and (\hyperlink{NC.2}{NC.2}), there exists a subsequence $v_{m_n}\in V_{m_n}$, $n\in \mathbb{N}$, where $(m_n)_{n\in\setN}\hspace{-0.16em}\subseteq\hspace{-0.16em}\setN$ with $m_n\hspace{-0.16em}\to\hspace{-0.16em} \infty$ $(n\hspace{-0.16em}\to\hspace{-0.16em}\infty)$, and $v\hspace{-0.16em}\in\hspace{-0.16em} V$ such that ${v_{m_n}\hspace{-0.16em}\weakto\hspace{-0.16em} v}$~in~$Y$~${(n\hspace{-0.16em}\to\hspace{-0.16em} \infty)}$. 
	   In addition, (ii) implies ${\sup_{n\in\setN}{\|(\textup{id}_V)^*A_nv_n\|_{V^*}}\hspace{-0.1em}\leq\hspace{-0.1em} \sup_{n\in\setN}{\|A_nv_n\|_{X_n^*}}\hspace{-0.1em}<\hspace{-0.1em} \infty}$\footnote{Here, for every $n\in \mathbb{N}$, $(\textup{id}_V)^*:X_n^*\to V^*$ denote the adjoint operator to $\textup{id}_V:V\to X_n$.\vspace{-12mm}}, 
	   so that, by the reflexivity of $V^*$, there exist a not relabeled subsequence~and~some~${h\in V^*}$~such~that
	   \begin{align}
	       (\textup{id}_V)^*A_{m_n}\!v_{m_n}\weakto h\quad\textrm{ in }V^*\quad(n\to \infty)\,.\label{thm:non-conform_lim.5}
	   \end{align}
	   Let $z\in D$ be arbitrary. Then, (\hyperlink{NC.1}{NC.1}) yields a sequence $z_n\in V_n$, $n\in \mathbb{N}$, such that ${\|z_n-z\|_{X_n}\!\to\! 0}$ $(n\!\to\! \infty)$. In addition, (\hyperlink{NC.2}{NC.2}) yields that
	   $z_n\!\weakto\! z$~in~$Y$~${(n\!\to\! \infty)}$,~so~that, resorting to (\hyperlink{BN.2}{BN.2}), we deduce that
	   \begin{align}
	       \langle f_{m_n},z_{m_n}\rangle_{X_{m_n}}\to \langle f,z\rangle_V\quad(n\to \infty)\,.\label{thm:non-conform_lim.6}
	   \end{align}
	   On the other hand, taking into account (ii) and \eqref{thm:non-conform_lim.5}, we deduce that
	   \begin{align}
	   \begin{aligned}
	       \vert \langle A_{m_n}\!v_{m_n},z_{m_n}\rangle_{X_{m_n}} -\langle h,z\rangle_V\vert 
	       &\leq  \| A_{m_n}\!v_{m_n}\|_{X_{m_n}^*}\|z_{m_n}-z\|_{X_{m_n}} \\&\quad +\vert \langle(\textup{id}_V)^*A_{m_n}\!v_{m_n}- h,z\rangle_V\vert\to 0\quad(n\to\infty)\,.
	       \end{aligned}\label{thm:non-conform_lim.7}
	   \end{align}
	   Since $\langle A_{m_n}\!v_{m_n},z_{m_n}\rangle_{X_{m_n}}=\langle f_{m_n},z_{m_n}\rangle_{X_{m_n}}$ for all $n\in \setN$ (cf.~(i)) and $z\in D$ was chosen arbitrary, combining \eqref{thm:non-conform_lim.6} and \eqref{thm:non-conform_lim.7}, we find that~${\langle f,z\rangle_V=\langle h,z\rangle_V}$~for~all~${z\in D}$, i.e., due to the density of $D$ in $V$, it holds $f =h$ in $V^*$. Thus, using (\hyperlink{BN.2}{BN.2})~and~\eqref{thm:non-conform_lim.5},~we~get
	   \begin{align*}
	       \limsup_{n\to \infty}{\langle A_{m_n}\!v_{m_n},v_{m_n}\!-v\rangle_{X_{m_n}}}&= \limsup_{n\to \infty}{\big[\langle A_{m_n}\!v_{m_n},v_{m_n}\rangle_{X_{m_n}}\!-\langle (\textup{id}_V)^*A_{m_n}v_{m_n},v\rangle_V\big]}
	       \\[-0.5mm]&=
	       \lim_{n\to \infty}{\langle f_{m_n}v_{m_n},v_{m_n}\rangle_{X_{m_n}}}-\langle f,v\rangle_V=0\,,
	   \end{align*}
	   so that the non-conforming pseudo-monotonicity of $A_n:X_n\to X_n^*$, $n\in \setN$, (\hyperlink{BN.2}{BN.2}),~(ii) and \eqref{thm:non-conform_lim.5}, for every $z\in V$, imply that
	   \begin{align}
	       \label{thm:non-conform_lim.8}
	       \langle Av,v-z\rangle_V&\leq \liminf_{n\to \infty}{\langle A_{m_n}\!v_{m_n},v_{m_n}-z\rangle_{X_{m_n}}}
	       \\[-0.5mm]&= \lim_{n\to \infty}{\big[\langle f_{m_n},v_{m_n}\rangle_{X_{m_n}}}-\lim_{n\to \infty}{\langle (\textup{id}_V)^*A_{m_n}\!v_{m_n},z\rangle_V\big]}
	       \leq \langle f,v-z\rangle_V\,.\notag
	   \end{align}
	   Eventually, choosing $z=v\pm\tilde{z}\in V$ for arbitrary $\tilde{z}\in V$ in \eqref{thm:non-conform_lim.8}, we conclude that $\langle Av-f,\tilde{z}\rangle= 0$ for all $\tilde{z}\in V$, i.e., $Av=f$ in $V^*$.
	\end{proof}
	
\section{Well-posedness, \hspace{-0.1mm}stability \hspace{-0.1mm}and \hspace{-0.1mm}weak \hspace{-0.1mm}convergence}\label{sec:application}

\!\!In \hspace{-0.1mm}this~\hspace{-0.1mm}section,~\hspace{-0.1mm}we~\hspace{-0.1mm}prove the existence of a solution of Problem
(Q$_h$) (cf.~\eqref{eq:primal0.1}, \eqref{eq:primal0.2}), Problem
(P$_h$) (cf.~\eqref{eq:primal}), and the operator equation 
\eqref{eq:op} (i.e., their well-posedness), resp., their stability
(i.e., a priori estimates), and the weak convergence of the discrete
solutions to a solution of Problem (Q) (cf.~\eqref{eq:q1}, \eqref{eq:q2}) and
Problem (P) (cf.~\eqref{eq:p}), resp. To this end, we want to apply the abstract
framework presented in Section \ref{sec:non-conform}~to~the~shifted~problem~\eqref{eq:op} and, thus,  we must first show that
the discretely divergence-free
spaces~$(V_{h_n}^k\hspace{-0.1em}(0))_{n\in \setN}$, where
$(h_n)_{n\in \setN}\hspace{-0.1em}\subseteq\hspace{-0.1em} \mathbb{R}_{>0}$
satisfies $h_n\hspace{-0.1em}\to\hspace{-0.1em} 0$
$(n\hspace{-0.1em}\to\hspace{-0.1em} \infty)$, form a non-conforming
approximation~of $\Vo(0)$  with respect to $(\WDGn)_{n\in \setN}$ and
$L^p(\Omega)$, i.e., $\Vo(0)$ takes to role of $V$, $(\WDGn)_{n\in
  \setN}$ of $(X_n)_{n\in \setN}$, $L^p(\Omega)$ of $Y$, and
$(V_{h_n}^k\hspace{-0.1em}(0))_{n\in \setN}$ of $(V_n)_{n\in \setN}$
in Definition~\ref{def:non-conform_approx}.

\begin{lemma}\label{lem:nonconform}
    Let $(h_n)_{n\in \setN}\hspace{-0.1em}\subseteq\hspace{-0.1em} \mathbb{R}_{>0}$ be such that $h_n\hspace{-0.1em}\to\hspace{-0.1em} 0$ $(n\hspace{-0.1em}\to\hspace{-0.1em} \infty)$~and~let~${p\hspace{-0.1em}\in\hspace{-0.1em} (1,\infty)}$. Then, the sequence 
    $(V_{h_n}^k\hspace{-0.1em}(0))_{n\in \setN}$ is a non-conforming approximation of $\Vo(0)$ with respect to $(\WDGn)_{n\in \setN}$ and $L^p(\Omega)$. 
\end{lemma}

\begin{proof}
    Clearly, $(\Vo(0),\|\nabla\cdot\|_p)$, $(L^p(\Omega),\|\cdot\|_p)$ and $(\WDGn,\|\cdot\|_{\nabla,p,h})$, $n\in \setN$, are Banach spaces with $\|\cdot\|_{\nabla,p,h_n}=\|\nabla\cdot\|_p$ on $\Vo(0)$ and $\Vo(0)\subseteq \WDGn\subseteq L^p(\Omega)$ for all $n\in \setN$. In particular, $\Vo(0)$ and $V_{h_n}^k(0)$, $n\in \setN$, are reflexive Banach spaces.
    
	So, let us verify  (\hyperlink{NC.1}{NC.1}) 
	and (\hyperlink{NC.2}{NC.2}):
	
    \textit{(\hyperlink{NC.1}{NC.1}).} For $\bfv \in D\coloneqq \{\bfw\hspace{-0.1em}\in\hspace{-0.1em} C^\infty_0(\Omega)\mid \divo\bfw\hspace{-0.1em}=\hspace{-0.1em}0\}$, we set   
	${\bfv_n\coloneqq \smash{\Uppi_{h_n}^k}\!\bfv\in
        \smash{V_{h_n}^k(0)}}$, ${n\in \setN}$ (cf.~\eqref{eq:div-dg}). 
        In addition, by the approximation properties~of~$\PiDG$ 
	(cf.~\cite[Corollary~A.8, Corollary A.19]{kr-phi-ldg}), it holds $\|\bfv-\bfv_n\|_{\nabla,p,h_n}
	\leq 
	c\, h_n\|\bfv\|_{W^{2,p}(\Omega)}\to 0$ $(n\to\infty)$. 
	
        \textit{(\hyperlink{NC.2}{NC.2}).} Let
        $\bfv_n\in \smash{V_{h_{m_n}}^k\!(0)}$, $n\in \setN$, where
        $(m_n)_{n\in\setN}\subseteq
        \setN$~with~${m_n\to \infty}$~${(n\to \infty)}$, be such that
        $\sup_{n\in
          \setN}{\|\bfv_n\|_{\nabla,p,h_{m_n}}}\!\!<\!\infty$. Then,~appealing~to~\cite[Lemma~A.37]{kr-phi-ldg},~there~exist
        a subsequence $(\bfv_{n_\ell})_{\ell \in \setN}$ of
        $(\bfv_n)_{n\in \setN}$ and an element $\bfv\in \Vo$ such that
	\begin{align*}
	         \bfv_{n_\ell}\weakto       \bfv\quad\textrm{ in }L^p(\Omega)\quad(\ell\to \infty)\,.
	\end{align*}
	Next, let us verify that $\bfv\in \Vo(0)$. To this end, we
    take an arbitrary $z\in C_0^\infty(\Omega)$~and~set
    $\smash{z_n\coloneqq \mathcal{I}_{h_{m_n}}^kz
    \in
      Q_{h_{m_n},c}^k}$,
    $n\in
    \setN$,~where~$\smash{\mathcal{I}_{h_{m_n}}^k\colon C^0(\overline{\Omega})\to
      Q_{h_{m_n},c}^k}$~denotes~the nodal
    interpolation operator. Due to \cite[(11.17)]{EG21}, it
    holds $z_n\to z$~in~$W^{1,2}(\Omega)$~${(n\to \infty)}$. Since
    $(\bfv_{n_\ell},\nabla z_{n_\ell})\!=\!-(\smash{\Divhnk}\bfv_{n_\ell}, z_{n_\ell})=0$ for all
    $\ell\!\in\! \setN$, cf. \eqref{eq:div-dg.0}, we~conclude, by passing for $\ell\!\to\! \infty$, that
    $(\divo\bfv, z)=-(\bfv,\nabla z)=0$ for all
    $z\in C_0^\infty(\Omega)$,~i.e.,~it~holds~${\bfv\in \Vo(0)}$.
\end{proof}

Next, we show that the operators satisfy the assumptions of Theorem \ref{thm:non-conform_lim}.

\begin{lemma}\label{lem:ldg_stress}
    Let $(h_n)_{n\in \setN}\subseteq \mathbb{R}_{>0}$ be such that $h_n\to 0$ $(n\to \infty)$~and~let~${p\in (1,\infty)}$.
    Then, the sequence 
    $\smash{S_{h_n}^k}\colon\! \WDGn\!\to\! (\WDGn)^*$, $n\!\in\!
    \setN$, defined in \eqref{eq:bS}, satisfies
    (\hyperlink{AN.1}{AN.1})--(\hyperlink{AN.3}{AN.3})~and is
    non-conforming pseudo-monotone with respect to $(V_{h_n}^k(0))_{n
      \in\mathbb{N}} $ and ${S\colon\!\Vo(0)\!\to\!\Vo(0)^*}$ defined via
    $\langle S\bfu,\bfz\rangle_{\smash{\Vo(0)}}\!\coloneqq \!\hskp{\SSS
      (\bfD\bfu\!+\!\bfD\bfv^*\hspace{-0.5mm})}{\hspace{-0.2mm}\bfD\bfz}$  for every ${\bfu,\bfz\!\in\! \Vo(0)}$.
\end{lemma}

\begin{proof}
  \textit{(\hyperlink{AN.1}{AN.1}).} Using, for fixed
  $n\hspace{-0.13em}\in\hspace{-0.13em} \setN$, the continuity of
  ${\Ghnk\hspace{-0.13em}\colon\hspace{-0.13em}\WDGn\hspace{-0.13em}\to\hspace{-0.13
  em}
    L^p(\Omega)}$ and
  $(\bfu_{h_n}\mapsto \jump{\bfu_{h_n}\otimes \bfn})\colon\WDGn\to
  L^p(\Gamma_{h_n})$ (cf.~\eqref{eq:eqiv0}), from the standard theory
  of Nemytski\u{\i} operators, we deduce 
  the well-definedness,~continuity, monotonicity,~and, thus,
  pseudo-monotonicity of
  $S_{h_n}^k\colon\WDGn\to
  (\WDGn)^*$. \\[-3mm]
		
	\textit{(\hyperlink{AN.2}{AN.2}).} For every $n\in \setN$ and $\bfu_{h_n}\in \smash{V_{h_n}^k(0)}$, following the ansatz \eqref{ansatz}, i.e., $\bfv_{h_n}\coloneqq \bfu_{h_n}+\PiDGn\!\bfv^*\in \Vhnk$, and 
    abbreviating $\smash{\bfL_{h_n}^{\textup{sym}}}\coloneqq \smash{\Dhnk\!\bfv_{h_n}}+\smash{\Rhnk\bfv^*}\in L^p(\Omega)$ in the shifts to shorten the notation, we have that
{	\begin{align}\label{lem:ldg_stress.1}
        \begin{aligned}
	   &\big\langle S_{h_n}^k\!\bfu_{h_n},\bfu_{h_n}\big\rangle_{\WDGn}
	   \\
	    &=\bighskp{\SSS(\Dhnk\!\bfu_{h_n}+[\bfH_{h_n}^k\!\bfv^*]^{\textup{sym}})}{\Dhnk\!\bfu_{h_n}}	  
	    \\
	    &\quad+\alpha \big\langle\SSS_{\sssn}(h_n^{-1}\jump{(\bfu_{h_n}+\bfh_{h_n}^k\!\bfv^*)\otimes \bfn}),\jump{\bfu_{h_n}\otimes \bfn}\big\rangle_{\Gamma_{h_n}}
        \\
        &=\bighskp{\SSS(\Dhnk\!\bfv_{h_n}+\Rhnk\bfv^*)}{\Dhnk\!(\bfv_{h_n}-\PiDGn\!\bfv^*)}
        \\
        &\quad+\alpha
	        \big\langle\SSS_{\sssn}(h_n^{-1}\jump{(\bfv_{h_n}-\bfv^*)\otimes \bfn}),\jump{(\bfv_{h_n}-\PiDGn\!\bfv^*)\otimes \bfn}\big\rangle_{\Gamma_{h_n}}\,.
         \end{aligned}
	\end{align}
Next, we can follow the argumentation\footnote{Note that in \cite[Lemma 4.1]{kr-phi-ldg} the case of a full gradient is treated. However, the same argumentation works also for symmetric gradients.} of the proof of \cite[Lemma 4.1]{kr-phi-ldg}, to deduce the existence of a constant $c>0$, depending only on the characteristics of $\SSS$ and the chunkiness $\omega_0>0$, and of a constant $c_{\alpha}>0$, additionally depending on $\max\{1,\alpha\}$,~such~that
    \begin{align}
          &\bighskp{\SSS(\Dhnk\!\bfv_{h_n}+\Rhnk\bfv^*)}{\Dhnk\!(\bfv_{h_n}-\PiDGn\!\bfv^*)} \notag
          \\
          &\quad+\alpha
          \big\langle\SSS_{\sssn}(h_n^{-1}\jump{(\bfv_{h_n}-\bfv^*)\otimes  \bfn})
          ,\jump{(\bfv_{h_n}-\PiDGn\!\bfv^*)\otimes \bfn}\big\rangle_{\Gamma_{h_n}}\label{lem:ldg_stress.3}
          \\
          &\ge c\, \min \{1,\alpha\} \Big (\int_\Omega \varphi(\vert
          \Dhnk(\bfv_{h_n}-\bfv^*)\vert )\, \textup{d}x
          +h_n\int_ {\Gamma_{h_n}}\hspace*{-2mm}\varphi(h_n^{-1}\vert \jump{(\bfv_{h_n}-\bfv^*)\otimes  \bfn} \vert
          )\, \textup{d} s \Big )\notag
          \\
          &\quad -c_\alpha \, \int_\Omega \varphi(\vert \nabla \bfv^*\vert )\, \textup{d}x \,.\notag
    \end{align}
    Using, again, the ansatz \eqref{ansatz}, the approximation properties of $\PiDGn$ (cf. \cite[(A.17)]{dkrt-ldg}), that $\varphi(t)\ge c_0t^p-c_1\delta^p$ for all $t\ge 0$, $h_n\mathscr{H}^{d-1}(\Gamma_{h_n})\leq c \vert \Omega\vert$, and that $\varphi(t)\leq c(t^p+\delta^p)$ for all $t\ge 0$, we find that
	\begin{align}
        &\min \{1,\alpha\} \Big (\int_\Omega \varphi(\vert
          \Dhnk(\bfv_{h_n}-\bfv^*)\vert )\, \textup{d}x
          +h_n\int_ {\Gamma_{h_n}}\hspace*{-2mm}\varphi(h_n^{-1}\vert \jump{(\bfv_{h_n}-\bfv^*)\otimes  \bfn} \vert
          )\, \textup{d} s \Big ) \notag 
          \\
          &\ge c\, \min \{1,\alpha\} \Big (\int_\Omega \varphi(\vert
          \Dhnk \bfu_{h_n})\vert )\, \textup{d}x
          +h_n\int_ {\Gamma_{h_n}}\hspace*{-2mm}\varphi(h_n^{-1}\vert \jump{\bfu_{h_n}\otimes  \bfn} \vert
          )\, \textup{d} s \Big ) \label{lem:ldg_stress.4}
          \\
          &\quad 
          -c\, \min \{1,\alpha\} \Big (\int_\Omega \varphi(\vert
          \Dhnk(\PiDGn\!\bfv^*-\bfv^*)\vert )\, \textup{d}x
          +h_n\int_ {\Gamma_{h_n}}\hspace*{-2mm}\varphi(\vert h_n^{-1}(\PiDGn\!\bfv^*-\bfv^*)\vert
          )\, \textup{d} s \Big ) \notag 
          \\
          &\ge c\, \min \{1,\alpha\} \Big (\norm{\bfu_{h_n}}_{\bfD,p,h_n}^p-c\,\delta^p\,\vert \Omega\vert -c\,\delta^p\,h_n\,\mathscr{H}^{d-1}(\Gamma_{h_n})-c \,
          \int_\Omega \varphi(\vert \nabla \bfv^*\vert )\, \textup{d}x\Big ) \notag 
          \\
          &\ge c\, \min \{1,\alpha\} \Big (\norm{\bfu_{h_n}}_{\bfD,p,h_n}^p-c\,\delta^p\,\vert \Omega\vert -c_\alpha \,
         \|\nabla \bfv^*\|_p^p \Big )\,. \notag 
        \end{align}
      }%
        Eventually, using in \eqref{lem:ldg_stress.1} the estimates \eqref{lem:ldg_stress.3},
        \eqref{lem:ldg_stress.4}, the norm equivalence
        \eqref{eq:equi2}, and the discrete Korn inequality
        \eqref{eq:equi1}, we conclude that
    \begin{align}\label{lem:ldg_stress.4.1}
        \smash{\big\langle S_{h_n}^k\!\bfu_{h_n},\bfu_{h_n}\big\rangle_{\smash{\WDGn}}
         \ge c\, \min \{1,\alpha\}\,\|\bfu_{h_n}\|_{\nabla,p,h_n}^p-c_\alpha\,\delta^p\,\vert \Omega\vert-c_\alpha\, \|\nabla \bfv^*\|_p^p\,.}
    \end{align}

    \textit{(\hyperlink{AN.3}{AN.3}).} For every $n\!\in\! \setN$, $\bfu_{h_n}\!\in\! \smash{V_{h_n}^k\!(0)}$ and $\bfw_{h_n}\!\in\! \WDGn$,~again,~abbreviating $\smash{\bfL_{h_n}^{\textup{sym}}\coloneqq \!\Dhnk\!\bfu_{h_n}+[\bfH_{h_n}^k\!\bfv^*]^{\textup{sym}}\!\in\! L^p(\Omega)}$ in the shifts to shorten the notation,~we~find~that
    \begin{align}\label{lem:ldg_stress.5}
       &\langle S_{h_n}^k\!\bfu_{h_n},\bfw_{h_n}\rangle_{\WDGn}
	       \\& \quad\leq \big\|\SSS(\Dhnk\!\bfu_{h_n}+[\bfH_{h_n}^k\!\bfv^*]^{\textup{sym}})\big\|_{p'}\big\|\Dhnk\bfw_{h_n}\big\|_p\notag
	        \\&\quad\quad+\alpha \,\smash{h_n^{\frac{1}{p'}}}\big\|\SSS_{\smash{\sssn}}(h_n^{-1}\jump{(\bfu_{h_n}+\bfh_{h_n}^k\!\bfv^*)\otimes \bfn})\big\|_{p',\Gamma_{h_n}}\!\smash{h_n^{\frac{1}{p}}}\big\|h_n^{-1}\jump{\bfw_{h_n}\otimes \bfn}\big\|_{p,\Gamma_{h_n}}\notag
	        \\&\quad=\vcentcolon I_1\,\big\|\Dhnk\bfw_{h_n}\big\|_p+I_2 \,\smash{h_n^{\frac{1}{p}}}\,\big\|h_n^{-1}\jump{\bfw_{h_n}\!\otimes\! \bfn}\big\|_{p,\Gamma_{h_n}}\,.\notag
    \end{align}
    Note that appealing to \cite[Corollary A.8, Corollary A.19]{kr-phi-ldg},~for~every~${n\in \setN}$,~it~holds
	\begin{align}\label{lem:ldg_stress.2}
          h_n^{\frac{1}{p}}\smash{\big\|h_n^{-1}\jump{\bfh_{h_n}^k\!\bfv^*\otimes \bfn}\big\|_{p,\Gamma_{h_n}}+\big\|\bfH_{h_n}^k\!\bfv^*\big\|_p\leq c\,\|\nabla\bfv^*\|_p\,.}
	\end{align}
    Using this and that $\vert\SSS(\bfA)\vert \leq
    c\,(\delta^{p-1}+\vert\bfA\vert^{p-1})$ for all ${\bfA\in
      \mathbb{R}^{d\times d}}$ (cf.~Assumption~\ref{assum:extra_stress},~\eqref{eq:shift} and Remark \ref{rem:phi_a}) 
        ~we~deduce~that
    \begin{align}
    \label{lem:ldg_stress.6}
    \begin{aligned}
        I_1^{p'}&
        \leq c\,\big\|\Dhnk\!\bfu_{h_n}+[\bfH_{h_n}^k\!\bfv^*]^{\textup{sym}}\big\|_p^p+c\,\delta^p\,\vert \Omega\vert
        \\&
        \leq c\,\big\|\Ghnk\bfu_{h_n}\big\|_p^p+c\,\delta^p\,\vert \Omega\vert+c\,\|\nabla \bfv^*\|_p^p\,.
    \end{aligned}
    \end{align}
    Similarly, using that
    $\vert\SSS_a(\bfA)\vert \leq
    c\,(\delta^{p-1}+a^{p-1}+\vert\bfA\vert^{p-1})$ for all
    $\bfA\in \mathbb{R}^{d\times d}$~and~${a\ge 0}$ (cf.~Assumption
    \ref{assum:extra_stress}, \eqref{eq:shift} and Remark \ref{rem:phi_a}),
    $h_n\mathscr{H}^{d-1}(\Gamma_{h_n})\leq c\,\vert\Omega\vert$,
    \eqref{lem:ldg_stress.2},~the~\mbox{discrete} trace inequality
    (cf.~\cite[Corollary~A.19]{kr-phi-ldg}) and the
    $L^p$-stability~of~$\Pian\!$~(cf.~\cite[Corollary~A.8]{kr-phi-ldg}), we obtain
    \begin{align}
    \label{lem:ldg_stress.7}
    \begin{aligned}
        I_2^{p'}&\leq c\,h_n\big\|h_n^{-1}\jump{(\bfu_{h_n}+\bfh_{h_n}^k\!\bfv^*)\otimes \bfn})\big\|_{p,\Gamma_{h_n}}^p
        \\&\quad
        +c\,\delta^p\,h_n\,\mathscr{H}^{d-1}(\Gamma_{h_n})+c\,h_n\big\|\sssn\big\|_{p,\Gamma_{h_n}}^p
        \\&\leq c\,h_n\big\|h_n^{-1}\jump{\bfu_{h_n}\otimes \bfn}\big\|_{p,\Gamma_{h_n}}^p+c\,h_n\big\|h_n^{-1}\jump{\bfh_{h_n}^k\!\bfv^*\otimes \bfn}\big\|_{p,\Gamma_{h_n}}^p
        \\&\quad
        +c\,\delta^p\,\vert \Omega\vert+c\,\|\bfL_{h_n}\|_p^p
        \\&\leq c\,h_n\big\|h_n^{-1}\jump{\bfu_{h_n}\otimes \bfn}\big\|_{p,\Gamma_{h_n}}^p+c\,\delta^p\,\vert \Omega\vert+c\,\big\|\Ghnk\bfu_{h_n}\big\|_p^p+c\,\|\nabla \bfv^*\|_p^p\,.
        \end{aligned}
    \end{align}
    Finally, using \eqref{lem:ldg_stress.6}, \eqref{lem:ldg_stress.7} and the norm equivalence \eqref{eq:eqiv0} in \eqref{lem:ldg_stress.5}, we~conclude~that
	\begin{align}\label{lem:ldg_stress.7.1}
	    \big\|S_{h_n}^k\!\bfu_{h_n}\big\|_{\smash{(\WDGn)^*}}^{p'}\leq c\,\|\bfu_{h_n}\|_{\nabla,p,h_n}^p+c\,\|\nabla \bfv^*\|_p^p+c\,\delta^p\,\vert \Omega\vert\,.
	\end{align}
	
	\textit{Non-conforming pseudo-monotonicity.} 
    Let $\bfu_{h_{m_n}}\!\in\!  \smash{V_{h_{m_n}}^k(0)}$, $n\in \setN$,~where~$(m_n)_{n\in \setN} $ $\subseteq \setN$ with $m_n\to \infty $ $(n\to\infty)$, 
	be a sequence such that
	\begin{gather}
	\sup_{n\in \setN}{\|\bfu_{h_{m_n}}\|_{\nabla,p,h_{m_n}}}<\infty\,,
	\qquad\bfu_{h_{m_n}}\weakto\bfu\quad\text{ in }L^p(\Omega)\quad(n\to\infty)\,,\label{eq:ldg.3}
	\\[-1mm]
	\limsup_{n\to\infty}{\big\langle S_{h_{m_n}}^k\!\!\bfu_{h_{m_n}}
		,\bfu_{h_{m_n}}-\bfu\big\rangle_{W^{1,p}(\mathcal{T}_{h_{m_n}})}}\leq 0\,.\label{eq:ldg.4}
	\end{gather}
	We assume that $m_n\hspace{-0.1em}=\hspace{-0.1em}n$ for every
        $n\hspace{-0.1em}\in\hspace{-0.1em} \setN$. Recall that
        (\hyperlink{NC.2}{NC.2}) and \eqref{eq:ldg.3} imply ${\bfu\hspace{-0.1em}\in\hspace{-0.1em} \Vo(0)}$. In addition, using \cite[Lemma A.37]{kr-phi-ldg} and \eqref{eq:ldg.3}, we find that
        \begin{align}
          \Ghnk\bfu_{h_n}\rightharpoonup\nabla\bfu\quad\text{ in }L^p(\Omega)\quad(n\to \infty)\,.\label{eq:Dconv}
        \end{align}
        We abbreviate
        $\smash{\bfL_{h_n}^{\textup{sym}}\hspace{-0.1em}\coloneqq \hspace{-0.1em}
          \Dhnk\!\bfu_{h_n}\hspace{-0.1em}+\hspace{-0.1em}[\bfH_{h_n}^k\!\bfv^*]^{\textup{sym}}}$
        for all $n\in \setN$. For~$\bfz\hspace{-0.1em}\in\hspace{-0.1em} \Vo(0)$~and~${s\hspace{-0.1em}\in\hspace{-0.1em} \left(0,1\right)}$, we set
        $\bfz_s\coloneqq (1-s)\bfu+s\bfz\in \Vo(0)$.  Using the
        monotonicity~of~$\SSS$ (cf.~Assumption~\ref{assum:extra_stress}) and $\Dhnk\hspace{-0.1em}\bfz\hspace{-0.1em}=\hspace{-0.1em}\bfD\bfz$,
        $\Dhnk\hspace{-0.1em}\bfu\hspace{-0.1em}=\hspace{-0.1em}\bfD\bfu$~in~$L^p(\Omega)$ in the volume terms, that
        ${\jump{\bfu\hspace{-0.12em}\otimes\hspace{-0.12em} \bfn}\hspace{-0.12em}=\hspace{-0.12em}\jump{\bfz\hspace{-0.12em}\otimes\hspace{-0.12em} \bfn}\hspace{-0.12em}=\hspace{-0.12em}\bfzero}$~on~$\Gamma_h$,  adding and subtracting $\bfh_{h_n}^k\bfv^*$, and,
        again, the monotonicity~of~$\SSS$ for the~terms~on~$\Gamma_h$,
        we deduce that
		\begin{align}
		&(1-s)\,\big\langle S_{h_n}^k\bfu_{h_n},\bfu_{h_n}-\bfu\big\rangle_{\smash{\WDGn}}+s\,\big\langle S_{h_n}^k\!\bfu_{h_n},\bfu_{h_n}-\bfz\big\rangle_{\smash{\WDGn}}\notag
		\\&\quad\ge (1-s)\,\big(\SSS (\bfD\bfz_s+[\bfH_{h_n}^k\!\bfv^*]^{\textup{sym}}),
		\Dhnk\!\bfu_{h_n}-	\bfD\bfu\big)\notag\\&\quad\quad
		+s\,\big(\SSS (\bfD\bfz_s+[\bfH_{h_n}^k\!\bfv^*]^{\textup{sym}}),
		\Dhnk\!\bfu_{h_n}-	\bfD\bfz\big)\notag
		\\&\quad\quad+\alpha\,\big\langle \SSS_{\smash{\sssn}}(h^{-1}\jump{(\bfu_{h_n}\!+\bfh_{h_n}^k\!\bfv^*)\otimes \bfn}),
		\jump{(\bfu_{h_n}\!\pm\bfh_{h_n}^k\!\bfv^*)\otimes \bfn}\big\rangle_{\smash{\Gamma_{h_n}}}\label{eq:ldg.5}
		\\&\quad\ge (1-s)\,\big(\SSS (\bfD\bfz_s+[\bfH_{h_n}^k\!\bfv^*]^{\textup{sym}}),
		\Dhnk\!\bfu_{h_n}-	\bfD\bfu\big)\notag\\&\quad\quad
		+s\,\big(\SSS (\bfD\bfz_s+[\bfH_{h_n}^k\!\bfv^*]^{\textup{sym}}),
		\Dhnk\!\bfu_{h_n}-	\bfD\bfz\big)\notag
		\\&\quad\quad-\alpha\,\big\langle\SSS_{\smash{\sssn}}(h^{-1}\jump{(\bfu_{h_n}+\bfh_{h_n}^k\!\bfv^*)\otimes \bfn}),
		\jump{\bfh_{h_n}^k\!\bfv^*\otimes \bfn}\big\rangle_{\smash{\Gamma_{h_n}}}\,.\notag
		\end{align}
		Then, proceeding as in \eqref{lem:ldg_stress.7}, we find that
		\begin{align*}
		    \sup_{n\in \setN}{h_n^{\frac{1}{p'}}\smash{\big\|\SSS_{\smash{\sssn}}(h^{-1}\jump{(\bfu_{h_n}+\bfh_{h_n}^k\!\bfv^*)\otimes \bfn})\big\|_{p',\Gamma_{h_n}}<\infty}\,.}
		\end{align*}
		Thus, using that $\smash{h_n^{\frac{1}{p}}}\|h_n^{-1}\bfh_{h_n}^k\!\bfv^*\|_{p,\Gamma_{h_n}}\to 0$ $(n\to \infty)$ (cf.~\cite[Lemma A.27]{kr-phi-ldg}),~we~obtain
		\begin{align}
		    \big\langle\SSS_{\smash{\sssn}}(h_n^{-1}\jump{(\bfu_{h_n}+\bfh_{h_n}^k\!\bfv^*)\otimes \bfn}),
		\jump{\bfh_{h_n}^k\!\bfv^*\otimes \bfn}\big\rangle_{\smash{\Gamma_{h_n}}}\to 0\quad(n\to \infty)\,.\label{eq:ldg.6}
		\end{align}
	    Finally, 
		dividing by $s>0$ and taking the limes inferior in \eqref{eq:ldg.5} 
		with respect~to~${n\to \infty}$, using \eqref{eq:ldg.4}--\eqref{eq:ldg.6} and $[\bfH_{h_n}^k\!\bfv^*]^{\textup{sym}}\!\to \bfD\bfv^*$ in $L^p(\Omega)$ $(n\to \infty)$ (cf.~\cite[Lemma~A.27]{kr-phi-ldg}) we conclude, for every $\bfz\in \Vo(0)$,  that
		\begin{align*}
		\liminf_{n\to \infty}{\big\langle S_{h_n}^k\!\bfu_{h_n},
			\bfu_{h_n}-\bfz\big\rangle_{\smash{\WDGn}}}\ge\langle S\bfz_s,
		\bfu-\bfz\rangle_{\Vo(0)}\to\langle S\bfu,
		\bfu-\bfz\rangle_{\Vo(0)}\quad(s\to 0)\,.
		\end{align*}
\end{proof}

\begin{lemma}\label{lem:ldg_conv}
  Let $(h_n)_{n\in \setN}\subseteq \mathbb{R}_{>0}$ be such that
  $h_n\!\to\! 0$ $(n\!\to\!
  \infty)$~and~let~${p\in (\frac{3d}{d+2},\infty)}$.
  Then,~the~sequence
  $\smash{B_{h_n}^k\colon\! \WDGn\!\to\!  (\WDGn)^*}$,
  $n\!\in\! \setN$, defined in \eqref{eq:bB}, satisfies
  (\hyperlink{AN.1}{AN.1}) and (\hyperlink{AN.3}{AN.3}), and is
  non-conforming pseudo-monotone with respect to
  $(V_{h_n}^k\!(0))_{n\in \setN }$ and
  ${B
    \hspace{-0.22em}\colon\hspace{-0.22em}\Vo(0)\hspace{-0.22em}\to\hspace{-0.22em}\Vo(0)^*\hspace{-0.2em}}$
  defined via
  $\skp{B\bfu}{\bfz}_{\smash{\Vo(0)}}\coloneqq \hskp{[\nabla(\bfu+\bfv^*)](\bfu+\bfv^*)}{\bfz}$
  for every $\bfu,\bfz\in \Vo(0)$.
\end{lemma}

\begin{proof}
    \textit{(\hyperlink{AN.1}{AN.1}).}
    Using for fixed $n\hspace{-0.1em}\in\hspace{-0.1em} \setN$, the continuity of ${\Ghnk\hspace{-0.1em}\colon\hspace{-0.1em}\WDGn\hspace{-0.1em}\to\hspace{-0.1em} L^p(\Omega)}$ and
    $\Uppi_{h_n}^k\colon \WDGn\to \Vhk$, the compact embedding $\WDGn\hookrightarrow\hookrightarrow L^q(\Omega)$~for~${q<p^*}$, where $p^*\coloneqq \smash{\frac{dp}{d-p}}$ if $p<d$ and $p^*\coloneqq \infty$ else, which results from \cite[Lemma A.37]{kr-phi-ldg} and an element-wise application of the Rellich--Kondrachov theorem,
    from  the standard theory of Nemytski\u{\i} operators, we deduce that for every $n\in \setN$, the well-definedness,~strong~continuity, and, thus, pseudo-monotonicity
	of  $B_{h_n}^k\colon\WDGn\to (\WDGn)^*$,~${ n\in \setN }$.

	\textit{(\hyperlink{AN.3}{AN.3}).} For every $n\in \setN$, $\bfu_{h_n}\in V_{h_n}^k(0)$ and $\bfw_{h_n}\in \WDGn$, using Hölder's inequality,
	the norm equivalence \eqref{eq:eqiv0}, the discrete Sobolev
        embedding theorem (cf.~\cite[Theorem 5.2]{ern-book}) and
        $\smash{\|\Uppi_{h_n}^k\!\bfw_{h_n}\|_{\nabla,q,h_n}\!\leq\!
          c\,\|\bfw_{h_n}\|_{\nabla,q,h_n}}$, $q\! \in\! (1,\infty)$ (cf.~\hspace{-0.1em}~\cite[(A.18)]{dkrt-ldg},~we~get
    \begin{align}
       \big\langle B_{h_n}^k\!\bfu_{h_n},\bfw_{h_n}\big\rangle_{\smash{\WDGn}}&
       \leq \big\|\bfu_{h_n}+\Uppi_{h_n}^k\!\bfv^*\big\|_{2p'}^2\big\|\Dhnk\!\bfw_{h_n}\big\|_p\label{lem:ldg_conv.1}
       \\[-0.5mm]&\quad+\big\|\Uppi_{h_n}^k\!\bfw_{h_n}\big\|_{2p'}\big\|\bfu_{h_n}\!+\Uppi_{h_n}^k\!\bfv^*\big\|_{2p'}\big\| \Ghnk\bfu_{h_n}\!+\bfH_{h_n}^k\!\bfv^*\!-g\mathbf{I}_d\big\|_p\notag
       \\[-0.5mm]&\leq c\,\big(\|\bfu_{h_n}\|_{\nabla,p,\Gamma_{h_n}}+\|\bfv^*\|_{1,p}+\|g\|_p\big)^2\|\bfw_{h_n}\|_{\nabla,p,\Gamma_{h_n}}\,,\notag
    \end{align}
    where we also used that $2p' \le p^*$ for $\smash{p\in (\frac{3d}{d+2},\infty)}$. From \eqref{lem:ldg_conv.1}, in~turn,~we~obtain
    \begin{align}\label{lem:ldg_conv.2}
        \smash{\|B_{h_n}^k\!\bfu_{h_n}\|_{\smash{(\WDGn)^*}}\leq c\,\big(\|\bfu_{h_n}\|_{\nabla,p,\Gamma_{h_n}}+\|\bfv^*\|_{1,p}+\|g\|_p\big)^2\,.}
    \end{align}

    \textit{Non-conforming pseudo-monotonicity.}
    Let $\bfu_{h_{m_n}}\!\in \!\smash{V_{h_{m_n}}^k}(0)$, $n\in \setN$,~where~$(m_n)_{n\in \setN}$ $\subseteq \setN$ with $m_n\to \infty $ $(n\to\infty)$, 
	be a sequence such that
	\begin{gather}
	\sup_{n\in \setN}{\|\bfu_{h_{m_n}}\|_{\nabla,p,h_{m_n}}}<\infty\,,\qquad
	\bfu_{h_{m_n}}\weakto\bfu\quad\text{ in }L^p(\Omega)\quad(n\to\infty)\,,\label{eq:conv.3}
	\\[-1mm]
	\limsup_{n\to\infty}{\big\langle B_{h_{m_n}}^k\!\!\bfu_{h_{m_n}}
		,\bfu_{h_{m_n}}-\bfu\big\rangle_{\smash{W^{1,p}(\mathcal{T}_{h_{m_n}})}}}\leq 0\,.\label{eq:conv.4}
	\end{gather}
	We assume that $m_n\hspace{-0.1em}=\hspace{-0.1em}n$ for every
        $n\hspace{-0.1em}\in\hspace{-0.1em} \setN$. Recall that
        (\hyperlink{NC.2}{NC.2}) and \eqref{eq:conv.3} 
		 imply~${\bfu\hspace{-0.1em}\in\hspace{-0.1em} \Vo(0)}$. Resorting to \cite[Lemma A.37]{kr-phi-ldg} and the discrete Rellich--Kondrachov theorem (cf.~\cite[Theorem 5.6]{ern-book}) in \eqref{eq:conv.3},  we find that
		\begin{align}
		\begin{aligned}\label{eq:conv.5}
			\Ghnk\bfu_{h_n}&\rightharpoonup\nabla\bfu&&\quad\text{ in }L^p(\Omega)\quad(n\to \infty)\,,\\
			\bfu_{h_n}&\to\bfu&&\quad\text{ in }L^q(\Omega)\quad(n\to \infty)\quad q<p^*\,,
			\end{aligned}
		\end{align}
		Since  $\smash{\Uppi_{h_n}^k}\!\bfv^*\to \bfv^*$ in
                $L^q(\Omega)$ $(n\to \infty)$ and
                $\smash{\Uppi_{h_n}^k}\!\bfz\to \bfz$ in $L^q(\Omega)$ $(n\to
                \infty)$~for~all~${q<p^*}$
                (cf.~\cite[Corollary~A.8]{kr-phi-ldg} and
                \cite[Theorem 5.6]{ern-book}), 
                $\smash{\bfH_{h_n}^k}\!\bfv^*\to \nabla
                \bfv^*$~in~$L^p(\Omega)$ ${(n\to \infty)}$
                (cf.~\cite[Lemma A.27]{kr-phi-ldg}) and
                \eqref{eq:conv.5}, for every $\bfz\in \Vo(0)$,
                using~that ${\smash{\Dhnk}\bfz=\bfD\bfz}$~in~$L^p(\Omega)$, we conclude that
		\begin{align}
		    \lim_{n\to \infty}\,
    	    \big\langle B_{h_n}^k\!\bfu_{h_n},\bfu_{h_n}-\bfz\big\rangle_{\smash{\WDGn}}
		    & =-\tfrac{1}{2}\bighskp{(\bfu+\bfv^*)\otimes (\bfu+\bfv^*)}{     \bfD\bfu-\bfD\bfz}\label{eq:conv.6}\\&\quad+\tfrac{1}{2}\bighskp{(\bfu-\bfz)\otimes (\bfu+\bfv^*)}{     \nabla\bfu+\nabla\bfv^*-g\mathbf{I}_d} \notag
		    \\&= \langle B\bfu,\bfu-\bfz\rangle_{\Vo(0)}\,,\notag
		\end{align}
		where we used integration-by-parts for the last equality sign. 
\end{proof}

\begin{lemma}\label{lem:ldg}
    Let $(h_n)_{n\in \setN}\subseteq \mathbb{R}_{>0}$ be such that $h_n\to 0$ $(n\to \infty)$~and~let~${p\in (2,\infty)}$. Then,~the~sequence  $A_{h_n}^k\colon\! W^{1,p}(\mathcal{T}_{h_n}\hspace{-0.1em})\!\to\! (W^{1,p}(\mathcal{T}_{h_n}\hspace{-0.1em}))^*$, $n\!\in\! \setN$, 
    satisfies~(\hyperlink{AN.1}{AN.1})--(\hyperlink{AN.3}{AN.3})~and is~non-conforming pseudo-monotone with respect to $\smash{(V_{h_n}^k(0))_{n\in \setN }}$~and~$\smash{A\colon\!\Vo(0)\!\to\!\Vo(0)^*}$,  defined, for every $\bfu\in \Vo(0)$, via $A\bfu\coloneqq S\bfu+B\bfu$ in $\Vo(0)^*$.
\end{lemma}

\begin{proof}
    Note that $2\ge
        \frac{3d}{d+3}$ for $d\in \set{2,3}$. Thus, since (\hyperlink{AN.1}{AN.1}), (\hyperlink{AN.3}{AN.3})~and~\mbox{non-con}-forming pseudo-monotinicity  are stable under summation (cf.~Lemma~\ref{def:non-conform_sum}), appealing to Lemma~\ref{lem:ldg_stress} and Lemma~\ref{lem:ldg_conv}, we only have to show that  ${A_{h_n}^k\colon\hspace{-0.05em} \WDGn\hspace{-0.05em}\to\hspace{-0.05em}  (\WDGn)^*\!}$, ${n\in \setN}$, satisfies~(\hyperlink{AN.2}{AN.2}). For every $n\in \setN$ and $\bfu_{h_n}\in \smash{V_{h_n}^k(0)}$, it holds
    \begin{align*}
	\begin{aligned}
	    \langle B_{h_n}^k\!\bfu_{h_n},\bfu_{h_n}\rangle_{\WDGn}
	    &=\tfrac{1}{2}\bighskp{\bfu_{h_n}\otimes (\bfu_{h_n}+\Uppi_{h_n}^k\!\bfv^*)}{\bfH_{h_n}^k\!\bfv^*-g\mathbf{I}_d}
	    \\&\quad-\tfrac{1}{2}\bighskp{\Uppi_{h_n}^k\!\bfv^*\otimes (\bfu_{h_n}+\Uppi_{h_n}^k\!\bfv^*)}{ \Ghnk\bfu_{h_n}}\,,
	    \end{aligned}
	\end{align*}
	which, using the $\varepsilon$-Young inequality
Sldg        $\eqref{ineq:young}_{2}$ with $\psi\!=\!\vert \cdot\vert^p$ together with $p\!>\!2$, the discrete Sobolev embedding theorem \cite[Theorem 5.3]{ern-book}, the stability~properties~of~$\PiDGn$~and~\eqref{lem:ldg_stress.2},  for every $\varepsilon\hspace{-0.1em}>\hspace{-0.1em}0$, yields a constant $c_\varepsilon\hspace{-0.1em}>\hspace{-0.1em}0$ such that for~every~${\bfu_{h_n}\hspace{-0.1em}\in\hspace{-0.1em} \smash{V_{h_n}^k(0)}}$,~it~holds
	\begin{align}\label{lem:ldg.1}
	\begin{aligned}
	    &\vert \langle
            B_{h_n}^k\!\bfu_{h_n},\bfu_{h_n}\rangle_{\WDGn}\vert
            \\
            &\quad\leq \|\bfu_{h_n}\|_{2p'}\big\|\bfu_{h_n}+\Uppi_{h_n}^k\!\bfv^*\big\|_{2p'}\big\|\bfH_{h_n}^k\!\bfv^*-g\mathbf{I}_d\big\|_p\\&\quad\quad
	    +\big\|\Uppi_{h_n}^k\!\bfv^*\big\|_{2p'}\big\|\bfu_{h_n}+\Uppi_{h_n}^k\!\bfv^*\big\|_{2p'}\big\|\Ghnk\bfu_{h_n}\big\|_p\\
	    &\quad\leq c\,\|\bfu_{h_n}\|_{\nabla,p,h_n}^2\big
            (\|\bfv^*\|_{1,p} +\|g\|_p\big )+\big (\|\bfv^*\|_{1,p}
            +\|g\|_p\big )^2\|\bfu_{h_n}\|_{\nabla,p,h_n}
	    \\[-0.5mm]
	    &\quad\leq c_\varepsilon\,\left( \big
              (\|\bfv^*\|_{1,p}+\|g\|_p\big )^{2p'}+\big(\|\bfv^*\|_{1,p}+\|g\|_p\big)^{\frac{p}{p-2}}\right)+\varepsilon\,\|\bfu_{h_n}\|_{\nabla,p,h_n}^p\,.
	    \end{aligned}
	\end{align}
	Combining \eqref{lem:ldg_stress.4.1} and \eqref{lem:ldg.1} for sufficiently small $\varepsilon>0$, we conclude that~the~sequence~of operators $A_{h_n}^k\colon \WDGn\to (\WDGn)^*$, $n\in \setN$, satisfies (\hyperlink{AN.2}{AN.2}).
\end{proof}

Finally, we show that also the assumptions on the right-hand side in Theorem \ref{thm:non-conform_lim} are fulfilled in our setting.

\begin{lemma}\label{lem:rhs}
    Let $(h_n)_{n\in \setN}\subseteq \mathbb{R}_{>0}$ be such that $h_n\to 0$ $(n\to \infty)$~and~let~${p\in (1,\infty)}$.
     Then, the sequence $f_{h_n}^k\in (\WDGn)^*$, $n\in \setN$, satisfies (\hyperlink{BN.1}{BN.1}) and (\hyperlink{BN.2}{BN.2})~with~respect to $f\in \Vo(0)^*$,  defined, for every $\bfz\in \Vo(0)$, via $\langle f,\bfz\rangle_{\smash{\Vo(0)}}\coloneqq (\bfg,\bfz)+(\bfG,\nabla\bfz)$.
\end{lemma}

\begin{proof}
    Let us first verify (\hyperlink{BN.1}{BN.1}). To this end, let $\bfw_{h_n}\in \WDGn$ be arbitrary. Then, we have that 
    \begin{align}
        \big\langle f_{h_n}^k,\bfw_{h_n}\big\rangle_{\smash{W^{1,p}(\mathcal{T}_{h_n})}}\leq \|\bfg\|_{p'}\|\bfw_{h_n}\|_p+\|\bfG\|_{p'}\big\|\Ghnk\bfw_{h_n}\big\|_p\,.\label{lem:rhs.2}
    \end{align}
    Appealing to \cite[Lemma A.9]{dkrt-ldg}, it holds $\|\bfw_{h_n}\|_p\leq c\,\|\bfw_{h_n}\|_{\nabla,p,h_n}$, so that~from~\eqref{lem:rhs.2}, using the norm equivalence \eqref{eq:eqiv0}, it follows that ${\|f_{h_n}^k\|_{\smash{(\WDGn)^*}}\leq c\,\|\bfg\|_{p'}+c\,\|\bfG\|_{p'}}$.
    
    Next, let $\bfu_{h_{m_n}}\!\in\! V_{h_{m_n}}^k\!(0)$, $n\!\in \!\setN$, where $(m_n)_{n\in\setN}\!\subseteq\!\setN$~with~${m_n\!\to\! \infty}$~${(n\!\to\!\infty)}$,~be~a~se-quence satisfying \eqref{eq:conv.3} for some $\bfu\in \smash{\Vo(0)}$. Then,~using~\mbox{\cite[Lemma~A.37]{kr-phi-ldg}},~we~find~that
    \begin{align*}
		\begin{aligned}
		\boldsymbol{\mathcal{G}}_{h_{m_n}}^k\!\!\bfu_{h_{m_n}}&\weakto\nabla\bfu&&\quad\text{ in }L^p(\Omega)\quad(n\to \infty)\,,\\
			\bfu_{h_{m_n}}&\weakto\bfu&&\quad\text{ in }L^p(\Omega)\quad(n\to \infty)\,,
		\end{aligned}
	\end{align*}
	which, in view of \eqref{eq:bb}, immediately implies that (\hyperlink{BN.2}{BN.2}) holds.
\end{proof}

Now we have prepared everything to prove the well-posedness, stability, and (weak) convergence of our schemes.

\begin{proposition}[Well-posedness]\label{prop:well_posed}
    Let $(h_n)_{n\in \setN}\subseteq \mathbb{R}_{>0}$ be such that ${h_n\to 0}$ ${(n\to \infty)}$ and let ${p\in (2,\infty)}$.
    Then, for every $n\in \setN$, there exist
    $\smash{\bfv_{h_n}\!\in\! \Vhnk}$~\mbox{solving}~\eqref{eq:primal} and
    $\smash{q_{h_n}\in \Qhnkco}$ such that
    $\smash{(\bfv_{h_n},q_{h_n})^\top\in \Vhnk\times \Qhnkco}$ solves
    \eqref{eq:primal0.1}, \eqref{eq:primal0.2}.
\end{proposition}

\begin{proof}
    Appealing to Lemma \ref{lem:nonconform}, Lemma \ref{lem:ldg} and Lemma \ref{lem:rhs}, for the~existence~of $\bfu_{h_n}\in V_{h_n}^k(0)$ solving \eqref{eq:op}, we can just resort to Theorem \ref{thm:non-conform_lim} (i).
    As a consequence, $\smash{\bfv_{h_n}\hspace{-0.1em}\coloneqq \hspace{-0.1em}\bfu_{h_n}\hspace{-0.1em}+\hspace{-0.1em}\PiDGn\bfv^*\hspace{-0.2em}\in\hspace{-0.1em} \Vhnk}$ solves \eqref{eq:primal}.
    Apart~from~that,~resorting~to~\cite[Lemma~6.10]{ern-book}, we deduce the existence of a constant $\beta>0$ such that~for~every~${z_{h_n}\in \Qhnkco}$, it holds the LBB condition 
    \begin{align}
    \label{eq:lbb}
        \beta\|z_{h_n}\|_{p'}\leq \sup_{\bfz_{h_n}\in \Vhnk\setminus\{\mathbf{0}\}}{\frac{(\nabla z_{h_n},\bfz_{h_n})}{\|\bfz_{h_n}\|_{\nabla,p,h_n}}}
        =\vcentcolon\big\|\smash{\widetilde{\nabla}_{h_n}^k} z_{h_n}\big\|_{\smash{(\Vhnk)^*}}\,.
    \end{align}
    A direct consequence of the LBB condition \eqref{eq:lbb} is the
    surjectivity~and,~thus,~bijectivi-ty of the injective (discretely
    weak) gradient operator $\smash{\widetilde{\nabla}_{h_n}^k}\!
    \colon\! \Qhnkco\!\to\!(\Vhnk(0))^\circ\! $,~\mbox{defined} by
    $\langle\smash{\widetilde{\nabla}_{h_n}^k}
    z_{h_n},\bfz_{h_n}\rangle_{\Vhnk}\hspace{-0.1em}\coloneqq \hspace{-0.1em}(\nabla
    z_{h_n},\bfz_{h_n})\hspace{-0.1em}=\hspace{-0.1em}-(z_{h_n},\Divhnk\bfz_{h_n})
    $ for all~${\bfz_{h_n}\hspace{-0.1em}\in\hspace{-0.1em}
      \Vhnk}$~and~${z_{h_h}\hspace{-0.1em}\in\hspace{-0.1em} \Qhnkco}$
    (cf.~\eqref{eq:div-dg}), where $\smash{(\Vhnk(0))^\circ\hspace{-0.1em}\coloneqq \hspace{-0.1em}\{f_{h_n}\hspace{-0.1em}\in\hspace{-0.1em} (\Vhnk)^*\mid\langle f_{h_n},\bfz_{h_n}\rangle_{\Vhnk}\hspace{-0.1em}=\hspace{-0.1em}0\textup{ for all }\bfz_{h_n}\hspace{-0.1em}\in\hspace{-0.1em} \Vhnk(0)\}} $ denotes the annihilator of $\Vhnk(0)$. Thus, since ${(\textup{id}_{\smash{\Vhnk(0)}})^*(A_{h_n}^k\!\bfu_{h_n}\!-\!f_{h_n}^k)\!\in\! (\Vhnk(0))^\circ\!}$ (cf.~\!\eqref{eq:op}), we conclude the existence of $q_{h_n}\in \Qhnkco$ such that
    \begin{align}
     \label{eq:well_posed}
        \smash{\widetilde{\nabla}_{h_n}^k q_{h_n}=(\textup{id}_{\smash{\Vhnk}})^*\big(A_{h_n}^k\!\bfu_{h_n}-f_{h_n}^k\big)\quad\text{ in }(\Vhnk)^*}\,,
    \end{align}
    i.e., $\smash{(\bfv_{h_n},q_{h_n})^\top\in \Vhnk\times \Qhnkco}$
    solves \eqref{eq:primal0.1}, \eqref{eq:primal0.2}.
\end{proof}

\begin{proposition}[Stability]\label{prop:stab}
    Let $(h_n)_{n\in \setN}\subseteq \mathbb{R}_{>0}$ be such that~${h_n\to
      0}$ ${(n\to \infty)}$ and let ${p\in (2,\infty)}$.
    For every $n\in \setN$, let $\smash{(\bfv_{h_n},q_{h_n})^\top\in
      \Vhnk\times \Qhnkco}$ be a solution of
    \mbox{\eqref{eq:primal0.1}, \eqref{eq:primal0.2}}.  Then, we have that\vspace{-1mm}
    \begin{align*}
        \|\bfv_{h_n}\|_{\nabla,p,h_n}^p&\leq c\,\delta^p\,\vert \Omega\vert+c\,\|\bfg\|_{p'}^{p'}+
        c\,\|\bfG\|_{p'}^{p'}+c(\bfv^*)\,,\\
        \|q_{h_n}\|_{p'}^{p'}&\leq c\,\|\bfv_{h_n}\|_{\nabla,p,h_n}^{2p'}+
        c\,\delta^p\,\vert \Omega\vert+c\,\|\bfg\|_{p'}^{p'}+c\,\|\bfG\|_{p'}^{p'}+c(\bfv^*)\,,
    \end{align*}
    where $c(\bfv^*)\vcentcolon =\smash{\| \bfv^*\|_{1,p}^p+\big (\|\bfv^*\|_{1,p}+\|g\|_p\big
      )^{\frac{p}{p-2}}+\big (\|\bfv^*\|_{1,p}+\|g\|_p\big )^{2p'}}$,
    with a constant $c>0$ depending~only on $\alpha>0$, the characteristics of $\SSS$, and the chunkiness $\omega_0>0$.
\end{proposition}

\begin{proof}
    Since $\bfu_{h_n}\vcentcolon =\bfv_{h_n}-\PiDGn \bfv^*\in V_{h_n}^k(0)$ solves \eqref{eq:op}, 
    combining \eqref{lem:ldg_stress.4.1}, \eqref{lem:ldg.1} and
    \eqref{lem:rhs.2} in Lemma \ref{lem:rhs}, and \eqref{eq:eqiv0},
    we find that for every $\varepsilon>0$, there exists~a~constant $c_\varepsilon>0$ such that
    \begin{align}
            \|\bfu_{h_n}\|_{\nabla,p,h_n}^p-c(\bfv^*)-c\,\delta^p\,\vert \Omega\vert
            &\leq \bigskp{A_{h_n}^k\!\bfu_{h_n}}{\bfu_{h_n}}_{\smash{\WDGn}}
            =     \bigskp{f_{h_n}^k}{\bfu_{h_n}}_{\smash{\WDGn}}\notag
            \\&\leq c_\varepsilon\,\|\bfg\|_{p'}^{p'}+c_\varepsilon\,\|\bfG\|_{p'}^{p'}+c\,\varepsilon\,\|\bfu_{h_n}\|_{\nabla,p,h_n}^p\,.\label{lem:stab.1}
    \end{align}
    Consequently, due to $\|\PiDGn\bfv^*\|_{\nabla,p,h_n}\leq c\,\|\bfv^*\|_{1,p}$ (cf.~\cite[(A.19)]{dkrt-ldg}), for~sufficiently~small $\varepsilon>0$ in \eqref{lem:stab.1}, we conclude that
    \begin{align}
        \begin{aligned}\label{lem:stab.2}
            \|\bfv_{h_n}\|_{\nabla,p,h_n}^p\leq c\,\delta^p\,\vert     \Omega\vert+
            c\,\|\bfg\|_{p'}^{p'}+c\,\|\bfG\|_{p'}^{p'}+c(\bfv^*)\,.
        \end{aligned}
    \end{align}
    Apart from that, appealing to \eqref{eq:lbb} and \eqref{eq:well_posed}, we have that
    \begin{align}
            \beta \|q_{h_n}\|_{p'}&\leq \big\|\widetilde{\nabla}^k_{h_n} q_{h_n} \big\|_{\smash{(\Vhnk)^*}}
            =\big\|(\textup{id}_{\Vhnk})^*(A^k_{h_n}\! \bfu_{h_n}-f_{h_n}^k)\big\|_{\smash{(\Vhnk)^*}}\label{lem:stab.3}\\&\leq \big\|S^k_{h_n}\! \bfu_{h_n}\big\|_{\smash{(\WDGn)^*}}\!+\big\|B^k_{h_n}\! \bfu_{h_n}\big\|_{\smash{(\WDGn)^*}}\! +\big\|f^k_{h_n}\big\|_{\smash{(\WDGn)^*}}\,. \notag 
    \end{align}
    Eventually, combining \eqref{lem:ldg_stress.7.1}, \eqref{lem:ldg_conv.2} and \eqref{lem:stab.1}, we conclude that
    \begin{align*}
        \|q_{h_n}\|_{p'}^{p'}&\leq c\,\|\bfu_{h_n}\|_{\nabla,p,h_n}^p+c\,\delta^p\,\vert  \Omega\vert+ c\,\|\bfg\|_{p'}^{p'}+c\,\|\bfG\|_{p'}^{p'}+ c(\bfv^*) + c\,\|\bfu_{h_n}\|_{\nabla,p,h_n}^{2p'}\,.
    \end{align*}
\end{proof}

\begin{theorem}[Convergence]\label{thm:convergence}
  Let $(h_n)_{n\in \setN}\subseteq \mathbb{R}_{>0}$ be such
  that~${h_n\to 0}$ ${(n\to \infty)}$ and let ${p\in (2,\infty)}$.
  For every $n\in \setN$, let
  $\smash{(\bfv_{h_n},q_{h_n})^\top\in \Vhnk\times \Qhnkco}$
  be~a~solution~of \mbox{\eqref{eq:primal0.1}, \eqref{eq:primal0.2}}.
  Then, there exists a subsequence
  $(h_{m_n})_{n\in \setN}\subseteq \mathbb{R}_{>0}$, where
  $(m_n)_{n\in \setN}\subseteq \setN$ with $m_n\to \infty$
  $(n\to \infty)$, and $(\bfv,q)^\top\in V(g)\times \Qo$~such~that
    \begin{align}\label{thm:convergence.0}
        \begin{aligned}
            \boldsymbol{\mathcal{G}}_{h_{m_n}}^k\!\!\bfv_{h_{m_n}}&\weakto \nabla \bfv&&\quad\text{ in }L^p(\Omega)&&\quad(n\to \infty)\,,\\
             \boldsymbol{\mathcal{D}}_{h_{m_n}}^k\!\!\bfv_{h_{m_n}}&\to \bfD \bfv&&\quad\text{ in }L^p(\Omega)&&\quad(n\to \infty)\,,\\
            \bfv_{h_{m_n}}&\to \bfv&&\quad\text{ in }L^q(\Omega)&&\quad(n\to \infty)\quad q<p^*\,,\\
            q_{h_{m_n}}&\weakto q&&\quad\text{ in }L^{p'}(\Omega)&&\quad(n\to \infty)\,,
        \end{aligned}
    \end{align}
    where $p^*\coloneqq \!\frac{dp}{d-p}$ if $p\!<\!d$ and
    $p^*\coloneqq \!\infty$ if $p\!\ge\! d$, and $\bfv\!\in\! V(g)$
    solves~Problem~(P),~cf.~\eqref{eq:p}, while
    $\smash{(\bfv,q)^\top\in V(g)\times \Qo}$ solves Problem (Q),
    cf.~\eqref{eq:q1}, \eqref{eq:q2}.
\end{theorem}

\begin{proof}
    Appealing to Lemma \ref{lem:nonconform}, Lemma \ref{lem:ldg} and Lemma \ref{lem:rhs},  Theorem~\ref{thm:non-conform_lim}~(iii) and the discrete Rellich--Kondrachov theorem (cf.~\cite[Theorem 5.6]{ern-book}) guarantee the existence of a subsequence $(h_{m_n})_{n\in \setN}\subseteq \mathbb{R}_{>0}$, where $(m_n)_{n\in \setN}\subseteq \setN$~with~${m_n\to \infty}$~${(n\to \infty)}$, and $\bfu\in \Vo(0)$ such that
    \begin{align}\label{thm:convergence.1}
        \begin{aligned}
            \boldsymbol{\mathcal{G}}_{h_{m_n}}^k\!\bfu_{h_{m_n}}&\weakto \nabla \bfu&&\quad\text{ in }L^p(\Omega)&&\quad(n\to \infty)\,,\\[-0.5mm]
            \bfu_{h_{m_n}}&\to \bfu&&\quad\text{ in }L^q(\Omega)&&\quad(n\to \infty)\quad q<p^*\,,
        \end{aligned}
    \end{align}
    and ${A \bfu=f}$~in~$\Vo(0)^*$. Setting $\bfv_{h_{m_n}}\coloneqq \bfu_{h_{m_n}}-\PiDGn\bfv^*\in \Vhnk(0)$ for   every $n\in \setN$, and $\bfv\coloneqq \bfu+\bfv^*\in V(g)$, we see
that \eqref{thm:convergence.1} and    \cite[Lemma A.27]{kr-phi-ldg}
imply that
    \begin{align}\label{thm:convergence.2}
        \begin{aligned}
            \boldsymbol{\mathcal{G}}_{h_n}^k\!\bfv_{h_n}&\weakto \nabla \bfv&&\quad\text{ in }L^p(\Omega)&&\quad(n\to \infty)\,,\\
            \bfv_{h_n}&\to \bfv&&\quad\text{ in }L^q(\Omega)&&\quad(n\to \infty)\quad q<p^*\,,
        \end{aligned}
    \end{align}
    and that $\bfv\in V(g)$ solves Problem (P), cf.~\eqref{eq:p}.

    It remains to deal with Problem (Q) and
    \eqref{thm:convergence.1}$_2$.  For simplicity, we assume that
    $m_n=n$ for all $n \in \setN$. Since~the~sequence ${(q_{h_n})_{n\in
        \setN}\subseteq \Qo}$ is bounded (cf.~Proposition~\ref{prop:stab}), there exists a not
    relabeled~subsequence~and~${q\in \Qo}$ such that
    \begin{align}\label{thm:convergence.3}
        q_{h_n}\weakto q\quad\textup{ in }L^{p'}(\Omega)\quad(n\to \infty)\,.
    \end{align}
    Next, let us verify that $(\bfv,q)^\top\!\in\! V\times\Qo$ solves
    Problem (Q), cf.~\eqref{eq:q1}, \eqref{eq:q2}.~To~this~end, take
    $\bfz\!\in\! \Vo$ and set
    $\bfz_{h_n}\coloneqq \!\PiDGn\!\bfz_n\!\in\!\Vhnk$ for all
    $n\!\in\!\setN$.~Due~to~\eqref{eq:well_posed},~for~all~${n\!\in\! \setN}$,~it~holds
    \begin{align}\label{thm:convergence.4}
            \bighskp{q_{h_n}}{\Divhnk \bfz_{h_n}}=
          \bigskp{A_{h_n}^k\!\bfu_{h_n}-f_{h_n}^k}{\bfz_{h_n}}_{W^{1,p}(\mathcal{T}_{h_n})}\,.
    \end{align}
    The approximation properties of $\PiDGn$ (cf.~\cite[Lemma A.27]{kr-phi-ldg}),  \eqref{thm:convergence.1} and \eqref{thm:convergence.3} yield
    \begin{align}\label{thm:convergence.5}
        \begin{aligned}
            \bighskp{q_{h_n}}{\Divhnk \bfz_{h_n}}&\to  \hskp{q}{\divo \bfz}&&\quad(n\to \infty)\,,\\
            \bigskp{f_{h_n}^k-B_{h_n}^k\!\bfu_{h_n}}{\bfz_{h_n}}_{W^{1,p}(\mathcal{T}_{h_n})}&\to (\bfg,\bfz)+(\bfG,\nabla\bfz)
            -([\nabla\bfv]\bfv,\bfz)&&\quad(n\to \infty) \,.
        \end{aligned}
    \end{align}
    Hence, it is only left to determine that $\smash{\langle
      S_{h_n}^k\!\bfu_{h_n},\bfz_{h_n}\rangle_{\smash{W^{1,p}(\mathcal{T}_{h_n})}}\!\to\!
      (\SSS(\bfD\bfv),\bfD\bfz)}$~${(n\!\to\! \infty)}$. To see this,
    again, abbreviating
    $\smash{\bfL_{h_n}^{\textup{sym}}\coloneqq \Dhnk\!\bfu_{h_n}+[\bfH_{h_n}^k\!\bfv^*]^{\textup{sym}}\in
      L^p(\Omega)}$ for all $n\in \setN$ to shorten the notation, we first find, using \eqref{eq:op} and the monotonicity of $\SSS$, that
    \begin{align} 
        \begin{aligned}\label{thm:convergence.6}
            &\bighskp{\SSS(\bfL_{h_n}^{\textup{sym}})-\SSS(\Dhnk\!\PiDGn\!\bfu+[\bfH_{h_n}^k\!\bfv^*]^{\textup{sym}})}{\Dhnk\!\bfu_{h_n}-\Dhnk\!\PiDGn\!\bfu}
            \\&\quad=\bigskp{f_{h_n}^k-B_{h_n}^k\!\bfu_{h_n}}{\bfu_{h_n}-\PiDGn\!\bfu}_{W^{1,p}(\mathcal{T}_{h_n})}\\&\quad\quad+
            \bighskp{\SSS(\Dhnk\!\PiDGn\!\bfu+[\bfH_{h_n}^k\!\bfv^*]^{\textup{sym}})}{\Dhnk\!\PiDGn\!\bfu-\Dhnk\!\bfu_{h_n}}
            \\&\quad\quad-\alpha \,\big\langle\SSS_{   \sssn}(h_n^{-1}\jump{(\bfu_{h_n}+\bfh_{h_n}^k\!\bfv^*)\otimes\bfn}),\jump{(\bfu_{h_n}-\PiDGn\!\bfu  )\otimes\bfn}\big\rangle_{\Gamma_{h_n}}
            \\&\quad\leq\bigskp{f_{h_n}^k-B_{h_n}^k\!\bfu_{h_n}}{\bfu_{h_n}-\PiDGn\!\bfu}_{W^{1,p}(\mathcal{T}_{h_n})}\\&\quad\quad+
            \bighskp{\SSS(\Dhnk\!\PiDGn\!\bfu+[\bfH_{h_n}^k\!\bfv^*]^{\textup{sym}})}{\Dhnk\!\PiDGn\!\bfu-\Dhnk\!\bfu_{h_n}}
            \\&\quad\quad-\alpha \,\big\langle\SSS_{   \sssn}(h_n^{-1}\jump{(\bfu_{h_n}+\bfh_{h_n}^k\!\bfv^*)\otimes\bfn}),\jump{(\bfh_{h_n}^k\!\bfv^*-\PiDGn\!\bfu  )\otimes\bfn}\big\rangle_{\Gamma_{h_n}}\,.
            \end{aligned}
    \end{align}
    To show that the limit superior with respect to $n\to \infty$ in
    \eqref{thm:convergence.6} is non-positive, we show that all terms
    on the right-hand side of \eqref{thm:convergence.6} converge to
    zero. For the first term, we use Lemma \ref{lem:ldg}, Lemma
    \ref{lem:rhs} and
    $\smash{\|\bfu -\PiDGn\!\bfu\|_{\nabla,p,h_n}}\to 0$
    $(n\to \infty)$~(cf.~\mbox{\cite[Lemma~A.27]{kr-phi-ldg}}). For the
    second term, we use again \cite[Lemma~A.27]{kr-phi-ldg}, which
    yields
    \begin{align}
     \smash{\Dhnk\!\PiDGn\!\bfu+[\bfH_{h_n}^k\!\bfv^*]^{\textup{sym}}\to
      \bfD\bfu+\bfD\bfv^*=\bfD\bfv \quad \textrm { in } L^p(\Omega)\quad 
      (n\to\infty)\,.}\label{eq:strong}
    \end{align}
    Due to Assumption \ref{assum:extra_stress}, the extra stress tensor
    $\SSS$ is a Nemytski\u{\i} operator and, therefore, we obtain
    from \eqref{eq:strong} that 
    \begin{align}
     \smash{\SSS(\Dhnk\!\PiDGn\!\bfu+[\bfH_{h_n}^k\!\bfv^*]^{\textup{sym}})\to
      \SSS(\bfD\bfv) \quad \textrm { in } L^{p'}(\Omega)\quad 
      (n\to\infty)\,.}\label{eq:strong.1}
    \end{align}
    The convergences \eqref{thm:convergence.1}, \eqref{eq:strong}
    and \eqref{eq:strong.1} show that the second term~in~\eqref{thm:convergence.6}~converges to zero. To handle the third
    term, we use H\"older's inequality,  \eqref{lem:ldg_stress.7},
    \eqref{thm:convergence.1},~and \cite[Lemma~A.27]{kr-phi-ldg},
    to obtain  ${h_n\|h_n^{-1}\bfh_{h_n}^k\!\bfv^*\|_{\smash{p,\Gamma_{h_n}}}^{p}}+\smash{h_n\|h_n^{-1}\jump{\bfu_{h_n}\! \!\otimes \!\bfn}\|_{\smash{p,\Gamma_{h_n}}}^{p}\to\!
      0}$ ${(n\!\to\! \infty)}$. Thus, we showed, also using
    ${\Dhnk\!\bfu_{h_n}\!\!-\!\Dhnk\!\PiDGn\!\bfu\!=\!\bfL_{h_h}^{\textup{sym}} \!-\!(\Dhnk\!\PiDGn\!\bfu\!+\![\bfH_{h_n}^k\!\bfv^*]^{\textup{sym}})}$,~that 
    \begin{align*}
        \smash{\limsup_{n\to
      \infty}{\bighskp{\SSS(\bfL_{h_n}^{\textup{sym}})\!-\!\SSS(\Dhnk\!\PiDGn\!\bfu+
      [\bfH_{h_n}^k\!\bfv^*]^{\textup{sym}})}{\bfL_{h_h}^{\textup{sym}} \!-\!(\Dhnk\!\PiDGn\!\bfu+[\bfH_{h_n}^k\!\bfv^*]^{\textup{sym}})}}\leq
      0\,. }
    \end{align*}
    This,~in~turn,~implies, using \eqref{eq:strong} and
    \eqref{eq:strong.1}, that 
    \begin{align}\label{thm:convergence.8}
         \smash{\limsup_{n\to \infty}{\bighskp{\SSS(\bfL_{h_n}^{\textup{sym}})-\SSS(\bfD\bfv)}{\bfL_{h_n}^{\textup{sym}}-\bfD\bfv}}\leq 0\,.}
    \end{align}
    Appealing to Assumption \ref{assum:extra_stress} and Remark \ref{rem:phi_a}, since $p>2$, for every $n\in \mathbb{N}$,~we~have~that
    \begin{align}\label{thm:convergence.8.1}
        \|\bfL_{h_n}^{\textup{sym}}-\bfD\bfv\|_p^p\leq  c\bighskp{\SSS(\bfL_{h_n}^{\textup{sym}})-\SSS(\bfD\bfv)}{\bfL_{h_n}^{\textup{sym}}-\bfD\bfv}\,.
    \end{align}
    Combining \eqref{thm:convergence.8}, \eqref{thm:convergence.8.1} and $\smash{[\bfH_{h_n}^k\!\bfv^*]^{\textup{sym}}\!\to\! \bfD \bfv^*}$ in $L^p(\Omega)$ $(n\!\to\! \infty)$,~we~conclude~that
    \begin{align}\label{thm:convergence.8.2}
        \begin{aligned}
            \bfL_{h_n}^{\textup{sym}}&\to \bfD\bfv&&\quad\textrm{ in }L^p(\Omega)\quad(n\to \infty)\,,\\
            \boldsymbol{\mathcal{D}}_{h_n}^k\!\bfv_{h_n}&\to \bfD\bfv&&\quad\textrm{ in }L^p(\Omega)\quad(n\to \infty)\,,
        \end{aligned}
    \end{align}
    which proves \eqref{thm:convergence.0}$_2$.
    Since $\SSS$ is a Nemytski\u{\i} operator, we, in turn, deduce from
    \eqref{thm:convergence.8.2} that
   $\SSS(\bfL_{h_n}^{\textup{sym}})\to \SSS(\bfD\bfv)$ in
    $\smash{L^{p'}(\Omega)}$ $(n\to\infty)$~and,~thus,
     \begin{align}\label{thm:convergence.9}
         \bigskp{S_{h_n}^k\!\bfu_{h_n}}{\bfz_{h_n}}\to (\SSS(\bfD\bfv),\bfD\bfz)\quad(n\to \infty)\,.
     \end{align}
    As a result, by passing for $n\to \infty$ in \eqref{thm:convergence.4}, using \eqref{thm:convergence.3},  \eqref{thm:convergence.5} and \eqref{thm:convergence.9}~in~doing~so, for every $\bfz\in \Vo$, we conclude that
    \begin{align}
        \begin{aligned}\label{thm:convergence.10}
            (q,\divo \bfz)=(\SSS(\bfD\bfv),\bfD\bfz)+([\nabla\bfv] \bfv,\bfz)-(\bfg,\bfz)-(\bfG,\bfD\bfz)\,,
        \end{aligned}
    \end{align}
    i.e., $(\bfv,q)^\top\in V\times\Qo$ solves Problem (Q).
\end{proof}
\begin{remark}
  {\rm To the best of the authors knowledge, this is the first
    convergence results of a DG method for the $p$-Navier-Stokes
    problem \eqref{eq:p-navier-stokes}. A similar result for the Navier-Stokes problem ($p=2$)
    can be found in \cite{ern,ern-book}, {where a different
    discretization of the convective term is treated.} We would like to emphasize that~we~treat the fully non-homogeneous problem, which is the
    reason that we restrict ourselves to the case $p>2$. In the
    homogeneous case, i.e., $\bfv_0=\bfzero$ and $g=0$, the same methods
    would work for $p>\frac {3d}{d+2}$.  Even for the Navier-Stokes
    problem ($p=2$) most previous results (except \cite{ern,ern-book})
    prove convergence of the method through appropriate convergence
    rates under additional regularity assumptions on the solution of
    the limiting problem. 
    In~the~second part of this paper, we also show convergence rates
    for our method in the homogeneous case  and for $\bfG=\bfzero$.
  }
\end{remark}

This method of proof of course also works for the $p$-Stokes
problem, i.e., we neglect the \hspace{-0.2mm}convective \hspace{-0.2mm}term \hspace{-0.2mm}in
\hspace{-0.2mm}\eqref{eq:p-navier-stokes}, \hspace{-0.2mm}Problem \hspace{-0.2mm}(P\hspace{-0.2mm}), \hspace{-0.2mm}Problem \hspace{-0.2mm}(P$_h$\hspace{-0.2mm}), \hspace{-0.2mm}Problem \hspace{-0.2mm}(\hspace{-0.2mm}Q),
\hspace{-0.2mm}and~\hspace{-0.2mm}\mbox{Problem~\hspace{-0.2mm}(\hspace{-0.2mm}Q$_h$\hspace{-0.2mm})}. For every $p\in (1,\infty)$, the existence of solutions for the corresponding
discrete problems follows  from Lemma
\ref{lem:ldg_stress}, Lemma \ref{lem:rhs} and Theorem
\ref{thm:non-conform_lim} (i) in the same way as in the proof of
Proposition \ref{prop:well_posed}. Also the assertions of Proposition
\ref{prop:stab} hold for the $p$-Stokes problem analogously for every $p
\in (1,\infty)$, but now with the constant $c(\bfv^*)\coloneqq \|\bfv^*\|_{1,p}^p$.
Concerning the convergence of the discrete solutions, we have the
following result:\enlargethispage{5mm}
\begin{theorem}[Convergence]\label{thm:convergence1}
  Let $(h_n)_{n\in \setN}\subseteq \mathbb{R}_{>0}$ be such
  that~${h_n\to 0}$ ${(n\to \infty)}$ and let ${p\in (1,\infty)}$.
  For every $n\in \setN$, let
  $\smash{(\bfv_{h_n},q_{h_n})^\top\in \Vhnk\times \Qhnkco}$
  be~a~solution~of \mbox{\eqref{eq:primal0.1}, \eqref{eq:primal0.2}}
  without the terms coming from the convective term.
  Then, there exists a subsequence
  $(h_{m_n})_{n\in \setN}\subseteq \mathbb{R}_{>0}$, where
  $(m_n)_{n\in \setN}\subseteq \setN$ with $m_n\to \infty$
  $(n\to \infty)$, and $(\bfv,q)^\top\in V(g)\times \Qo$~such~that\vspace{-1mm}
    \begin{align}\label{thm:convergence.00}
        \begin{aligned}
            \boldsymbol{\mathcal{G}}_{h_{m_n}}^k\!\!\bfv_{h_{m_n}}&\weakto \nabla \bfv&&\quad\text{ in }L^p(\Omega)&&\quad(n\to \infty)\,,\\[-0.5mm]
            \bfv_{h_{m_n}}&\to \bfv&&\quad\text{ in }L^q(\Omega)&&\quad(n\to \infty)\quad q<p^*\,,\\[-0.5mm]
            q_{h_{m_n}}&\weakto q&&\quad\text{ in }L^{p'}(\Omega)&&\quad(n\to \infty)\,,
        \end{aligned}
    \end{align}
    where $p^*\coloneqq \frac{dp}{d-p}$ if $p<d$ and
    $p^*\coloneqq \infty$ if $p\ge d$, and $\bfv\in V(g)$
    solves~Problem~(P) without convective term,~cf.~\eqref{eq:p}, while
    $\smash{(\bfv,q)^\top\in V(g)\times \Qo}$ solves Problem (Q)
    without convective term,
    cf.~\eqref{eq:q1}, \eqref{eq:q2}.
\end{theorem}
\begin{proof}
  We proceed exactly as in the proof of Theorem \ref{thm:convergence}
  until \eqref{thm:convergence.8}, noting that, due to the absence of
  the convective term, all arguments work for  $p \in
  (1,\infty)$. For $p>2$, we can finish the proof as before, 
  showing $\boldsymbol{\mathcal{D}}_{h_n}^k\!\bfv_{h_n}\to \bfD\bfv$
  in $L^p(\Omega)$ $(n\to \infty)$. For $p\le 2$, we proceed slightly differently.
  Appealing to Assumption~\ref{assum:extra_stress} and Remark~\ref{rem:phi_a},~it~holds
  {  \begin{align*}
      {  \int_\Omega
      \phi_{\abs{\bfD\bfv}}(\bfL_{h_n}^{\textup{sym}}-\bfD\bfv)\,
      \textup{d}x 
      \leq
      c\bighskp{\SSS(\bfL_{h_n}^{\textup{sym}})-\SSS(\bfD\bfv)}{\bfL_{h_n}^{\textup{sym}}-\bfD\bfv}\to
      0 \quad (n \to \infty)\,.}
    \end{align*}
  }%
  This implies, as in \cite[Lemma 18]{r-cetraro}, that
    $\bfL_{h_n}^{\textup{sym}} \to \bfD\bfv$ a.e.~in $\Omega$, which 
    together with \eqref{eq:ass_S}, \eqref{thm:convergence.1} and the
    classical result in \cite[Lemma 1.19]{GGZ} yields $\SSS(\bfL_{h_n}^{\textup{sym}})\weakto \SSS(\bfD\bfv)$ in
    $\smash{L^{p'}(\Omega)}$ $(n\to\infty)$, so that we conclude as in the
    proof of Theorem \ref{thm:convergence}.
\end{proof}
\section{Numerical experiments}\label{sec:experiments}

In this section, we apply the LDG scheme \eqref{eq:DG} (or \eqref{eq:primal0.1}--\eqref{eq:primal0.2} and  \eqref{eq:primal}) to solve
numerically the system~\eqref{eq:p-navier-stokes}~with $\SSS\colon\mathbb{R}^{d\times d}\to\mathbb{R}^{d\times d}$, for every $\bfA\in\mathbb{R}^{d\times d}$ defined by  ${\SSS(\bfA) \coloneqq (\delta+\vert \bfA^{\textup{sym}}\vert)^{p-2}\bfA^{\textup{sym}}}$,
where $\delta\coloneqq 1\textrm{e}{-}4$ and $ p>2$.
We approximate the discrete solution~${\bfv_h\in V^k_h}$ of the
non-linear problem~\eqref{eq:DG}~deploying the Newton solver from
\mbox{\textsf{PETSc}} (version 3.17.3), cf.~\cite{LW10}, with an
absolute tolerance~of $\tau_{abs}\!=\! 1\textrm{e}{-}8$ and a relative
tolerance of $\tau_{rel}\!=\!1\textrm{e}{-}10$. The linear system
emerging in each Newton step is solved using a sparse direct solver
from \textsf{MUMPS} (version~5.5.0),~cf.~\cite{mumps}. For the
numerical flux \eqref{def:flux-S}, we choose the fixed parameter
$\alpha=2.5$. This choice is in accordance with the choice in
\mbox{\cite[Table~1]{dkrt-ldg}}. In the implementation, the uniqueness
of the pressure is enforced via a zero mean condition. 

All experiments were carried out using the finite element software package~\mbox{\textsf{FEniCS}} (version 2019.1.0), cf.~\cite{LW10}. 

For \hspace{-0.1mm}our \hspace{-0.1mm}numerical \hspace{-0.1mm}experiments, \hspace{-0.1mm}we \hspace{-0.1mm}choose \hspace{-0.1mm}$\Omega\!=\! (-1,1)^2$ \hspace{-0.1mm}and \hspace{-0.1mm}linear~\hspace{-0.1mm}elements,~\hspace{-0.1mm}i.e.,~\hspace{-0.1mm}${k\!=\! 1}$. We \hspace{-0.1mm}choose \hspace{-0.1mm}$\smash{\bfg\!\in\! L^{p'}\!(\Omega)}$, \hspace{-0.1mm}$\smash{\bfG\!\in\! L^{p'}\!(\Omega)}$ \hspace{-0.1mm}and $\hspace{-0.1mm}\smash{\bfv_0\!\in\! W^{1-\frac{1}{p},p}(\partial\Omega)}$~\hspace{-0.1mm}such~\hspace{-0.1mm}that~\hspace{-0.1mm}${\bfv\!\in\! W^{1,p}(\Omega)}$~\hspace{-0.1mm}and~\hspace{-0.1mm}${q\!\in\! \Qo}$, for every $x\coloneqq (x_1,x_2)^\top\in \Omega$ defined by
\begin{align}
	\smash{\bfv(x)\coloneqq \vert x\vert^\beta(x_2,-x_1)^\top, \qquad q(x)\vcentcolon =\vert x\vert^{\gamma}-\langle\,\vert \!\cdot\!\vert^{\gamma}\,\rangle_\Omega}
\end{align}
are a solution of \eqref{eq:p-navier-stokes}\footnote{Note that $g=0$ in this example.\vspace{-7mm}}. For
$\rho\in \{0.01,0.05,0.1\}$, we choose
$\beta=\smash{\frac{2(\rho-1)}{p}}$,~which~yields that
$\vert \nabla^{\rho}\bfF(\bfD\bfv)\vert \in L^2(\Omega)$, and
$\gamma=\rho-\smash{\frac{2}{p'}}$, which
yields~that~$\vert \nabla^{\rho} q\vert\in L^{p'}(\Omega)$, i.e., we
expect a convergence rate of order
$\rho\smash{\frac{p'}{2}}\in (0,\infty)$. Here, $\bfF(\bfA)\coloneqq (\delta+\vert
\bfA^{\textup{sym}}\vert)^{\smash{\frac{p-2}2}}\bfA^{\textup{sym}}$ for all $\bfA\in \mathbb{R}^{d\times d}$.  This setup is motivated by Part II of
the paper (cf.~\cite{kr-pnse-ldg-2}).

We \hspace{-0.1mm}construct \hspace{-0.1mm}a \hspace{-0.1mm}initial \hspace{-0.1mm}triangulation \hspace{-0.1mm}$\mathcal
T_{h_0}$, \hspace{-0.1mm}where \hspace{-0.1mm}$h_0\hspace{-0.2em}=\hspace{-0.2em}\smash{\frac{1}{\sqrt{2}}}$, \hspace{-0.1mm}by \hspace{-0.1mm}subdividing~\hspace{-0.1mm}a~\hspace{-0.1mm}\mbox{rectangular} cartesian grid~into regular triangles with different orientations.  Finer triangulations~$\mathcal T_{h_i}$, $i=1,\dots,5$, where $h_{i+1}=\frac{h_i}{2}$ for all $i=1,\dots,5$, are 
obtained by
regular subdivision of the previous grid: Each \mbox{triangle} is subdivided
into four equal triangles by connecting the midpoints of the edges, i.e., applying the red-refinement rule, cf. \cite[Definition~4.8~(i)]{Ba16}.

Then, for the resulting series of triangulations $\mathcal T_{h_i}$, $i\!=\!1,\dots,5$, we apply~the~above Newton scheme to compute the corresponding numerical solutions $(\bfv_i, \bfL_i, \bfS_i,\bfK_i)^\top\coloneqq \smash{(\bfv_{h_i},\bfL_{h_i}, \bfS_{h_i},\bfK_{h_i})^\top\in V_{h_i}^k\times X_{h_i}^k\times X_{h_i}^k\times X_{h_i}^k}$, $i=1,\dots,5$, 
and the error quantities
\begin{align*}
	\left.\begin{aligned}
		e_{\bfL,i}&\coloneqq \|\bfF(\bfL_i^{\textup{sym}})-\bfF(\bfL^{\textup{sym}})\|_2\,,\\[0.5mm]
		e_{\bfS,i}&\coloneqq \|\bfF^*(\bfS_i)-\bfF^*(\bfS)\|_2\,,\\[0.5mm]
		e_{\jump{},i}&\coloneqq m_{\phi_{\smash{\avg{\abs{\Pi_{h_i}^0\!\bfL_i^{\textup{sym}}}}}},h_i}(\bfv_i-\bfv)^{\smash{\frac{1}{2}}}\,,\\[0.5mm]
		e_{q,i}&\coloneqq (\|q_i-q\|_{p'}^{\smash{p'}})^{\smash{\frac{1}{2}}}\,,
	\end{aligned}\quad\right\}\quad i=1,\dots,5\,,
\end{align*}
where $\bfF^*(\bfA)\coloneqq (\delta^{p-1}+\vert
\bfA^{\textup{sym}}\vert)^{\smash{\frac{p'-2}{2}}}\bfA^{\textup{sym}}$ for all $\bfA\in \mathbb{R}^{d\times d}$. 
As estimation~of~the~conver-gence rates,  the experimental order of convergence~(EOC)
\begin{align*}
	\texttt{EOC}_i(e_i)\coloneqq \frac{\log(e_i/e_{i-1})}{\log(h_i/h_{i-1})}\,, \quad i=1,\dots,5\,,
\end{align*}
where for any $i\!=\! 1,\dots,5$, we denote by $e_i$
either 
$e_{\bfL,i}$ , $e_{\bfS,i}$,
$e_{\jump{},i}$,~or~$e_{q,i}$,~\mbox{resp.},~is~recorded.  For
 \hspace{-0.1mm}$p\!\in\! \{2.2, 2.5, 3, 3.5\}$, \hspace{-0.1mm}$\rho\!\in\! \{0.01,0.05,0.1\}$ \hspace{-0.1mm}and \hspace{-0.1mm}a \hspace{-0.1mm}series \hspace{-0.1mm}of
\hspace{-0.1mm}triangulations~\hspace{-0.1mm}$\mathcal{T}_{h_i}$,~\hspace{-0.1mm}${i\!=\! 1,\dots,5}$, obtained by
regular, global refinement as described above, the EOC is computed and
presented in Table~\ref{tab1}, Table~\ref{tab2}, Table~\ref{tab3} and
Table~\ref{tab4}, resp. In each case, we observe the expected a
convergence~ratio~of about
$\texttt{EOC}_i(e_i)\approx \rho\smash{\frac{p'}{2}}$, $i=1,\dots, 5$,
indicating the convergence of LDG scheme \eqref{eq:DG} (or
\eqref{eq:primal0.1}, \eqref{eq:primal0.2} and \eqref{eq:primal}). We
have~also~carried~out numerical experiments for $\rho =0$. Even in
this case, the error for all quantities decreases and we see a positive but
decreasing very small convergence rate.~This~indicates that the
algorithm converges even in this case, confirming the assertions~of~Theorem~\ref{thm:convergence}.
\enlargethispage{4mm}
\begin{table}[H]
    \setlength\tabcolsep{2pt}
	\centering
	\begin{tabular}{c |c|c|c|c|c|c|c|c|c|c|c|c|} \cmidrule(){1-13}
	\multicolumn{1}{|c||}{\cellcolor{lightgray}$\rho$}	& \multicolumn{4}{c||}{\cellcolor{lightgray}0.01} & \multicolumn{4}{c||}{\cellcolor{lightgray}0.05}    & \multicolumn{4}{c|}{\cellcolor{lightgray}0.1}\\ 
		\hline 
		   
		   \multicolumn{1}{|c||}{\cellcolor{lightgray}\diagbox[height=1.1\line]{\vspace*{-0.6mm}\footnotesize$i$}{\\[-5.5mm]\footnotesize $p$}}
		   & \cellcolor{lightgray}2.2 & \cellcolor{lightgray}2.5  & \cellcolor{lightgray}3.0 & \multicolumn{1}{c||}{\cellcolor{lightgray}3.5} & \multicolumn{1}{c|}{\cellcolor{lightgray}2.2}   & \cellcolor{lightgray}2.5  & \cellcolor{lightgray}3.0 &   \multicolumn{1}{c||}{\cellcolor{lightgray}3.5}     & \multicolumn{1}{c|}{\cellcolor{lightgray}2.2}    & \cellcolor{lightgray}2.5  & \cellcolor{lightgray}3.0 & \cellcolor{lightgray}3.5 \\ \hline\hline
			\multicolumn{1}{|c||}{\cellcolor{lightgray}$1$}                		& 0.0039 & 0.0032 & 0.0045 & \multicolumn{1}{c||}{0.0063} & \multicolumn{1}{c|}{0.041} & 0.037 & 0.035 & \multicolumn{1}{c||}{0.035} & \multicolumn{1}{c|}{0.087} & 0.079 & 0.073 & 0.070 \\ \hline
			\multicolumn{1}{|c||}{\cellcolor{lightgray}$2$}                  	& 0.0106 & 0.0093 & 0.0079 & \multicolumn{1}{c||}{0.0072} & \multicolumn{1}{c|}{0.048} & 0.043 & 0.038 & \multicolumn{1}{c||}{0.035} & \multicolumn{1}{c|}{0.094} & 0.086 & 0.076 & 0.071 \\ \hline
			\multicolumn{1}{|c||}{\cellcolor{lightgray}$3$}                      & 0.0101 & 0.0088 & 0.0077 & \multicolumn{1}{c||}{0.0071} & \multicolumn{1}{c|}{0.048} & 0.043 & 0.038 & \multicolumn{1}{c||}{0.035} & \multicolumn{1}{c|}{0.094} & 0.085 & 0.076 & 0.070 \\ \hline
			\multicolumn{1}{|c||}{\cellcolor{lightgray}$4$}               		& 0.0097 & 0.0086 & 0.0076 & \multicolumn{1}{c||}{0.0071} & \multicolumn{1}{c|}{0.047} & 0.043 & 0.038 & \multicolumn{1}{c||}{0.035} & \multicolumn{1}{c|}{0.094} & 0.085 & 0.076 & 0.070 \\ \hline
			\multicolumn{1}{|c||}{\cellcolor{lightgray}$5$}               		& 0.0095 & 0.0086 & 0.0076 & \multicolumn{1}{c||}{0.0071} & \multicolumn{1}{c|}{0.047} & 0.043 & 0.038 & \multicolumn{1}{c||}{0.035} & \multicolumn{1}{c|}{0.094} & 0.085 & 0.076 & 0.070 \\ \hline\hline
			\multicolumn{1}{|c||}{\cellcolor{lightgray}$\rho\frac{p'}{2}\!$}     & 0.0092 & 0.0083 & 0.0075 & \multicolumn{1}{c||}{0.0070} & \multicolumn{1}{c|}{0.046} & 0.042 & 0.038 & \multicolumn{1}{c||}{0.035} & \multicolumn{1}{c|}{0.092} & 0.083 & 0.075 & 0.070 \\ \hline
	\end{tabular}\vspace{-2mm}
	\caption{\small Experimental order of convergence: $\texttt{EOC}_i(e_{\bfL,i})$,~${i=1,\dots,5}$.}
\label{tab1}
\end{table}\vspace{-8mm}

\begin{table}[H]
    \setlength\tabcolsep{2pt}
	\centering
	\begin{tabular}{c |c|c|c|c|c|c|c|c|c|c|c|c|} \cmidrule(){1-13}
	\multicolumn{1}{|c||}{\cellcolor{lightgray}$\rho$}	& \multicolumn{4}{c||}{\cellcolor{lightgray}0.01} & \multicolumn{4}{c||}{\cellcolor{lightgray}0.05}    & \multicolumn{4}{c|}{\cellcolor{lightgray}0.1}\\ 
		\hline 
		   
		   \multicolumn{1}{|c||}{\cellcolor{lightgray}\diagbox[height=1.1\line]{\vspace*{-0.6mm}\footnotesize$i$}{\\[-5.5mm]\footnotesize $p$}}
		   & \cellcolor{lightgray}2.2 & \cellcolor{lightgray}2.5  & \cellcolor{lightgray}3.0 & \multicolumn{1}{c||}{\cellcolor{lightgray}3.5} & \multicolumn{1}{c|}{\cellcolor{lightgray}2.2}   & \cellcolor{lightgray}2.5  & \cellcolor{lightgray}3.0 &   \multicolumn{1}{c||}{\cellcolor{lightgray}3.5}     & \multicolumn{1}{c|}{\cellcolor{lightgray}2.2}    & \cellcolor{lightgray}2.5  & \cellcolor{lightgray}3.0 & \cellcolor{lightgray}3.5 \\ \hline\hline
			\multicolumn{1}{|c||}{\cellcolor{lightgray}$1$}                		& 0.0259 & 0.0176 & 0.0088 & \multicolumn{1}{c||}{0.0033} & \multicolumn{1}{c|}{0.063} & 0.051 & 0.039 & \multicolumn{1}{c||}{0.031} & \multicolumn{1}{c|}{0.109} & 0.094 & 0.077 & 0.066 \\ \hline
			\multicolumn{1}{|c||}{\cellcolor{lightgray}$2$}                  	& 0.0149 & 0.0103 & 0.0081 & \multicolumn{1}{c||}{0.0070} & \multicolumn{1}{c|}{0.052} & 0.044 & 0.038 & \multicolumn{1}{c||}{0.035} & \multicolumn{1}{c|}{0.098} & 0.086 & 0.076 & 0.070 \\ \hline
			\multicolumn{1}{|c||}{\cellcolor{lightgray}$3$}                      & 0.0117 & 0.0091 & 0.0077 & \multicolumn{1}{c||}{0.0070} & \multicolumn{1}{c|}{0.049} & 0.043 & 0.038 & \multicolumn{1}{c||}{0.035} & \multicolumn{1}{c|}{0.096} & 0.085 & 0.075 & 0.070 \\ \hline
			\multicolumn{1}{|c||}{\cellcolor{lightgray}$4$}               		& 0.0103 & 0.0087 & 0.0076 & \multicolumn{1}{c||}{0.0070} & \multicolumn{1}{c|}{0.048} & 0.042 & 0.038 & \multicolumn{1}{c||}{0.035} & \multicolumn{1}{c|}{0.094} & 0.085 & 0.075 & 0.070 \\ \hline
			\multicolumn{1}{|c||}{\cellcolor{lightgray}$5$}               		& 0.0098 & 0.0085 & 0.0075 & \multicolumn{1}{c||}{0.0070} & \multicolumn{1}{c|}{0.047} & 0.042 & 0.038 & \multicolumn{1}{c||}{0.035} & \multicolumn{1}{c|}{0.094} & 0.085 & 0.075 & 0.070 \\ \hline\hline
			\multicolumn{1}{|c||}{\cellcolor{lightgray}$\rho\frac{p'}{2}\!$}       & 0.0092 & 0.0083 & 0.0075 & \multicolumn{1}{c||}{0.0070} & \multicolumn{1}{c|}{0.046} & 0.042 & 0.038 & \multicolumn{1}{c||}{0.035} & \multicolumn{1}{c|}{0.092} & 0.083 & 0.075 & 0.070 \\ \hline
	\end{tabular}\vspace{-2mm}
	\caption{\small Experimental order of convergence: $\texttt{EOC}_i(e_{\jump{},i})$,~${i=1,\dots,5}$.} 
	\label{tab2}
\end{table}\vspace{-8mm}

 \begin{table}[H]
     \setlength\tabcolsep{2pt}
 	\centering
 	\begin{tabular}{c |c|c|c|c|c|c|c|c|c|c|c|c|} \cmidrule(){1-13}
	\multicolumn{1}{|c||}{\cellcolor{lightgray}$\rho$}	& \multicolumn{4}{c||}{\cellcolor{lightgray}0.01} & \multicolumn{4}{c||}{\cellcolor{lightgray}0.05}    & \multicolumn{4}{c|}{\cellcolor{lightgray}0.1}\\ 
		\hline 
		   
		   \multicolumn{1}{|c||}{\cellcolor{lightgray}\diagbox[height=1.1\line]{\vspace*{-0.6mm}\footnotesize$i$}{\\[-5.5mm]\footnotesize $p$}}
		   & \cellcolor{lightgray}2.2 & \cellcolor{lightgray}2.5  & \cellcolor{lightgray}3.0 & \multicolumn{1}{c||}{\cellcolor{lightgray}3.5} & \multicolumn{1}{c|}{\cellcolor{lightgray}2.2}   & \cellcolor{lightgray}2.5  & \cellcolor{lightgray}3.0 &   \multicolumn{1}{c||}{\cellcolor{lightgray}3.5}     & \multicolumn{1}{c|}{\cellcolor{lightgray}2.2}    & \cellcolor{lightgray}2.5  & \cellcolor{lightgray}3.0 & \cellcolor{lightgray}3.5 \\ \hline\hline
 			\multicolumn{1}{|c||}{\cellcolor{lightgray}$1$}                		& 0.0041 & 0.0176 & 0.0055 & \multicolumn{1}{c||}{0.0079} & \multicolumn{1}{c|}{0.041} & 0.037 & 0.036 & \multicolumn{1}{c||}{0.036} & \multicolumn{1}{c|}{0.087} & 0.080 & 0.074 & 0.071 \\ \hline
 			\multicolumn{1}{|c||}{\cellcolor{lightgray}$2$}                  	& 0.0106 & 0.0103 & 0.0079 & \multicolumn{1}{c||}{0.0073} & \multicolumn{1}{c|}{0.048} & 0.043 & 0.038 & \multicolumn{1}{c||}{0.035} & \multicolumn{1}{c|}{0.095} & 0.086 & 0.076 & 0.071 \\ \hline
 			\multicolumn{1}{|c||}{\cellcolor{lightgray}$3$}                      & 0.0101 & 0.0091 & 0.0077 & \multicolumn{1}{c||}{0.0071} & \multicolumn{1}{c|}{0.048} & 0.043 & 0.038 & \multicolumn{1}{c||}{0.035} & \multicolumn{1}{c|}{0.094} & 0.085 & 0.076 & 0.070 \\ \hline
 			\multicolumn{1}{|c||}{\cellcolor{lightgray}$4$}               		& 0.0097 & 0.0087 & 0.0076 & \multicolumn{1}{c||}{0.0071} & \multicolumn{1}{c|}{0.047} & 0.043 & 0.038 & \multicolumn{1}{c||}{0.035} & \multicolumn{1}{c|}{0.094} & 0.085 & 0.076 & 0.070 \\ \hline
 			\multicolumn{1}{|c||}{\cellcolor{lightgray}$5$}               		& 0.0095 & 0.0085 & 0.0076 & \multicolumn{1}{c||}{0.0071} & \multicolumn{1}{c|}{0.047} & 0.043 & 0.038 & \multicolumn{1}{c||}{0.035} & \multicolumn{1}{c|}{0.094} & 0.085 & 0.076 & 0.070 \\ \hline\hline
 			\multicolumn{1}{|c||}{\cellcolor{lightgray}$\rho\frac{p'}{2}\!$}       & 0.0092 & 0.0083 & 0.0075 & \multicolumn{1}{c||}{0.0070} & \multicolumn{1}{c|}{0.046} & 0.042 & 0.038 & \multicolumn{1}{c||}{0.035} & \multicolumn{1}{c|}{0.092} & 0.083 & 0.075 & 0.070 \\ \hline
 	\end{tabular}\vspace{-2mm}
 	\caption{\small Experimental order of convergence: $\texttt{EOC}_i(e_{\bfS,i})$,~${i=1,\dots,5}$.} 
 	\label{tab3}
 \end{table}\vspace{-1cm}

\begin{table}[H]
    \setlength\tabcolsep{2pt}
	\centering
	\begin{tabular}{c |c|c|c|c|c|c|c|c|c|c|c|c|} \cmidrule(){1-13}
	\multicolumn{1}{|c||}{\cellcolor{lightgray}$\rho$}	& \multicolumn{4}{c||}{\cellcolor{lightgray}0.01} & \multicolumn{4}{c||}{\cellcolor{lightgray}0.05}    & \multicolumn{4}{c|}{\cellcolor{lightgray}0.1}\\ 
		\hline 
		   
		   \multicolumn{1}{|c||}{\cellcolor{lightgray}\diagbox[height=1.1\line]{\vspace*{-0.6mm}\footnotesize$i$}{\\[-5.5mm]\footnotesize $p$}}
		   & \cellcolor{lightgray}2.2 & \cellcolor{lightgray}2.5  & \cellcolor{lightgray}3.0 & \multicolumn{1}{c||}{\cellcolor{lightgray}3.5} & \multicolumn{1}{c|}{\cellcolor{lightgray}2.2}   & \cellcolor{lightgray}2.5  & \cellcolor{lightgray}3.0 &   \multicolumn{1}{c||}{\cellcolor{lightgray}3.5}     & \multicolumn{1}{c|}{\cellcolor{lightgray}2.2}    & \cellcolor{lightgray}2.5  & \cellcolor{lightgray}3.0 & \cellcolor{lightgray}3.5 \\ \hline\hline
			\multicolumn{1}{|c||}{\cellcolor{lightgray}$1$}                	& 0.0292 & 0.0290 & 0.0344 & \multicolumn{1}{c||}{0.0438} & \multicolumn{1}{c|}{0.066} & 0.062 & 0.064 & \multicolumn{1}{c||}{0.071} & \multicolumn{1}{c|}{0.111} & 0.104 & 0.101 & 0.106 \\ \hline     
			\multicolumn{1}{|c||}{\cellcolor{lightgray}$2$}                  & 0.0159 & 0.0121 & 0.0090 & \multicolumn{1}{c||}{0.0074} & \multicolumn{1}{c|}{0.052} & 0.045 & 0.039 & \multicolumn{1}{c||}{0.036} & \multicolumn{1}{c|}{0.098} & 0.087 & 0.077 & 0.071 \\ \hline   
			\multicolumn{1}{|c||}{\cellcolor{lightgray}$3$}                  & 0.0118 & 0.0098 & 0.0084 & \multicolumn{1}{c||}{0.0079} & \multicolumn{1}{c|}{0.048} & 0.043 & 0.039 & \multicolumn{1}{c||}{0.036} & \multicolumn{1}{c|}{0.094} & 0.085 & 0.076 & 0.071 \\ \hline   
			\multicolumn{1}{|c||}{\cellcolor{lightgray}$4$}               	& 0.0102 & 0.0088 & 0.0079 & \multicolumn{1}{c||}{0.0077} & \multicolumn{1}{c|}{0.047} & 0.042 & 0.038 & \multicolumn{1}{c||}{0.036} & \multicolumn{1}{c|}{0.093} & 0.084 & 0.076 & 0.071 \\ \hline  
			\multicolumn{1}{|c||}{\cellcolor{lightgray}$5$}               	& 0.0096 & 0.0085 & 0.0076 & \multicolumn{1}{c||}{0.0074} & \multicolumn{1}{c|}{0.046} & 0.042 & 0.038 & \multicolumn{1}{c||}{0.035} & \multicolumn{1}{c|}{0.092} & 0.084 & 0.075 & 0.070 \\ \hline\hline
			\multicolumn{1}{|c||}{\cellcolor{lightgray}$\rho\frac{p'}{2}\!$}   & 0.0092 & 0.0083 & 0.0075 & \multicolumn{1}{c||}{0.0070} & \multicolumn{1}{c|}{0.046} & 0.042 & 0.038 & \multicolumn{1}{c||}{0.035} & \multicolumn{1}{c|}{0.092} & 0.083 & 0.075 & 0.070 \\ \hline
	\end{tabular}\vspace{-2mm}
	\caption{\small Experimental order of convergence: $\texttt{EOC}_i(e_{q,i})$,~${i=1,\dots,5}$.}
	\label{tab4}
\end{table}\vspace{-1cm}

\def\cprime{$'$} \def\cprime{$'$} \def\cprime{$'$}


\begin{thebibliography}{10}

\bibitem{mumps}
{\sc P.~R. Amestoy, I.~S. Duff, J.~Koster, and J.-Y. L'Excellent}, {\em A fully
  asynchronous multifrontal solver using distributed dynamic scheduling}, SIAM
  Journal on Matrix Analysis and Applications, 23 (2001), pp.~15--41.

\bibitem{arnold-brezzi}
{\sc D.~N. Arnold, F.~Brezzi, B.~Cockburn, and L.~D. Marini}, {\em Unified
  analysis of discontinuous {G}alerkin methods for elliptic problems}, SIAM J.
  Numer. Anal., 39 (2001/02), pp.~1749--1779.

\bibitem{BL1993a}
{\sc J.~W. Barrett and W.~B. Liu}, {\em Finite element approximation of the
  {$p$}-{L}aplacian}, Math. Comp., 61 (1993), pp.~523--537.

\bibitem{BL1994b}
{\sc J.~W. Barrett and W.~B. Liu}, {\em Quasi-norm error bounds for the finite
  element approximation of a non-{N}ewtonian flow}, Numer. Math., 68 (1994),
  pp.~437--456.

\bibitem{Ba16}
{\sc S.~Bartels}, {\em Numerical approximation of partial differential
  equations}, vol.~64 of Texts in Applied Mathematics, Springer, 2016,
    \url{https://doi.org/10.1007/978-3-319-32354-1}.

\bibitem{Bar21}
{\sc S.~Bartels}, {\em Nonconforming discretizations of convex minimization
  problems and precise relations to mixed methods}, Comput. Math. Appl., 93
  (2021), pp.~214--229, 
  \url{https://doi.org/10.1016/j.camwa.2021.04.014}.

\bibitem{hugo-boundary}
{\sc H.~Beir{\~a}o~da Veiga}, {\em Navier-{S}tokes equations with
  shear-thickening viscosity. {R}egularity up to the boundary}, J. Math. Fluid
  Mech., 11 (2009), pp.~233--257.

\bibitem{hugo-petr-rose}
{\sc H.~Beir{\~a}o~da Veiga, P.~Kaplick\'y, and M.~R{\r u}{\v z}i{\v c}ka},
  {\em Boundary regularity of shear--thickening flows}, J. Math. Fluid Mech.,
  13 (2011), pp.~387--404.

\bibitem{bdr-phi-stokes}
{\sc L.~Belenki, L.~C. Berselli, L.~Diening, and M.~R{\r u}{\v z}i{\v c}ka},
  {\em On the {F}inite {E}lement approximation of $p$-{S}tokes systems}, SIAM
  J. Numer. Anal., 50 (2012), pp.~373--397.

\bibitem{alex-rose-nonconform}
{\sc L.~C. Berselli, A.~Kaltenbach, and M.~R{\r u}\v{z}i\v{c}ka}, {\em Analysis
  of fully discrete, quasi non-conforming approximations of evolution equations
  and applications}, Math. Models Methods Appl. Sci., 31 (2021),
  pp.~2297--2343, \href{http://dx.doi.org/10.1142/S0218202521500494}
  {doi:10.1142/S0218202521500494}.

\bibitem{br-multiple-approx}
{\sc L.~C. Berselli and M.~R\r{u}\v{z}i\v{c}ka}, {\em Natural second-order
  regularity for parabolic systems with operators having $(p,\delta)$-structure
  and depending only on the symmetric gradient}, Calc. Var. PDEs,  (2022),
  p.~Paper No.~137, \href{http://dx.doi.org/10.1007/s00526-022-02247-y}
  {doi:10.1007/s00526-022-02247-y}.

\bibitem{bird}
{\sc R.~Bird, R.~Armstrong, and O.~Hassager}, {\em Dynamic of Polymer Liquids},
  John Wiley, 1987.
\newblock 2nd edition.

\bibitem{BCPH20}
{\sc M.~Botti, D.~Castanon~Quiroz, D.~A. Di~Pietro, and A.~Harnist}, {\em A
  hybrid high-order method for creeping flows of non-{N}ewtonian fluids}, ESAIM
  Math. Model. Numer. Anal., 55 (2021), pp.~2045--2073,
  \href{http://dx.doi.org/10.1051/m2an/2021051} {doi:10.1051/m2an/2021051}.

\bibitem{BPG19}
{\sc M.~Botti, D.~Di~Pietro, and A.~Guglielmana}, {\em A low-order
  nonconforming method for linear elasticity on general meshes}, Computer
  Methods in Applied Mechanics and Engineering, 354 (2019), pp.~96--118,
  \href{http://dx.doi.org/10.1016/j.cma.2019.05.031}
  {doi:10.1016/j.cma.2019.05.031}.

\bibitem{BS08}
{\sc S.~Brenner and L.~Scott}, {\em The mathematical theory of finite element
  methods}, vol.~15 of Texts in Applied Mathematics, Springer, New York,
  third~ed., 2008, \href{http://dx.doi.org/10.1007/978-0-387-75934-0}
  {doi:10.1007/978-0-387-75934-0}.

\bibitem{BF1991}
{\sc F.~Brezzi and M.~Fortin}, {\em Mixed and hybrid finite element methods},
  vol.~15 of Springer Series in Computational Mathematics, Springer-Verlag, New
  York, 1991.

\bibitem{BufOrt09}
{\sc A.~Buffa and C.~Ortner}, {\em Compact embeddings of broken {S}obolev
  spaces and applications}, IMA J. Numer. Anal., 29 (2009), pp.~827--855.

\bibitem{ern-p-laplace}
{\sc E.~Burman and A.~Ern}, {\em Discontinuous {G}alerkin approximation with
  discrete variational principle for the nonlinear {L}aplacian}, C. R. Math.
  Acad. Sci. Paris, 346 (2008), pp.~1013--1016.

\bibitem{bustinza-fluid}
{\sc R.~Bustinza and G.~Gatica}, {\em A mixed local discontinuous {G}alerkin
  method for a class of nonlinear problems in fluid mechanics}, J. Comput.
  Phys., 207 (2005), pp.~427--456.

\bibitem{CCQ17}
{\sc A.~Cesmelioglu, B.~Cockburn, and W.~Qiu}, {\em Analysis of a hybridizable
  discontinuous {G}alerkin method for the steady-state incompressible
  {N}avier-{S}tokes equations}, Math. Comp., 86 (2017), pp.~1643--1670,
  \href{http://dx.doi.org/10.1090/mcom/3195} {doi:10.1090/mcom/3195},
  \url{https://doi.org/10.1090/mcom/3195}.

\bibitem{CKS05}
{\sc B.~Cockburn, G.~Kanschat, and D.~Sch\"{o}tzau}, {\em A locally
  conservative {LDG} method for the incompressible {N}avier-{S}tokes
  equations}, Math. Comp., 74 (2005), pp.~1067--1095,
  \href{http://dx.doi.org/10.1090/S0025-5718-04-01718-1}
  {doi:10.1090/S0025-5718-04-01718-1},
  \url{https://doi.org/10.1090/S0025-5718-04-01718-1}.

\bibitem{CKS09}
{\sc B.~Cockburn, G.~Kanschat, and D.~Sch\"{o}tzau}, {\em An equal-order {DG}
  method for the incompressible {N}avier-{S}tokes equations}, J. Sci. Comput.,
  40 (2009), pp.~188--210, \href{http://dx.doi.org/10.1007/s10915-008-9261-1}
  {doi:10.1007/s10915-008-9261-1}.

\bibitem{CS16}
{\sc B.~Cockburn and J.~Shen}, {\em A hybridizable discontinuous {G}alerkin
  method for the $p$-laplacian}, SIAM Journal on Scientific Computing, 38
  (2016), pp.~A545--A566, \href{http://dx.doi.org/10.1137/15M1008014}
  {doi:10.1137/15M1008014}.

\bibitem{CHSW13}
{\sc S.~Congreve, P.~Houston, E.~S\"{u}li, and T.~P. Wihler}, {\em
  Discontinuous {G}alerkin finite element approximation of quasilinear elliptic
  boundary value problems {II}: strongly monotone quasi-{N}ewtonian flows}, IMA
  J. Numer. Anal., 33 (2013), pp.~1386--1415,
  \href{http://dx.doi.org/10.1093/imanum/drs046} {doi:10.1093/imanum/drs046}.

\bibitem{ern}
{\sc D.~Di~Pietro and A.~Ern}, {\em Discrete functional analysis tools for
  discontinuous {G}alerkin methods with application to the incompressible
  {N}avier-{S}tokes equations}, Math. Comp., 79 (2010), pp.~1303--1330.

\bibitem{ern-book}
{\sc D.~Di~Pietro and A.~Ern}, {\em Mathematical aspects of discontinuous
  Galerkin methods}, vol.~69 of Math\'ematiques \& Applications, Springer,
  Berlin, 2012.

\bibitem{die-ett}
{\sc L.~Diening and F.~Ettwein}, {\em Fractional estimates for
  non-differentiable elliptic systems with general growth}, Forum Math., 20
  (2008), pp.~523--556.

\bibitem{DK08}
{\sc L.~Diening and C.~Kreuzer}, {\em Linear convergence of an adaptive finite
  element method for the {$p$}-{L}aplacian equation}, SIAM J. Numer. Anal., 46
  (2008), pp.~614--638, \href{http://dx.doi.org/10.1137/070681508}
  {doi:10.1137/070681508}.

\bibitem{die-sueli-2013}
{\sc L.~Diening, C.~Kreuzer, and E.~S{\"u}li}, {\em Finite element
  approximation of steady flows of incompressible fluids with implicit
  power-law-like rheology}, SIAM J. Numer. Anal., 51 (2013), pp.~984--1015,
  \href{http://dx.doi.org/10.1137/120873133} {doi:10.1137/120873133}.

\bibitem{dkrt-ldg}
{\sc L.~Diening, D.~Kr\"oner, M.~R{\r u}{\v z}i{\v c}ka, and I.~Toulopoulos},
  {\em A {L}ocal {D}iscontinuous {G}alerkin approximation for systems with
  $p$-structure}, IMA J. Num. Anal., 34 (2014), pp.~1447--1488,
  \href{http://dx.doi.org/doi: 10.1093/imanum/drt040} {doi:
  10.1093/imanum/drt040}.

\bibitem{john}
{\sc L.~Diening, M.~R{\r u}{\v z}i{\v c}ka, and K.~Schumacher}, {\em A
  decomposition technique for {J}ohn domains}, Ann. Acad. Sci. Fenn. Math., 35
  (2010), pp.~87--114, \href{http://dx.doi.org/10.5186/aasfm.2010.3506}
  {doi:10.5186/aasfm.2010.3506}.

\bibitem{ern-theory}
{\sc A.~Ern and J.~L. Guermond}, {\em Theory and practice of finite elements},
  vol.~159 of Applied Mathematical Sciences, Springer-Verlag, New York, 2004,
  \href{http://dx.doi.org/10.1007/978-1-4757-4355-5}
  {doi:10.1007/978-1-4757-4355-5}.
  
\bibitem{EG21}
{\sc A.~Ern and J.~L. Guermond}, {\em Finite Elements I: Approximation and
  Interpolation}, no.~1 in Texts in Applied Mathematics, Springer International
  Publishing, 2021, \href{http://dx.doi.org/10.1007/978-3-030-56341-7}
  {doi:10.1007/978-3-030-56341-7}.

\bibitem{FGS22}
{\sc P.~Farrell, P.~A. Gazca~Orozco, and E.~S\"{u}li}, {\em Finite element
  approximation and preconditioning for anisothermal flow of
  implicitly-constituted non-{N}ewtonian fluids}, Math. Comp., 91 (2022),
  pp.~659--697, \href{http://dx.doi.org/10.1090/mcom/3703}
  {doi:10.1090/mcom/3703}.

\bibitem{GGZ}
{\sc H.~Gajewski, K.~Gr\"oger, and K.~Zacharias}, {\em Nichtlineare
  Operator\-glei\-chungen und Operatordifferentialgleichungen},
  Akademie-Verlag, Berlin, 1974.

\bibitem{GS15}
{\sc G.~N. Gatica and F.~A. Sequeira}, {\em Analysis of an augmented {HDG}
  method for a class of quasi-{N}ewtonian {S}tokes flows}, J. Sci. Comput., 65
  (2015), pp.~1270--1308, \href{http://dx.doi.org/10.1007/s10915-015-0008-5}
  {doi:10.1007/s10915-015-0008-5}.

\bibitem{Hi13a}
{\sc A.~Hirn}, {\em Approximation of the {$p$}-{S}tokes equations with
  equal-order finite elements}, J. Math. Fluid Mech., 15 (2013), pp.~65--88.

\bibitem{kr-phi-ldg}
{\sc A.~Kaltenbach and M.~R{\r{u}}{\v{z}}i{\v{c}}ka}, {\em Convergence analysis
  of a {L}ocal {D}iscontinuous {G}alerkin approximation for nonlinear systems
  with balanced {O}rlicz-structure}.
\newblock submitted, 2022, \url{http://arxiv.org/abs/2204.09984}.

\bibitem{kr-unsteady-dg}
{\sc A.~Kaltenbach and M.~R{\r{u}}{\v{z}}i{\v{c}}ka}, {\em Analysis of a
  fully-discrete, non-conforming approximation of evolution equations and
  applications}, Math. Models Methods Appl. Sci.,  (2023),
  \url{https://arxiv.org/abs/2212.06648}.
\newblock accepted.

\bibitem{kr-pnse-ldg-2}
{\sc A.~Kaltenbach and M.~R{\r{u}}{\v{z}}i{\v{c}}ka}, {\em A {L}ocal
  {D}iscontinuous {G}alerkin approximation for the $p$-{N}avier-{S}tokes
  system, {P}art~{II}: {C}onvergence rates for the velocity}, SIAM J. Num.
  Anal.,  (2023), \url{https://arxiv.org/abs/2208.04107}.
\newblock accepted.

\bibitem{KPS18}
{\sc S.~Ko, P.~Pust\v{e}jovsk\'{a}, and E.~S\"{u}li}, {\em Finite element
  approximation of an incompressible chemically reacting non-{N}ewtonian
  fluid}, ESAIM Math. Model. Numer. Anal., 52 (2018), pp.~509--541,
  \href{http://dx.doi.org/10.1051/m2an/2017043} {doi:10.1051/m2an/2017043}.

\bibitem{KS19}
{\sc S.~Ko and E.~S\"{u}li}, {\em Finite element approximation of steady flows
  of generalized {N}ewtonian fluids with concentration-dependent power-law
  index}, Math. Comp., 88 (2019), pp.~1061--1090,
  \href{http://dx.doi.org/10.1090/mcom/3379} {doi:10.1090/mcom/3379}.

\bibitem{KS16}
{\sc C.~Kreuzer and E.~S\"{u}li}, {\em Adaptive finite element approximation of
  steady flows of incompressible fluids with implicit power-law-like rheology},
  ESAIM Math. Model. Numer. Anal., 50 (2016), pp.~1333--1369,
  \href{http://dx.doi.org/10.1051/m2an/2015085} {doi:10.1051/m2an/2015085}.

\bibitem{lady-bo}
{\sc O.~Lady{\v z}henskaya}, {\em The Mathematical Theory of Viscous
  Incompressible Flow}, Gordon and Breach, New York, 1969.
\newblock 2nd edition.

\bibitem{LW10}
{\sc A.~Logg and G.~N. Wells}, {\em Dolfin: Automated finite element
  computing}, ACM Transactions on Mathematical Software, 37 (2010), pp.~1--28,
  \href{http://dx.doi.org/10.1145/1731022.1731030}
  {doi:10.1145/1731022.1731030}.

\bibitem{mnrr}
{\sc J.~M{\'a}lek, J.~Ne{\v{c}}as, M.~Rokyta, and M.~R{\r u}{\v{z}}i{\v{c}}ka},
  {\em Weak and measure-valued solutions to evolutionary {PDE}s}, vol.~13 of
  Applied Mathematics and Mathematical Computation, Chapman \& Hall, London,
  1996.

\bibitem{ma-ra-model}
{\sc J.~M\'alek and K.~R. Rajagopal}, {\em Mathematical issues concerning the
  {N}avier-{S}tokes equations and some of its generalizations}, in Evolutionary
  equations. {V}ol. {II}, {H}andb. {D}iffer. {E}qu., Elsevier/North-Holland,
  Amsterdam, 2005, pp.~371--459.

\bibitem{mrr}
{\sc J.~M{\'a}lek, K.~R. Rajagopal, and M.~R{\r u}{\v{z}}i{\v{c}}ka}, {\em
  Existence and regularity of solutions and the stability of the rest state for
  fluids with shear dependent viscosity}, Math. Models Methods Appl. Sci., 5
  (1995), pp.~789--812.

\bibitem{sip}
{\sc T.~Malkmus, M.~R{\r u}{\v z}i{\v c}ka, S.~Eckstein, and I.~Toulopoulos},
  {\em Generalizations of {SIP} methods to systems with {$p$}-structure}, IMA
  J. Numer. Anal., 38 (2018), pp.~1420--1451,
  \href{http://dx.doi.org/10.1093/imanum/drx040} {doi:10.1093/imanum/drx040}.

\bibitem{r-mol-inhomo}
{\sc E.~Molitor and M.~R{\r u}{\v z}i{\v c}ka}, {\em On inhomogeneous
  $p$--{N}avier--{S}tokes systems}, in Recent Advances in PDEs and
  Applications, V.~Radulescu, A.~Sequeira, and V.~Solonnikov, eds., vol.~666 of
  Contemp. Math., AMS Proceedings, 2016, pp.~317--340.

\bibitem{NPC11}
{\sc N.~C. Nguyen, J.~Peraire, and B.~Cockburn}, {\em An implicit high-order
  hybridizable discontinuous {G}alerkin method for the incompressible
  {N}avier-{S}tokes equations}, J. Comput. Phys., 230 (2011), pp.~1147--1170,
  \href{http://dx.doi.org/10.1016/j.jcp.2010.10.032}
  {doi:10.1016/j.jcp.2010.10.032}.

\bibitem{QS19}
{\sc W.~Qiu and K.~Shi}, {\em Analysis on an {HDG} method for the
  {$p$}-{L}aplacian equations}, J. Sci. Comput., 80 (2019), pp.~1019--1032,
  \href{http://dx.doi.org/10.1007/s10915-019-00967-6}
  {doi:10.1007/s10915-019-00967-6}.

\bibitem{ren-rao}
{\sc M.~M. Rao and Z.~D. Ren}, {\em Theory of {O}rlicz spaces}, vol.~146 of
  Monographs and Textbooks in Pure and Applied Mathematics, Marcel Dekker Inc.,
  New York, 1991.

\bibitem{r-cetraro}
{\sc M.~R{\r {u}}{\v{z}}i{\v{c}}ka}, {\em Analysis of generalized {N}ewtonian
  fluids}, in Topics in mathematical fluid mechanics, vol.~2073 of Lecture
  Notes in Math., Springer, Heidelberg, 2013, pp.~199--238.

\bibitem{dr-nafsa}
{\sc M.~R{\r u}{\v z}i{\v c}ka and L.~Diening}, {\em Non--{N}ewtonian fluids
  and function spaces}, in Nonlinear Analysis, Function Spaces and
  Applications, Proceedings of {NAFSA} 2006 {P}rague, vol.~8, 2007,
  pp.~{95--144}.

\bibitem{San1993}
{\sc D.~Sandri}, {\em Sur l'approximation num\'erique des \'ecoulements
  quasi-newtoniens dont la viscosit\'e suit la loi puissance ou la loi de
  {C}arreau}, RAIRO Mod\'el. Math. Anal. Num\'er., 27 (1993), pp.~131--155.

\bibitem{zei-IIB}
{\sc E.~Zeidler}, {\em Nonlinear functional analysis and its applications.
  {II}/{B}}, Springer, New York, 1990.
\newblock Nonlinear monotone operators.

\end{thebibliography}
\end{document}